\documentclass[smallextended,referee,envcountsect, article]{svjour3}
\usepackage [latin1]{inputenc}
\usepackage{amsmath,amssymb}
\usepackage{marvosym}
\usepackage{algorithm}
\usepackage{algpseudocode}
\usepackage{amsmath}
\usepackage{graphicx}
\usepackage{subcaption}
\usepackage{lipsum}
\usepackage{epstopdf}
\usepackage[numbers,sort&compress]{natbib}
\usepackage[colorlinks,linkcolor=blue,urlcolor=blue,citecolor=blue]{hyperref}
\usepackage{xcolor}
\smartqed
\usepackage{graphicx}
\journalname{}

\usepackage{fancyhdr}
\usepackage{ntheorem}
\theoremheaderfont{\bfseries\upshape}
\theorembodyfont{\upshape}
\renewtheorem{remark}{\it Remark}[section]
\renewtheorem{example}{Example}[section]

\begin{document}

\title{ Convergence analysis of  a Tikhonov regularized inertial dynamical system  and algorithm   for  convex optimization problems}

\author{Xiangkai Sun$^{1}$\and Guoxiang Tian$^{1}$\and Huan Zhang$^{2}$}

\institute{\\Xiangkai Sun (Corresponding Author)  \at{\small sunxk@ctbu.edu.cn } \\
          \\Guoxiang Tian\at {\small Tiangx2025@163.com} \\ \\Huan Zhang
 \at{\small zhanghwxy@163.com}\\ \\
               $^{1}$Chongqing Key Laboratory of Statistical Intelligent Computing and Monitoring, School of Mathematics and Statistics,
 Chongqing Technology and Business University,
Chongqing 400067, China\\
$^{2}$College of Mathematics and Statistics, Chongqing University, Chongqing 401331, China.
}

\date{Received: date / Accepted: date}
%The correct dates will be entered by the editor.

\maketitle

\begin{abstract}
This paper deals with  a Tikhonov regularized second-order inertial dynamical system that incorporates time scaling, asymptotically vanishing damping and Hessian-driven damping for solving convex optimization problems. Under appropriate setting of the parameters, we first obtain fast convergence results of the function value along the  trajectory generated by the dynamical system.
Then, we show that  the trajectory generated by the dynamical system converges weakly to a minimizer  of the convex optimization problem.
We also demonstrate that, by properly tuning these parameters, both   fast convergence rates of the function value and   strong convergence
of the trajectory towards the minimum norm solution of the convex optimization problem can be achieved simultaneously. Furthermore, we  study
convergence properties of an  inertial proximal gradient algorithm  obtained by the temporal  discretization of the  dynamical system. Finally, we present numerical experiments to illustrate the obtained results.
\end{abstract}

\section{Introduction}

Let $\mathcal{X}$ be a real Hilbert space and  let $g:\mathcal{X}\rightarrow\mathbb{R}$ be a  differentiable convex function. Consider the following unconstrained convex optimization problem
\begin{eqnarray}\label{g}
\min_{x\in\mathcal{X}}{g(x)}.
\end{eqnarray}
In recent years,  the second-order dynamics approach, as one of the powerful frameworks for finding the optimal solution set of the problem (\ref{g}),
has attracted the interest of many researchers. In order to solve the problem (\ref{g}),  Su et al. \cite{SBC2016} propose the  second-order  dynamical system with asymptotically vanishing damping
\begin{eqnarray}\label{AVD}
\ddot{x}(t)+\frac{\alpha}{t}\dot{x}(t)+\nabla g(x(t))=0,
\end{eqnarray}
where $\alpha\ge3$  and $\frac{\alpha}{t}$ is the vanishing damping coefficient. They show that  the asymptotic convergence rate of the function value along the trajectory generated by  the system (\ref{AVD}) is $\mathcal{O}\left ( \frac{1}{t^2}\right )$. In particular, in the case $\alpha=3$,
the system (\ref{AVD}) can be regarded as a continuous version of the fast gradient method of Nesterov, see \cite{y1983} for more details.
To improve the convergence rates of the   system (\ref{AVD}),  Attouch et al. \cite{acr2019} propose the  second-order  dynamical system with asymptotically vanishing damping and  time scaling
$
\ddot{x}(t)+\frac{\alpha}{t}\dot{x}(t)+\delta(t)\nabla g(x(t))=0,
$
where $\alpha >0$  and $\delta :\left[ t_0,+\infty \right) \rightarrow \left[ 0,+\infty \right)$ is the time scaling function.
They show that the fast convergence rate of the  function value  along the trajectory is  $\mathcal{O}\left ( \frac{1}{\delta(t)t^2}\right )$  as time $t$ approaches infinity. In the last few years,
many successful treatments of the convergence properties of the second-order system  with general damping and time scaling coefficients for the problem (\ref{g}) have been  investigated from several different perspectives. We refer the readers to \cite{mp18a,botm22,al21amo,jems,mor,fcm25,hecoap} for more details.

Recently, in order to establish  strong convergence results of the  trajectory generated by the dynamical system (\ref{AVD}),     Attouch et al. \cite{ACr2018} introduce the   second-order dynamical system  involving a Tikhonov regularization term
\begin{eqnarray}\label{in1.2}
\ddot{x} ( t )+\frac{\alpha}{t}\dot{x}( t )+\nabla g( x ( t ) )+\epsilon ( t )x ( t )=0,
\end{eqnarray}
where $\alpha >0$ and $\epsilon :\left[ t_0,+\infty \right) \rightarrow \left[ 0,+\infty \right)$   is the Tikhonov regularization function. They first obtain the fast convergence  rate of the function  value along the trajectory generated by the system (\ref{in1.2}). Then, they establish  strong convergence results of the  trajectory generated by the system (\ref{in1.2}) towards the minimal norm solution of the  problem (\ref{g}).
Subsequently,   to improve the convergence rate of the function value along the trajectory generated by the  system  (\ref{in1.2}), Xu and Wen \cite{xuwen} introduce the  Tikhonov regularized second-order dynamical system  with time scaling
$
\ddot{x} ( t )+\frac{\alpha}{t}\dot{x}( t )+\delta(t)\Big(\nabla g( x ( t ) )+\epsilon ( t )x ( t )\Big)=0.
$
They show that the convergence rate of the   function value along the trajectories, as it converges towards the optimal value of the problem (\ref{g}), can be faster than $o\left ( \frac{1}{t^2}\right )$. They also  prove that the trajectory converges strongly to  the minimal norm solution of the  problem (\ref{g}).

More recently, based on the non-autonomous version of the Polyak  heavy ball method, Attouch et al. \cite{attjde22}  introduce the   Tikhonov regularized second-order dynamical system
\begin{eqnarray}\label{sys12}
\ddot{x}( t ) + \alpha \sqrt{\epsilon ( t )}\dot{x}( t ) +\nabla g( x( t ) ) +\epsilon ( t )x( t ) =0.
\end{eqnarray}
 Note that the viscous damping coefficient  is proportional to the square root of the Tikhonov regularization parameter    in   the system (\ref{sys12}). By properly tuning these parameters, they achieve a fast convergence rate of the function value and strong convergence of the trajectory towards the minimum norm solution, which, in turn, leads to an improvement in the results obtained in  \cite{ACr2018}. In particular, when $\epsilon ( t )=\frac{1}{t^p}$ and $ 0 <p<2 $, Attouch et al. \cite{attjde22} also study the convergence properties of the system
\begin{eqnarray}\label{012ac}
\ddot{x}( t ) +\frac{\alpha}{t^{\frac{p}{2}}}\dot{x}( t ) +\nabla g( x( t ) ) +\frac{1}{t^p}x( t ) =0.
\end{eqnarray}
As an extension of the dynamical system  (\ref{012ac}), L\'aszl\'o \cite{qp2023} gives  convergence analysis of the   system
\begin{eqnarray}\label{DS1}
\ddot{x}( t ) +\frac{\alpha}{t^q}\dot{x}( t ) +\nabla g( x( t ) ) +\frac{a}{t^p}x( t ) =0,
\end{eqnarray}
where $\alpha >0$, $b>0$, $0<q<1$ and $0<p<q+1$.  L\'aszl\'o \cite{2025LC} also gives convergence analysis of an inertial proximal algorithm obtained by the temporal  discretization of the   system (\ref{DS1}).
Furthermore, in order to improve the convergence rate of the function value along the trajectory generated by the system  (\ref{012ac}), Bagy et al. \cite{jnva} investigate   the   Tikhonov regularized second-order dynamical system with time scaling
$
\ddot{x}( t ) +\frac{\alpha}{t^{\frac{p}{2}}}\dot{x}( t ) +\delta(t)\nabla g( x( t ) ) +\frac{b}{t^p}x( t ) =0.
$
 For more   convergence results   for Tikhonov regularized second-order   dynamical systems with    asymptotically vanishing damping and time scaling terms, we refer the readers to  \cite{jems,TRIGS1,renjota,jota24,24amop10} for more details.

It appears that Tikhonov regularized second-order   dynamical systems, which  are  driven by the Hessian of the objective functions,  can induce  stabilization of the trajectories. Due to these situations,
 many researchers have focused their research on the development of Tikhonov regularized second-order dynamical systems with a Hessian-driven damping term. More  precisely,  to induce stabilization of the trajectory generated by the system (\ref{in1.2}), Bo\c{t} et al.  \cite{bot21mp}  propose the Tikhonov regularized second-order dynamical system with Hessian-driven damping
\begin{eqnarray}\label{botsys}
\ddot{x}( t ) +\frac{\alpha}{t}\dot{x}( t ) +\beta \nabla ^2g( x( t ) ) \dot{x}( t ) +\nabla g( x( t ) ) +\epsilon ( t ) x( t ) =0,
\end{eqnarray}
where $\alpha \ge 3$ and $\beta >0$. They show that  the convergence rate of the function value along the trajectory generated by the system (\ref{botsys}) is $\mathcal{O}\left( \frac{1}{t^2} \right)$ for $\alpha =3
$ and  $o\left( \frac{1}{t^2} \right)$ for $\alpha >3$, respectively. They  also  obtain strong convergence of the trajectory  to the   minimal norm solution.
To further accelerate the fast convergence rate of the system (\ref{botsys}), by using a time scaling technique,
Bagy et al. \cite{A.C. Bagy} consider the    dynamical  system with  Hessian-driven damping and  time scaling
\begin{eqnarray}\label{DS2}
\ddot{x}( t ) +\alpha \dot{x}( t ) +\beta \nabla ^2g( x( t ) ) \dot{x}( t ) +\delta( t ) \nabla g( x( t ) ) +ax( t ) =0,
\end{eqnarray}
where $a>0$  is a positive constant.   They also obtain  fast convergence rate of the function value and strong convergence of the  trajectory to the minimum norm solution. Bagy et al. \cite{2024AC} give  an inertial proximal algorithm
obtained by the temporal  discretization of  continuous dynamical system (\ref{DS2}).

On the other hand,  in order to  induce  stabilization of the trajectory generated by the system (\ref{sys12}),
Attouch et al. \cite{amo23att}   propose the   dynamical system  with Hessian-driven damping
 \begin{equation*}
 \ddot{x}(t)+\alpha \sqrt{\epsilon(t)}\dot{x}(t)+\beta\nabla^2g(x(t))\dot{x}(t)+\nabla g(x(t))+\epsilon(t)x(t)=0.
 \end{equation*}
They demonstrate that, simultaneously, the function value  exhibits a fast convergence rate, the gradient  exhibits a fast convergence rate towards zero, and the trajectory shows strong convergence towards the minimum norm solution of the problem (\ref{g}).
In particular, when $\epsilon ( t )=\frac{1}{t^p}$ and $ 0 <p<2 $, they also study the convergence properties of the system
$
\ddot{x}( t ) +\frac{\alpha}{t^{\frac{p}{2}}}\dot{x}( t )+\beta\nabla^2g(x(t))\dot{x}(t)+\nabla g(x(t)) +\frac{1}{t^p}x( t ) =0.
$
More results on the convergence rates of Tikhonov regularized second-order dynamical system with Hessian-driven damping for solving  the  problem (\ref{g}) can be found in \cite{jems,siam21,coap25s,coap24k,zhong}.

In this paper, we propose a modified version of the  second-order  dynamical  system (\ref{DS1}) by adding both Hessian-driven damping and time scaling terms
\begin{eqnarray}\label{DS}
\ddot{x}( t) +\frac{\alpha}{t^q}\dot{x}( t ) +\beta \nabla ^2g( x( t ) ) \dot{x}( t ) +\delta ( t ) \nabla g( x( t ) ) +\frac{a}{t^p}x( t ) =0,
\end{eqnarray}
where  $t\geq t_0>0$,  $\alpha, \beta, a, p>0 $ and $0<q<1$,  $\frac{\alpha}{t^q}$  is the viscous damping,  $\delta:[t_0,+\infty)\rightarrow(0,+\infty)$ is the time scaling function, and $\frac{a}{t^p}x( t )$   is the Tikhonov regularization term.
We  also assume that  $\delta(t) $ is a non-decreasing differentiable function and $\lim _{t\rightarrow +\infty}\delta ( t ) =+\infty$.
The purpose of this paper is to obtain some new convergence properties of the solution trajectory generated by the system (\ref{DS}).
Our contributions can be more specifically stated as follows:
\begin{itemize}

\item[{\rm (i)}] Under mild assumptions on the parameters, we  show that  the trajectory $x(t)$ generated by the  system (\ref{DS}) converges weakly to a minimizer of the problem (\ref{g}). We also establish a
simultaneous result concerning both the convergence results of the objective function value  in (\ref{g}) and the strong convergence of the trajectory generated by the   system (\ref{DS}).
\item[{\rm (ii)}] We propose a new inertial proximal gradient type   algorithm obtained by the  temporal  discretization of  the  system (\ref{DS}). We also study the convergence rates  of the proposed  algorithm.
\item[\rm(iii)] Through    numerical experiments, we demonstrate that the system (\ref{DS}) performs better than the  systems (\ref{DS1}) and (\ref{DS2}) in energy error and iteration error. We also show that the algorithm from temporal discretization of system (\ref{DS}) outperforms those from systems    (\ref{DS1}) and (\ref{DS2}).
\end{itemize}

 The paper is organized as follows. In Section 2, we  first recall  some preliminary results which will be used in the sequel and we obtain the existence and uniqueness of the solution trajectory generated by the  system (\ref{DS}).
In Section 3, we  present  fast  convergence rate of the objective function value and  weak convergence  of the trajectory generated by the system (\ref{DS}).
We also obtain   strong convergence of the trajectory  generated by the  system (\ref{DS}).
In Section 4, we  study   convergence properties of the  algorithm obtained by the temporal  discretization of    system (\ref{DS}).
In Section 5, we give numerical experiments to illustrate the obtained results.

\section{Preliminaries}
Throughout this paper, let $\mathcal{X}$ be a real Hilbert space equipped with the inner
product $ \langle\cdot,\cdot\rangle$ and the norm $\|\cdot\|$. The notations $L^1[t_0, +\infty)$ and $L_{loc}^1[t_0, +\infty)$  denotes the family of integrable and locally integrable functions, respectively.
Let $\psi:\mathcal{X}\rightarrow \mathbb{R}$ be  a   real-valued  function.  We say that $\nabla \psi$ is  Lipschitz continuous on $\mathcal{X}$ iff there exists  $L_\psi>0$ such that
$
\lVert \nabla \psi( x ) -\nabla \psi( y ) \rVert \le L_\psi \|x-y \| ,$  $\forall x, y\in \mathcal{X}.
$
For any  given $\sigma >0$, we say that $\psi$ is a $\sigma$-strongly convex function iff $\psi( \cdot ) -\frac{\sigma}{2}\lVert \cdot \rVert ^2$ is a convex function. Obviously,
\begin{eqnarray}\label{L}
\psi( y ) \ge \psi( x ) +\left< \nabla \psi( x ) ,y-x \right> +\frac{\sigma}{2}\lVert y-x \rVert ^2, \forall x, y\in \mathcal{X}.
\end{eqnarray}

The following important properties will be used in the sequel.
\begin{lemma}\label{L1}\textup{ \cite[Lemma 14]{qp2023}}
Let $F:[ t_0,+\infty ) \rightarrow \mathbb{R}$ be locally absolutely continuous and bounded from
below. Suppose that there exists $G\in L^1[ t_0,+\infty )$ such that
$
\frac{d}{dt}F( t ) \le G( t ),$ for almost every $ t \geq t_0 .
$
Then,  $\lim _{t\rightarrow +\infty}F( t )\in \mathbb{R}$.
\end{lemma}

\begin{lemma}\label{L2}
Let $\alpha >0$, $0<q<1$ and $u:[t_0,+\infty ) \rightarrow \mathbb{R}$ be a continuously differentiable function. Suppose that there exists a constant $L\in \mathbb{R}$ such that $\lim _{t\rightarrow +\infty}( \alpha -qt^{q-1} ) u( t ) +t^q\dot{u}( t )=L.$
Then, $\lim _{t\rightarrow +\infty}u( t )=\frac{L}{\alpha}$.
 \end{lemma}
 \begin{proof}
 Let $( \alpha -qt^{q-1} ) u( t ) +t^q\dot{u}( t ) =y( t )$. Clearly,
$
\left( \frac{\alpha}{t^q}-\frac{q}{t} \right) u( t )+\dot{u}( t ) =\frac{y( t )}{t^q}.
$
Then, it is easy to see that
$
\left( \frac{\alpha}{t^q}-\frac{q}{t} \right) t^{-q}e^{\frac{\alpha t^{1-q}}{1-q}}u\left( t \right) +t^{-q}e^{\frac{\alpha t^{1-q}}{1-q}}\dot{u}( t )=\frac{d}{dt}\left( t^{-q}e^{\frac{\alpha t^{1-q}}{1-q}}u( t ) \right)  =t^{-2q}e^{\frac{\alpha t^{1-q}}{1-q}}y( t ).
$
Let $s<t$. Integrating  the last equality from $s$ to $t$, we obtain
$
t^{-q}e^{\frac{\alpha t^{1-q}}{1-q}}u( t) -s^{-q}e^{\frac{\alpha s^{1-q}}{1-q}}u( s ) =\int_s^t{\frac{e^{\frac{\alpha \tau ^{1-q}}{1-q}}y( \tau )}{\tau ^{2q}}}d\tau.
$
It follows that
\begin{eqnarray}\label{60}
u( t ) =t^{q}e^{-\frac{\alpha t^{1-q}}{1-q}}\left( \int_s^t{\frac{e^{\frac{\alpha \tau ^{1-q}}{1-q}}y( \tau )}{\tau ^{2q}}}d\tau +s^{-q}e^{\frac{\alpha s^{1-q}}{1-q}}u( s ) \right).
\end{eqnarray}
Note that
$
\underset{t\rightarrow +\infty}{\lim}\frac{\int_s^t{\frac{e^{\frac{\alpha \tau ^{1-q}}{1-q}}y( \tau )}{\tau ^{2q}}}d\tau}{t^{-q}e^{\frac{\alpha t^{1-q}}{1-q}}}=\underset{t\rightarrow +\infty}{\lim}\frac{\frac{e^{\frac{\alpha t^{1-q}}{1-q}}y( t )}{t^{2q}}}{\left( -qt^{-q-1}+\alpha t^{-2q} \right) e^{\frac{\alpha t^{1-q}}{1-q}}}=\underset{t\rightarrow +\infty}{\lim}\frac{y( t )}{-qt^{q-1}+\alpha}=\frac{L}{\alpha}
$
and $\lim _{t\rightarrow +\infty}t^{q}e^{-\frac{\alpha t^{1-q}}{1-q}}s^{-q}e^{\frac{\alpha s^{1-q}}{1-q}}u\left( s \right) =0.$ Then, it follows from $(\ref{60})$ that
$
\underset{t\rightarrow +\infty}{\lim}u( t ) =\frac{L}{\alpha}.
$
 The proof is complete.\qed
\end{proof}

\begin{lemma}\label{L3}\textup{ \cite[Lemma 15]{qp2023}}
Let $S\subseteq \mathcal{X}$ be a nonempty set and $x:[t_ 0,+\infty) \rightarrow \mathcal{X}$. Assume
that the following statements are satisfied.
\begin{itemize}
\item[{\rm (i)}]For every $z\in S$,  $\lim _{t\rightarrow +\infty}\lVert x( t ) -z \rVert$ exists;
\item[{\rm (ii)}] Every weak sequential limit point of $x( t )$, as $t\rightarrow +\infty$, belongs to  $S$.
\end{itemize}
Then, $x(t)$ converges weakly to an element in $S$ as $t\rightarrow +\infty$.
\end{lemma}

 At the end of this section,  using a similar argument as in \cite[Theorem 3.1]{xuwen}, \cite[Theorem 12]{qp2023} and \cite[Theorem 2]{coap25s}, we establish the existence and uniqueness of the strong global solution of the  system $(\ref{DS})$.
 \begin{theorem}\label{the2.1}
Suppose that $\nabla g$ is $L_g$-Lipschitz continuous on $\mathcal{X}$. Then, for every initial point $( x( t_0 ) ,\dot{x}( t_0 ) )\in \mathcal{X}\times \mathcal{X}$, the system $(\ref{DS})$ admits a unique strong global solution $x:[ t_0,+\infty ) \rightarrow \mathcal{X}$.
\end{theorem} 
\begin{proof}  
Let $X( t ): =( x( t ) ,y( t ) )$ and
$
F( t,u,v ): =\left( v-\beta \nabla g(u) ,-\frac{\alpha}{t^q}v-\left( \delta ( t ) -\frac{\beta\alpha}{t^q} \right) \nabla g(u ) -\frac{a}{t^p}u \right).
$
Then, the system $(\ref{DS})$ becomes
\begin{eqnarray}\label{2.3}
	\dot{X}( t ) =F( t,X( t ) ).
\end{eqnarray}
Clearly, for every $( u,v ) $ and $ ( \bar{u},\bar{v} ) \in \mathcal{X}\times \mathcal{X}$, it follows   that
\begin{eqnarray*}\label{13}
\small{\begin{split}
&\lVert F( t,u,v ) -F( t,\bar{u},\bar{v} ) \rVert\\
=&\sqrt{\lVert ( v-\bar{v} ) +\beta \left( \nabla g( \bar{u} ) -\nabla g( u) \right) \rVert ^2+\left\lVert \frac{\alpha}{t^q}( \bar{v}-v ) +\left( \delta ( t ) -\frac{\beta\alpha}{t^q} \right)( \nabla g( \bar{u} ) -\nabla g( u )) +\frac{a}{t^p}( \bar{u}-u ) \right\rVert ^2}\\
\le &\sqrt{\left( 2+\frac{2\alpha ^2}{t^{2q}} \right) \lVert v-\bar{v} \rVert ^2+\left( 2\beta ^2+4\left( \delta ( t ) -\frac{\beta\alpha}{t^q} \right) ^2\right)\lVert \nabla g( \bar{u}) -\nabla g( u ) \rVert ^2+\frac{4a^2}{t^{2p}}\lVert \bar{u}-u \rVert ^2}.
\end{split}}
\end{eqnarray*}
This together  with the $L_g$-Lipschitz continuity of $\nabla g$   gives
\begin{eqnarray*}
\begin{split}
&\lVert F( t,u,v ) -F( t,\bar{u},\bar{v} ) \rVert\\
\le& \sqrt{\left( 2+\frac{2\alpha ^2}{t^{2q}} \right) \lVert v-\bar{v} \rVert ^2+\left( 2\beta ^2L_{g}^{2}+4L_{g}^{2}\left( \delta ( t ) -\frac{\beta\alpha}{t^q} \right) ^2+\frac{4 a^2}{t^{2p}} \right) \lVert \bar{u}-u \rVert ^2}\\
\le &\sqrt{2+\frac{2\alpha ^2}{t^{2q}}+2\beta ^2L_{g}^{2}+4L_{g}^{2}\left( \delta ( t ) -\frac{\beta\alpha}{t^q} \right) ^2+\frac{4 a^2}{t^{2p}}}\sqrt{\lVert v-\bar{v} \rVert ^2+\lVert \bar{u}-u \rVert ^2}\\
\le& \left( \sqrt{2}+\sqrt{2}\frac{\alpha}{t^q}+\sqrt{2}\beta L_g+2L_g\left| \delta ( t) -\frac{\beta\alpha}{t^q} \right| +\frac{2 a}{t^p} \right) \lVert ( u,v ) -( \bar{u},\bar{v} ) \rVert.
\end{split}
\end{eqnarray*}
Set $L( t ):=\sqrt{2}+\sqrt{2}\frac{\alpha}{t^q}+\sqrt{2}\beta L_g+2L_g\left| \delta ( t ) -\frac{\beta\alpha}{t^q} \right| +\frac{2 a}{t^p}$. Clearly, $L( \cdot ) \in L_{loc}^{1}( [ t_0,+\infty ) )$ and
\begin{eqnarray}\label{2.5}
\lVert F( t,u,v ) -F( t,\bar{u},\bar{v} ) \rVert \le L( t ) \lVert( u,v ) -( \bar{u},\bar{v} ) \rVert.
\end{eqnarray} On the other hand, for all $ u,  v\in \mathcal{X}$, it is  easy to show that
\begin{eqnarray*}
\begin{split}
&\int_{t_0}^T{\lVert F( t,u,v ) \rVert dt}\\
=&\int_{t_0}^T{\sqrt{\lVert v-\beta \nabla g( u ) \rVert ^2+\left\lVert -\frac{\alpha}{t^q}v-\left( \delta ( t ) -\frac{\beta\alpha}{t^q} \right) \nabla g( u ) -\frac{a}{t^p}u \right\rVert ^2}}dt\\
\le &\int_{t_0}^T{\sqrt{\left( 2+\frac{2\alpha ^2}{t^{2q}} \right) \lVert v \rVert ^2+\left(2\beta ^2+4\left( \delta ( t ) -\frac{\beta\alpha}{t^q} \right) ^2\right)\lVert \nabla g( u ) \rVert ^2+\frac{4 a^2}{t^{2p}}\lVert u \rVert ^2}}dt\\
\le& \sqrt{\lVert v \rVert ^2+\lVert \nabla g( u ) \rVert ^2+\lVert u \rVert ^2}\int_{t_0}^T{\sqrt{2+\frac{2\alpha ^2}{t^{2q}}+2\beta ^2+4\left( \delta ( t ) -\frac{\beta\alpha}{t^q} \right) ^2+\frac{4 a^2}{t^{2p}}}}dt.
\end{split}
\end{eqnarray*}
Then,
$
F\left( \cdot ,u,v \right) \in L_{loc}^{1}\left( \left[ t_0,+\infty \right) ,\ \mathcal{X}\times \mathcal{X} \right).
$ 
 Combining this   with ${(\ref{2.5})}$ and $L( \cdot ) \in L_{loc}^{1}( [ t_0,+\infty ) )$,  it follows from Cauchy-Lipschitz-Picard Theorem   that there exists a unique global solution of system  ${(\ref{2.3})}$. The proof  is complete. \qed
\end{proof}

\section{Asymptotic analysis of  the    system (\ref{DS})}

In this section, based on Lyapunov's ananlysis, we first establish  fast
convergence properties of  the objective function value along the trajectory generated
by the  system (\ref{DS}).   Then, under a tuning of these parameters, we improve the
convergence rate of the objective function value  and  show that the trajectory generated by the system $(\ref{DS})$ converges weakly to a minimizer of the  problem $(\ref{g})$.  Finally, we  show that both the fast convergence rate of the function value and the strong convergence of the trajectory can be achieved simultaneously.
\begin{theorem}\label{the4.1}
Let $q<p\leq 2$. Assume that $a\geq q( 1-q ) $ for $p=2$.  Suppose that
\begin{eqnarray}\label{b}
 \mbox{ there exists } K_1>0  \mbox{ such that  } t \dot{\delta}( t ) <K_1 \delta ( t ), \mbox{ for } t  \mbox{ big enough}.
\end{eqnarray}
 Let  $x:[ t_0,+\infty)\rightarrow\mathcal{X}$ be the unique global solution of the system $(\ref{DS})$.
Then, for any $x^*\in \textup{argmin}~g$, it holds that
\begin{enumerate}
\item[{\rm (i)}] If $2q<p\leq2$, one has $g( x( t ) ) -g( x^* ) =\mathcal{O}\left( \frac{1}{\delta ( t ) t^{2q}} \right)$ and $\lVert \dot{x}( t ) +\beta \nabla g( x( t )) \rVert =\mathcal{O}\left( \frac{1}{t^q} \right)$, as $t\rightarrow +\infty$.

\item[{\rm (ii)}] If $q<p\le 2q$, one has $g( x( t ) ) -g( x^* )=\mathcal{O}\left( \frac{1}{\delta ( t ) t^p} \right)$ and $\lVert \dot{x}( t ) +\beta \nabla g( x( t ))  \rVert =\mathcal{O}\left( \frac{1}{t^{\frac{p}{2}}} \right)$, as $t\rightarrow +\infty$.
\end{enumerate}
\end{theorem}
\begin{proof} Let $0<b<\alpha$. We introduce the   energy function
\begin{eqnarray}\label{2q}
\begin{split}
\mathcal{E}_b( t )=&\left( \delta ( t ) t^{2q}-\beta ( b+2qt^{q-1}-\alpha ) t^q \right) \left( g( x( t ) ) -g( x^* ) \right)\\&+\frac{1}{2} at^{2q-p} \lVert x( t ) \rVert ^2+\frac{1}{2}\lVert b( x( t ) -x^* ) +t^q\left( \dot{x}( t ) +\beta \nabla g( x( t ) ) \right) \rVert ^2\\
&+\frac{1}{2}b( \alpha -b-qt^{q-1} )\lVert x( t ) -x^* \rVert ^2.
\end{split}
\end{eqnarray}
By $\lim _{t\rightarrow +\infty}\delta ( t ) =+\infty$, $0<q<1$ and $0<b<\alpha$, there exists $t_1 \geq t_0 $ such that $\delta ( t ) t^{2q}-\beta ( b+2qt^{q-1}-\alpha ) t^q\geq0$ and $  \frac{1}{2}b( \alpha -b-qt^{q-1} ) \geq0 $ for all $t\geq t_1.$ Then,  $\mathcal{E}_b( t )\geq 0$  for all $t\geq t_1.$

 Now, we   analyze the time derivative of $\mathcal{E}_b( t )$.
Clearly, together with $(\ref{DS})$, we obtain
\begin{eqnarray}\label{6q}
\begin{split}
 &\dot{\mathcal{E}}_b( t )\\ =& \left( \dot{\delta}( t) t^{2q}+\Big( 2\delta ( t ) t^q-\beta b-2( 2q-1 ) \beta t^{q-1}+\alpha \beta \Big) qt^{q-1} \right) ( g( x( t ) ) -g( x^* ) )\\
&+\left( b-\alpha +qt^{q-1} \right) t^q\lVert \dot{x}( t ) \rVert ^2+\frac{1}{2}a( 2q-p ) t^{2q-p-1} \lVert x( t ) \rVert ^2\\
&-a \beta t^{2q-p}\left< \nabla g( x( t ) ) ,x( t ) \right> +\left( \beta ^2qt^{2q-1}-\beta \delta ( t ) t^{2q} \right) \lVert \nabla g( x( t ) ) \rVert ^2\\
&+ \frac{1}{2}bq( 1-q ) t^{q-2} \lVert x( t ) -x^* \rVert ^2-bt^q\left< at^{-p}x( t) + (\delta ( t ) -q\beta t^{-1})\nabla g(x(t))  ,x( t ) -x^* \right>.
\end{split}
\end{eqnarray}
Let the function $\phi_t:\mathcal{X}\rightarrow \mathbb{R}$ be defined as
$
\phi_t( x) := ( \delta ( t ) - q\beta t^{-1}  ) g( x ) +\frac{a}{2t^p}\lVert x \rVert ^2.
$
Since $\lim _{t\rightarrow +\infty}\delta ( t ) =+\infty$, there exists $t_2>t_1$ such that
$
\delta ( t ) - q\beta t^{-1}\geq0 $  for all $ t\geq t_2.
$
Then, $\phi_t(x)$ is a $\frac{a}{2t^p}$-strongly convex function for all $t\geq t_2,$  which follows that for all $  t\geq t_2$,
\begin{eqnarray*}
\begin{split}
&\left( \delta ( t ) - q\beta t^{-1}  \right) g( x^* ) +\frac{a}{2t^p}\lVert x^* \rVert ^2-\left( \delta ( t ) - q\beta t^{-1}  \right) g( x( t ) ) -\frac{a}{2t^p}\lVert x( t ) \rVert ^2\\
\ge& \left< \left( \delta ( t ) - q\beta t^{-1}  \right) \nabla g( x( t ) ) +at^{-p}x( t ) ,  x^*- x( t )\right> +\frac{a}{2t^p}\lVert x( t ) -x^* \rVert ^2.
\end{split}
\end{eqnarray*}
Consequently,
\begin{eqnarray}\label{7q}
\begin{split}
&-bt^q\left<at^{-p}x( t )+ \left( \delta ( t ) - q\beta t^{-1} \right) \nabla g( x( t ) )  ,x( t ) -x^* \right> \\
\le& bt^q\left( \delta ( t ) - q\beta t^{-1}  \right) \left( g( x^* ) -g( x( t ) ) \right) + \frac{1}{2}abt^{q-p} (\lVert x^* \rVert ^2-   \lVert x( t ) \rVert ^2-  \lVert x( t ) -x^* \rVert ^2).
\end{split}
\end{eqnarray}
Further, note that
\begin{eqnarray}\label{uv}
-a \beta t^{2q-p}\left< \nabla g( x( t ) ) ,x( t ) \right> \le \frac{1}{2}a \beta t^{2q} \lVert \nabla g( x( t ) ) \rVert ^2+ \frac{1}{2}a \beta t^{2q-2p} \lVert x( t ) \rVert ^2.
\end{eqnarray}
Thus, combining $(\ref{6q})$ with $(\ref{7q})$ and $(\ref{uv})$, we get for all $ t\geq t_2,$
\begin{eqnarray}\label{8q}
\begin{split}
 \dot{\mathcal{E}}_b( t )
\le&\Big( \dot{\delta}( t ) t^{2q}+ 2q\delta ( t ) t^{2q-1}-2q( 2q-1 ) \beta t^{2q-2}\\&+\alpha \beta qt^{q-1}-b \delta ( t )t^q \Big) ( g( x( t ))-g( x^* ) )+( b-\alpha +qt^{q-1}) t^q\lVert \dot x( t ) \rVert ^2\\
&+\frac{1}{2}\left(  a( 2q-p ) t^{2q-p-1} - abt^{q-p} + a\beta t^{2q-2p}  \right) \lVert x( t ) \rVert ^2\\
&+\frac{1}{2}\left(  bq( 1-q ) t^{q-2} - abt^{q-p}  \right) \lVert x( t ) -x^* \rVert ^2\\
&+\left( \beta ^2qt^{2q-1}-\beta \delta ( t ) t^{2q}+ \frac{1}{2}a\beta t^{2q}  \right) \lVert \nabla g( x( t ) ) \rVert ^2+ \frac{1}{2}abt^{q-p} \lVert x^* \rVert ^2.
\end{split}
\end{eqnarray}

On the other hand, note that
$\frac{1}{2}\lVert b( x( t ) -x^* ) +t^q ( \dot{x}( t ) +\beta\nabla g( x( t ) )  ) \rVert ^2
\le  b^2\lVert x ( t ) -x^* \rVert ^2+2t^{2q}\lVert \dot{x}( t ) \rVert ^2+2\beta ^2 t^{2q}\lVert \nabla  g( x( t ) ) \rVert ^2.
$
Then, it follows from $(\ref{2q})$ that
\begin{eqnarray}\label{9q}
\begin{split}
\mathcal{E}_b( t )  \le& \left( \delta ( t) t^{2q}-\beta ( b+2qt^{q-1}-\alpha \right) t^q ) \left( g( x( t ) ) -g( x^* ) \right) +2t^{2q}\lVert \dot{x}( t ) \rVert ^2\\
&+2\beta ^2t^{2q}\lVert \nabla  g( x( t ) ) \rVert ^2+\frac{1}{2}b\left( \alpha +b-qt^{q-1} \right)\lVert x (t) -x^* \rVert ^2+ \frac{1}{2}at^{2q-p} \lVert x( t ) \rVert ^2.
\end{split}
\end{eqnarray}
Let $r:=\max \{q,p-q\}$ and   $K$ be an arbitrarily positive constant. By $(\ref{8q})$ and $(\ref{9q})$, we obtain that  for all $ t\geq t_2,$
\begin{eqnarray}\label{10q}
\begin{split}
&\dot{\mathcal{E}}_b( t ) +\frac{K}{t^r}\mathcal{E}_b( t )\\
\le &t^q \left( \dot{\delta}( t )t^q+2q\delta ( t )t^{q-1}-b\delta ( t )+K\delta ( t )t^{q-r}+\alpha \beta qt^{-1}-2( 2q-1 ) \beta qt^{q-2} \right.\\
&\quad\quad\quad -K\beta ( b+2qt^{q-1}-\alpha ) t^{-r} \Big) ( g( x( t ) ) -g( x^* ) )\\
&+\left( b-\alpha +qt^{q-1}+2Kt^{q-r} \right) t^q\lVert \dot{x}( t ) \rVert ^2\\
&+\frac{1}{2}\left( a( 2q-p ) t^{q-1}-ab+a\beta t^{q-p}+Kat^{q-r} \right) t^{q-p}\lVert x( t ) \rVert ^2\\
&+\left( \beta ^2qt^{-1}-\beta \delta( t ) +\frac{1}{2}a\beta+2K\beta ^2t^{-r} \right) t^{2q}\lVert \nabla g( x( t ) ) \rVert ^2\\
&+\frac{1}{2}\Big( bq( 1-q ) t^{p-2}-ab+Kb( \alpha +b-qt^{q-1} ) t^{p-q-r}\Big) t^{q-p}\lVert x( t ) -x^* \rVert ^2\\
&+ \frac{1}{2}abt^{q-p} \lVert x^* \rVert ^2.
\end{split}
\end{eqnarray}
Together with  $(\ref{b})$, $0<b<\alpha$ and $a\geq q( 1-q ) $ for $p=2$,  we can take $0<K<\min\left\{  b,\frac{\alpha -b}{2}, \frac{a}{\alpha +b} \right\}$  for $q<p<2$ and $0<K<\min \left\{ b,\frac{\alpha -b}{2}, \frac{a-q( 1-q )}{\alpha +b} \right\}$  for $p=2$. Then, we can deduce from $(\ref{10q})$ that  there exists $t_3\ge t_2$ such that
$
\dot{\mathcal{E}}_b( t ) +\frac{K}{t^r}\mathcal{E}_b( t ) \le \frac{1}{2}abt^{q-p} \lVert x^* \rVert ^2 $    for all $t\ge t_3.
$
Then, using a similar argument as in \cite[Theorem 1]{qp2023}, we can get the desired results. The proof is complete.\qed
\end{proof}

In the case that $q+1<p\le2$, we can also obtain some convergence results and integral estimates.

\begin{theorem}\label{the4.3}
Let  $q+1<p\le2$ and let  $a\ge q\left( 1-q \right) $ for $p=2$. Assume that the condition $(\ref{b})$ holds. Let  $x:[ t_0,+\infty)\rightarrow\mathcal{X}$ be the unique global solution of the system $(\ref{DS})$. Then,  for any $x^*\in \textup{argmin}~g$, the trajectory $x(t)$ is bounded,
\begin{eqnarray*}
g( x( t ) ) -g( x^* ) =\mathcal{O}\left( \frac{1}{\delta ( t ) t^{2q}} \right), \lVert \dot{x}( t ) +\beta \nabla g( x( t )) \rVert =\mathcal{O}\left( \frac{1}{t^q} \right), \mbox{ as } t\rightarrow +\infty,
\end{eqnarray*}
and
\begin{eqnarray}\label{weakz}
\left\{
\begin{split}
&\int_{t_0}^{+\infty}\delta ( t ) t^q( g(x(t)) -g(x^*) )dt <+\infty,~~\int_{t_0}^{+\infty}\ t^{q-p}\lVert x( t) \rVert ^2dt <+\infty,\\
&\int_{t_0}^{+\infty}   t^{q}\lVert \dot{x}( t) \rVert ^2dt <+\infty,~~~~~~~~~~~~~~~~~\int_{t_0}^{+\infty}\delta (t ) t^{2q}\lVert \nabla g( x( t )) \rVert ^2dt <+\infty,\\
&\int_{t_0}^{+\infty} t^{q-p}\lVert x( t ) -x^* \rVert ^2dt <+\infty.\end{split}
\right.
\end{eqnarray}
Moreover, we have $\lim _{t\rightarrow +\infty}\mathcal{E}_b( t )\in \mathbb{R}.$
\end{theorem}
\begin{proof} We first consider the coefficients of the right side of   $(\ref{8q})$.
Obviously, there exist  $t_2'\ge t_2$ and $C_1,C_2,C_3>0$ such that for all $t\ge t_2'$,
\begin{eqnarray*}
\left\{{\begin{split}
&\dot{\delta}( t ) t^{2q}+ 2q\delta ( t ) t^{2q-1}-2q( 2q-1 ) \beta t^{2q-2}+\alpha \beta  qt^{q-1}-b\delta ( t )t^q
\\ \leq& (K_1+2q)\delta ( t ) t^{2q-1}-2q( 2q-1 ) \beta t^{2q-2}+\alpha \beta  qt^{q-1}-b \delta ( t )t^q \le  -C_1\delta ( t) t^q,\\
&( b-\alpha +qt^{q-1} ) t^q\le -C_2t^{q},~~~
\beta ^2qt^{2q-1}-\beta \delta ( t ) t^{2q}+\frac{1}{2}ab\beta t^{2q}
\le -C_3 \delta ( t ) t^{2q},
\end{split}}\right.
\end{eqnarray*}
and
$\frac{1}{2} ( a( 2q-p ) t^{2q-p-1} -abt^{q-p}+ab\beta t^{2q-2p} ) t^{q-p}\|x(t)\|^2
+\frac{1}{2} (  bq( 1-q ) t^{q-2}-abt^{q-p} ) \|x(t)-x^*\|^2\le 0.
$
Combining the above inequalities and $(\ref{8q})$,  we get for all $t\geq t_2',$
\begin{eqnarray}\label{dd}
\begin{split}
&\dot{\mathcal{E}}_b( t ) +C_1\delta ( t) t^q( g( x( t ) ) -g^* ) +C_2t^{q}\lVert \dot{x}( t ) \rVert ^2+C_3 \delta ( t ) t^{2q}\lVert \nabla g( x( t ) ) \rVert ^2\\
\le &\frac{1}{2} ab\lVert x^* \rVert ^2 \frac{1}{t^{p-q}}.
\end{split}
\end{eqnarray}
 Integrating the above inequality from $t_2'$ to $t$ and note that $p>q+1$,   there exists $C_6>0$ such that
\begin{eqnarray}\label{17q}
\begin{split}
&\int_{t_2'}^t{\dot{\mathcal{E}}_b( s )}ds+C_1\int_{t_2'}^t{\delta ( s ) s^q( g( x( s )) -g(x^*) )}ds+C_2\int_{t_2'}^t s^{q}\lVert \dot{x}( s) \rVert ^2ds\\
&+C_3\int_{t_2'}^t\delta ( s ) s^{2q}\lVert \nabla g( x( s )) \rVert ^2ds
\le
\frac{1}{2}ab\lVert x^* \rVert ^2 \int_{t_2'}^t{\frac{1}{s^{p-q}}}ds\\
&\le  \frac{1}{2}ab\lVert x^* \rVert ^2 \int_{t_2'}^{+\infty}{\frac{1}{s^{p-q}}ds\le C_6}.
\end{split}
\end{eqnarray}
Then, $ \mathcal{E}_b( t )$ is bounded, which implies that the trajectory $x(t)$ is bounded and
\begin{eqnarray*}
g( x( t ) ) -g( x^* ) =\mathcal{O}\left( \frac{1}{\delta ( t ) t^{2q}} \right), \ \ \lVert \dot{x}( t ) +\beta \nabla g( x( t )) \rVert =\mathcal{O}\left( \frac{1}{t^q} \right), \mbox{ as } t\rightarrow +\infty.
\end{eqnarray*}
Naturally, it follows from $(\ref{17q})$ that
$\int_{t_0}^{+\infty}\delta ( t ) t^q( g(x(t)) -g(x^*) )dt <+\infty$, $\int_{t_0}^{+\infty}   t^{q}\lVert \dot{x}( t) \rVert ^2dt <+\infty,$ and
$\int_{t_0}^{+\infty}\delta (t ) t^{2q}\lVert \nabla g( x( t )) \rVert ^2dt <+\infty.$
From the boundedness of the trajectory $x(t)$ and $\int_{t_0}^{+\infty} t^{q-p}dt <+\infty$, we have
$
\int_{t_0}^{+\infty}\ t^{q-p}\lVert x( t) \rVert ^2dt <+\infty$ and $\int_{t_0}^{+\infty} t^{q-p}\lVert x( t ) -x^* \rVert ^2dt <+\infty.
$
Further, set $F(t):=\mathcal{E}_b( t )$ and $G(t):= \frac{1}{2}ab\lVert x^* \rVert ^2 \frac{1}{t^{p-q}}$. Note that $\int_{t_2'}^{+\infty}\frac{1}{s^{p-q}}ds\le+\infty$. Then,   together with $(\ref{dd})$ and Lemma \ref{L1},  we have  $\lim _{t\rightarrow +\infty}\mathcal{E}_b( t )\in \mathbb{R}.$
The proof is complete.\qed
\end{proof}
\begin{remark} In the case where $\beta=0$ and $\delta(t)=1$,  the system (\ref{DS}) reduces to the system (\ref{DS1}), and
  the convergence rate of $\mathcal{O}\left(\frac{1}{t^{2q}}\right)$ for the function value along the trajectory generated by the system (\ref{DS1}) has been obtained in \cite[Theorem 4]{qp2023}.  However, in Theorem \ref{the4.3}, under the condition that  $\lim _{t\rightarrow +\infty}\delta ( t ) =+\infty$,    the convergence rate of $\mathcal{O}\left(\frac{1}{\delta(t)t^{2q}}\right)$ for  the function value is obtained   for the system (\ref{DS}).
\end{remark}

In the following theorem, we  demonstrate that the trajectory generated by the system $(\ref{DS})$ converges weakly to a minimizer of the problem $(\ref{g})$.  Additionally, we show that the convergence rate of the  function value can be enhanced under the conditions specified in Theorem $\ref{the4.3}$.
\begin{theorem}\label{the4.4}
Suppose that the conditions of Theorem $\ref{the4.3}$ are satisfied.   Then, $x( t )$ converges weakly to an element of $\textup{argmin}g$, as $t\rightarrow +\infty$. Further, for any $x^*\in \mathrm{argmin} g$,
\begin{eqnarray*}
g( x( t ) ) -g( x^* ) =o\left( \frac{1}{\delta ( t ) t^{2q}} \right) \mbox{ and } \lVert \dot{x}( t ) +\beta \nabla g( x( t )) \rVert =o\left( \frac{1}{t^q} \right), \mbox{ as } t\rightarrow +\infty.
\end{eqnarray*}
\end{theorem}
\begin{proof}
We will apply   Lemma $\ref{L3}$  to establish that $x(t)$ converges weakly to an element of $\textup{argmin}g$. To do this,  we first show that
$
\underset{\,\,t\rightarrow +\infty}{\lim}\lVert x( t ) -x^* \rVert \in \mathbb{R}$ for all $  x^*\in \textup{argmin}g.
$
Indeed, by $(\ref{uv})$, the equality $(\ref{6q})$ can be written as
\begin{eqnarray*}
\begin{split}
\dot{\mathcal{E}}_b( t ) &\le \left( \dot{\delta}( t ) t^{2q}+\left( 2\delta ( t ) t^q-\beta b-2( 2q-1 ) \beta t^{q-1}+\alpha \beta \right) qt^{q-1} \right) ( g( x( t ) ) -g( x^* ) )\\
&\quad+\left( b-\alpha +qt^{q-1} \right) t^q\lVert \dot{x}( t ) \rVert ^2+\frac{1}{2}\left( a( 2q-p ) t^{2q-p-1} + a\beta t^{2q-2p}  \right) \lVert x( t ) \rVert ^2\\
&\quad+\left( \beta ^2qt^{2q-1}-\beta \delta ( t) t^{2q}+ \frac{1}{2}a\beta t^{2q}  \right) \lVert \nabla g( x( t) ) \rVert ^2+\frac{1}{2}bq( 1-q ) t^{q-2}\lVert x( t ) -x^* \rVert ^2\\
&\quad\quad\quad-bt^q\left< at^{-p}x( t ) +\left( \delta ( t ) -q\beta t^{-1} \right) \nabla g( x( t ) ) ,x( t ) -x^* \right>, ~~\forall t\ge t_1.
\end{split}
\end{eqnarray*}
Note that there exists $t_1'>t_1$ such that $ b-\alpha +qt^{q-1} \leq 0$ and $\beta ^2qt^{2q-1}-\beta \delta ( t) t^{2q}+\frac{1}{2}a\beta t^{2q}\leq0$  for all $t\ge t_1'$. Therefore, for all $t\ge t_1'$,
\begin{eqnarray*}
\begin{split}
\dot{\mathcal{E}}_b( t )\le& \left( \dot{\delta}( t ) t^{2q}+\left( 2\delta ( t ) t^q-\beta b-2( 2q-1 ) \beta t^{q-1}+\alpha \beta \right) qt^{q-1} \right) ( g( x( t ) ) -g( x^* ) )\\
&+\frac{1}{2}\left(  a( 2q-p ) t^{2q-p-1}+a\beta t^{2q-2p} \right) \lVert x( t ) \rVert ^2+\frac{1}{2}bq( 1-q ) t^{q-2} \lVert x( t ) -x^* \rVert ^2\\
&-bt^q\left< at^{-p}x( t ) +\left( \delta ( t ) -q\beta t^{-1} \right) \nabla g( x( t ) ) ,x( t ) -x^* \right>.
\end{split}
\end{eqnarray*}
Let $T=\max\{t_2, t_1'\}$.
By integrating it from $T$ to $t$, we obtain
\begin{eqnarray}\label{27q}
\begin{split}
&\int_{T}^t{bs^q\left< ( \delta ( s ) -q\beta s^{-1} ) \nabla g( x( s ) ) ,x( s ) -x^* \right> ds}\\
\le &\mathcal{E}_b( T ) -\mathcal{E}_b( t )+\frac{1}{2}\int_{T}^t{\left(a( 2q-p ) s^{2q-p-1}+a\beta s^{2q-2p}\right)\| x( s ) \|^2ds}\\
& +\int_{T}^t{\left( \dot{\delta}( s ) s^{2q}+\left( 2\delta ( s ) s^q-\beta b-2( 2q-1 ) \beta s^{q-1}+\alpha \beta \right) qs^{q-1} \right) ( g( x( s ) ) -g( x^* ) ) ds}\\
&+\frac{1}{2}\int_{T}^t{bq( 1-q ) s^{q-2}}\lVert x( s ) -x^* \rVert ^2ds-\int_{T}^t{bs^q\left< as^{-p}x( s ) ,x( s) -x^* \right> ds}.
\end{split}
\end{eqnarray}
By virtue of  $\delta ( s ) -q\beta s^{-1} \geq0$ and the monotonicity of $\nabla g$, we have
\begin{eqnarray}\label{27q01}
\langle( \delta ( s ) -q\beta s^{-1} )\nabla g( x( s ) ) ,x( s ) -x^* \rangle \ge 0,\forall s\geq T.
\end{eqnarray}
 Note that
$
abs^q\left|\left< s^{-p}x( s ) ,x( s ) -x^* \right>\right| \le abs^{q-p}( \lVert x( s ) \rVert ^2+\lVert x(s ) -x^* \rVert ^2 ).
$
Then, together  with $(\ref{27q})$, $(\ref{27q01})$,   $\int_{t_0}^{+\infty} \delta (t)t^{q}(g( x( t ) ) -g( x^* ) )dt <+\infty$, $\int_{t_0}^{+\infty}\ t^{q-p}\lVert x( t) \rVert ^2dt <+\infty$ and $\int_{t_0}^{+\infty} t^{q-p}\lVert x( t ) -x^* \rVert ^2dt <+\infty$, we have
\begin{eqnarray}\label{28q}
\int_{t_0}^{+\infty}t^{q-1}\left< \nabla g( x( t ) ) ,x( t ) -x^* \right> dt<+\infty \mbox{ and } \int_{t_0}^{+\infty}\left< \nabla g( x( t ) ) ,x( t ) -x^* \right> dt<+\infty.
\end{eqnarray}
Let $0<b_1<\alpha$, $0<b_2<\alpha$ and $b_1\ne b_2$. Then, for all $t\ge t_0$,
\begin{eqnarray*}
\begin{split}
\mathcal{E}_{b_1}(t)-\mathcal{E}_{b_2}(t)=&( b_1-b_2 )\Big( -\beta t^q( g( x( t ) ) -g^* )+t^q\left< \dot{x}( t ) +\beta \nabla g( x( t ) ) ,x( t ) -x^* \right> \\
&\left. +\frac{1}{2}(\alpha -qt^{q-1})\lVert x( t ) -x^* \rVert ^2 \right).
\end{split}
\end{eqnarray*}
By Theorem $\ref{the4.3}$, we have
$
\underset{t\rightarrow +\infty}{\lim}\left( \mathcal{E}_{b_1}( t ) -\mathcal{E}_{b_2}( t ) \right) \in \mathbb{R}$ and $\underset{t\rightarrow +\infty}{\lim}t^q( g( x( t ) ) -g^* ) \in \mathbb{R}.$
Consequently,
$
\underset{t\rightarrow +\infty}{\lim}\left( t^q\left< \dot{x}( t ) +\beta \nabla g( x( t ) ) ,x( t) -x^* \right> +\frac{1}{2}(\alpha -qt^{q-1})\lVert x( t ) -x^* \rVert ^2 \right)\in \mathbb{R}.
$
Set
$
k( t ) :=t^q\left< \dot{x}( t ) +\beta \nabla g( x( t ) ) ,x( t ) -x^* \right> +\frac{1}{2}(\alpha -qt^{q-1})\lVert x( t ) -x^* \rVert ^2
$
and
\begin{eqnarray}\label{qtt}
Q( t ) :=\frac{1}{2}\lVert x( t ) -x^* \rVert ^2+\beta \int_{t_0}^t{\left< \nabla g( x( s ) ) ,x( s ) -x^* \right>}ds.
\end{eqnarray}
Clearly,
$
( \alpha -qt^{q-1} ) Q( t ) +t^q\dot{Q}( t ) =k( t ) +\beta( \alpha -qt^{q-1} ) \int_{t_0}^t{\left< \nabla g( x( s )) ,x( s ) -x^* \right>}ds.
$
Together with $(\ref{28q})$ and $\underset{t\rightarrow +\infty}{\lim}k( t )\in \mathbb{R}$, we obtain  $\underset{t\rightarrow +\infty}{\lim}( \alpha -qt^{q-1} ) Q( t ) +t^q\dot{Q}( t )\in \mathbb{R}.$ Then, it follows from  Lemma $\ref{L2}$ that $\underset{t\rightarrow +\infty}{\lim}Q( t )\in \mathbb{R}$. Combining this with   $(\ref{qtt})$ and $(\ref{28q})$, we obtain
$
\underset{t\rightarrow +\infty}{\lim}\lVert x( t ) -x^* \rVert \in \mathbb{R}.
$

Next, let $\bar{x}\in \mathcal{H}$ be a weak sequential limit point of $x( t )$. Then, there exists a sequence $( t_n ) _{n\in N}\subseteq [ t_0,+\infty )$ with $\lim _{n\rightarrow +\infty}t_n=+\infty$   such that $x( t_n )$ converges weakly to $\bar{x}$, as $n\rightarrow +\infty$. Moreover,
$
g( \bar{x} ) \le \underset{n\rightarrow +\infty}{\liminf}g( x( t_n ) )=\underset{n\rightarrow +\infty}{\lim}g( x( t_n ) )=\min g.
$
This means that $\bar{x}\in \textup{argmin}g$.
Thus,  from Lemma $\ref{L3}$,  $x(t)$ converges weakly to a minimizer of $g$, as $t\rightarrow +\infty$.

On the other hand, the  energy function $\mathcal{E}_b( t )$ introduced in $(\ref{2q})$ can be written as
\begin{eqnarray*}
\begin{split}
\mathcal{E}_b( t ) =&~\left( \delta ( t ) t^{2q}-\beta ( b+2qt^{q-1}-\alpha ) t^q \right) \left( g( x( t ) ) -g( x^* ) \right) +\frac{1}{2} a  t^{2q-p} \lVert x( t ) \rVert ^2\\
&\quad+\frac{1}{2}t^{2q}\lVert  \dot{x}( t ) +\beta \nabla g( x( t ) )  \rVert ^2+bk(t).
\end{split}
\end{eqnarray*}
Since $p>q+1>2q$ and the trajectory $x(t)$ is bounded, we have
\begin{eqnarray}\label{00tr}
\underset{t\rightarrow +\infty}{\lim} \frac{1}{2} a  t^{2q-p}\lVert x( t ) \rVert ^2=0.
\end{eqnarray}
Set $\Phi(t):=\left( \delta ( t ) t^{2q}-\beta ( b+2qt^{q-1}-\alpha ) t^q \right) \left( g( x( t ) ) -g( x^* ) \right) +\frac{1}{2}t^{2q}\lVert  \dot{x}( t ) +\beta \nabla g( x( t ) ) \rVert ^2.$
Combining   $(\ref{00tr})$ with $\underset{t\rightarrow +\infty}{\lim} \mathcal{E}_b( t ) \in \mathbb{R}$ and $\underset{t\rightarrow +\infty}{\lim}k( t )\in \mathbb{R}$, we have
$
\underset{t\rightarrow +\infty}{\lim}\Phi(t)\in \mathbb{R}.
$
Then, there exists $t_0'>0$ such that for  all  $t\geq t_0'$,
\begin{eqnarray}\label{phi}
0\leq\frac{1}{t^q}\Phi(t) \leq \delta ( t ) t^{q}\left( g( x( t ) ) -g( x^* ) \right) +\frac{1}{2}t^{q}\lVert  \dot{x}( t ) +\beta \nabla g( x( t ) ) \rVert ^2.
\end{eqnarray}
By $(\ref{weakz})$ and $\lim _{t\rightarrow +\infty}\delta ( t ) =+\infty$, we have
\begin{eqnarray}\label{phi0}
\int_{t_0}^{+\infty}   t^{q}\lVert \dot{x}( t) \rVert ^2dt <+\infty \mbox{ and }\int_{t_0}^{+\infty} t^{q}\lVert \nabla g( x( t )) \rVert ^2dt <+\infty.
\end{eqnarray}
Therefore, it follows from $(\ref{phi})$ and $(\ref{phi0})$ that
$
\int_{t_0}^{+\infty}  \frac{1}{t^q}\Phi(t)dt <+\infty.
$
Note that $\frac{1}{t^q}\not\in L^1[t_0, +\infty)$ and the limit $\lim _{t\rightarrow +\infty}\Phi ( t )$  exists. Thus, $\lim _{t\rightarrow +\infty} \Phi(t) =0,$ which implies that
\begin{eqnarray*}
\begin{split}
&\lim _{t\rightarrow +\infty}\left( \delta ( t ) t^{2q}-\beta ( b+2qt^{q-1}-\alpha ) t^q \right) \left( g( x( t ) ) -g( x^* ) \right)\\
 =&\lim _{t\rightarrow +\infty}\frac{1}{2}t^{2q}\lVert  \dot{x}( t ) +\beta \nabla g( x( t ) ) \rVert ^2=0.
\end{split}
\end{eqnarray*}
Thus,
$
g( x( t ) ) -g( x^* ) =o\left( \frac{1}{\delta ( t ) t^{2q}} \right)$   and $\lVert \dot{x}( t ) +\beta \nabla g( x( t )) \rVert =o\left( \frac{1}{t^q} \right),$   as $ t\rightarrow +\infty.$
The proof is complete.\qed
\end{proof}
\begin{remark}
 Clearly, when $\beta\equiv0$ and $\delta(t)\equiv1$ (i.e., the time scaling and Hessian-driven damping terms are removed from the system $(\ref{DS})$), we recover the results provided in \cite[Theorem 7]{qp2023}.
\end{remark}

In the sequel, we establish a simultaneous result that concerns both the convergence rate of the function value
and the strong convergence of the trajectory generated by the   system $(\ref{DS})$ towards   a minimal norm solution of the problem (\ref{g}). To do this, let $x_t$ be the unique solution of the   strongly convex optimization problem
$
\underset{x\in \mathcal{X}} \min ~g_t(x):=g( x ) +\frac{a}{2\delta ( t ) t^p}\lVert x \rVert ^2 .
$
Obviously,
 $\nabla g_t(x_t)=\nabla g(x_t)+\frac{a}{\delta(t)t^p}x_t=0.$
 Let $\bar{x}^*$ be an element of the minimal norm   solution set of the   (\ref{g}).  Thus, $\bar{x}^*=\mbox{proj}_{\textup{argmin}g}0,$ where $\mbox{proj}$ denotes the projection operator.
By the classical properties of the Tikhonov regularization given in \cite{bot21mp}, we have
 $\bar{x}^*=\lim _{t\rightarrow +\infty}x_t$ and $\lVert x_t \rVert \le \lVert \bar{x}^* \rVert $  for all $t\ge t_0.$
 Further, we can deduce from \cite[Lemma 2]{attjde22}  that
\begin{eqnarray}\label{4.3}
\left\lVert \frac{d}{dt}x_t \right\rVert \le \left( \frac{p}{t}+\frac{\dot{\delta}( t )}{\delta ( t )} \right) \lVert x_t \rVert, \mbox{ for almost  every }  t\ge t_0.
\end{eqnarray}

\begin{theorem}\label{the5.1}
Let $q<p<q+1$. Suppose that $t^p\dot{\delta }( t )+pt^{p-1}\delta(t)-a\beta\geq 0$ and  the condition $(\ref{b})$ holds.   Let  $x:[ t_0,+\infty)\rightarrow\mathcal{X}$  be the unique global solution of the system $(\ref{DS})$.
 Then,
\begin{enumerate}
\item[{\rm (i)}] $x(t)$ converges strongly to $\bar{x}^*=\mathrm{proj}_{\textup{argmin}g}0,$ that is $\lim _{t\rightarrow +\infty}\lVert x( t ) -\bar{x}^* \rVert =0$.
\item[{\rm (ii)}]If $q<p<\frac{1}{2}(3q+1)$, as  $t\rightarrow +\infty$ it holds
 \begin{eqnarray*}
\left\{\begin{split}
\|\dot{x}( t ) +\beta \nabla g( x( t ) ) \|&=\mathcal{O}\left( \frac{1}{t^{\frac{p+1-\max \{ q,p-q \}}{2}}} \right), \ \ \ \lVert x( t ) -x_t \rVert =\mathcal{O}\left( \frac{1}{t^{\frac{1-q}{2}}} \right), \\
\ g( x( t ) ) -g( \bar{x}^*)&=\mathcal{O}\left( \frac{1}{\delta ( t ) t^p} \right).
\end{split}\right.
\end{eqnarray*}
\item[{\rm (iii)}] If $\frac{1}{2}(3q+1)\leq p<q+1$,  as  $t\rightarrow +\infty$ it holds
 \begin{eqnarray*}
\begin{split}
&\lVert \dot{x}( t ) +\beta \nabla g( x( t ) ) \rVert =\mathcal{O}\left( \frac{1}{t^{2q-p+1}} \right)  \mbox{ and } \lVert x( t ) -x_t \rVert =\mathcal{O}\left( \frac{1}{t^{q-p+1}} \right).
\end{split}
\end{eqnarray*}
Further,  if $\frac{1}{2}(3q+1)\le p\le \frac{1}{3}(4q+2)
$, then $g( x( t ) ) -g( \bar{x}^*)=\mathcal{O}\left( \frac{1}{\delta \left( t \right) t^p} \right) ,$ as $ t\rightarrow +\infty,$  and for $\frac{1}{3}(4q+2)<p<q+1$, one has $g( x( t ) ) -g( \bar{x}^* )=\mathcal{O}\left( \frac{1}{\delta ( t ) t^{4p-2q+2}} \right),$  as   $t\rightarrow +\infty.$

    \end{enumerate}
\end{theorem}
\begin{proof}Let $\alpha >b>0$. We consider the   energy function
\begin{eqnarray}\label{4.8}
\begin{split}
 {E}( t ) =&\delta ( t ) t^{2q}( g_t( x( t ) ) -g_t( x_t ) ) + \frac{1}{2}b\left( \alpha -b-qt^{q-1} \right) \lVert x( t ) -x_t \rVert ^2\\
& +\frac{1}{2}\left\lVert b( x( t ) -x_t ) +t^q\Big( \dot{x}( t ) +\beta \nabla g( x( t ) )\Big) \right\rVert ^2.
\end{split}
\end{eqnarray}

Obviously, there exists $t_1\ge t_0$ such that $\alpha -b-qt^{q-1} >0$ for all $ t\ge t_1$. Hence, $E( t ) \ge 0$  for all $t\ge t_1$.
 We now analyze the time derivative of $E( t )$.
Firstly,
\begin{eqnarray*}
\begin{split}
\frac{d}{dt} g_t( x( t ) )
=&\left<\dot{x}( t ),\nabla g( x( t ) ) \right> +\frac{a}{\delta ( t ) t^p}\left< \dot{x}( t ) ,x( t) \right> -\left(\frac{pa}{2\delta ( t ) t^{p+1}}+\frac{a\dot{\delta}( t )}{2\delta^2 ( t )t^p}\right)\lVert x( t ) \rVert ^2\\
=&\left<\dot{x}( t ) ,\nabla g_t( x( t ) ) \right> -\frac{a}{2\delta ( t )t^p} \left( \frac{p}{  t}+\frac{\dot{\delta}( t )}{\delta  ( t )} \right) \lVert x( t ) \rVert ^2.
\end{split}
\end{eqnarray*}
Similarly, together with $\nabla g_t( x_t ) =0$, we have
$
\frac{d}{dt}( g_t( x_t ) ) =-\frac{a}{2\delta ( t ) t^p}\left( \frac{p}{t }+\frac{\dot{\delta}( t )}{\delta ( t )} \right) \lVert x_t \rVert ^2.
$
Then,
\begin{eqnarray}\label{4.11}
\begin{split}
&\frac{d}{dt}\Big(\delta ( t ) t^{2q}\left( g_t( x( t) ) -g_t( x_t) \right)\Big)\\
=&\left(\dot{\delta}( t ) t^{2q} +2qt^{2q-1}\delta ( t )\right)( g_t( x( t ) ) -g_t( x_t ) )+\delta ( t) t^{2q}\left< \dot{x}( t ) ,\nabla g_t( x( t ) ) \right>\\
& -\frac{1}{2}at^{2q-p}\left( \frac{p}{ t}+\frac{\dot{\delta}( t )}{\delta ( t ) } \right) \left(\lVert x( t ) \rVert ^2- \lVert x_t \rVert ^2\right) .
\end{split}
\end{eqnarray}

 Secondly, for any $\lambda >0$ and $t\geq t_1$,
\begin{eqnarray}\label{4180}
\begin{split}
&\frac{d}{dt}\left(  \frac{1}{2}b\left( \alpha -b-qt^{q-1} \right) \lVert x( t ) -x_t \rVert ^2\right)\\
=&\frac{1}{2}bq(1-q) t^{q-2} \lVert  {x}( t )-x_t\rVert ^2+b(\alpha- b -qt^{q-1} ) \left< x( t ) -x_t, \dot{x}( t ) -\frac{d}{dt}x_t \right>\\
\le&\frac{1}{2}bq(1-q) t^{q-2} \lVert  {x}( t )-x_t\rVert ^2+b(\alpha- b -qt^{q-1} ) \left< x( t ) -x_t, \dot{x}( t ) \right>\\
&+\frac{1}{2}\lambda b( \alpha -b-qt^{q-1} ) t^{p-q}\left\lVert \frac{d}{dt}x_t \right\rVert ^2+\frac{1}{2\lambda} b( \alpha -b-qt^{q-1} ) t^{q-p}\lVert x( t ) -x_t \rVert ^2.
\end{split}
\end{eqnarray}

 Thirdly, \begin{eqnarray*}
{\begin{split}
&\frac{d}{dt}\left(\frac{1}{2}\left\lVert b( x( t ) -x_t ) +t^q ( \dot{x}( t ) +\beta  \nabla g( x( t ) )  ) \right\rVert ^2\right)\\
=&\Big< b( x( t ) -x_t ) +t^q\dot{x}( t ) +\beta t^q \nabla g( x( t ) ),( b+qt^{q-1} ) \dot{x}( t )-b\frac{d}{dt}x_t \\
 &  +t^q\ddot{x}( t ) +\beta qt^{q-1} \nabla g( x( t ) ) +\beta t^q \nabla^2 g( x( t ) )\dot{x}( t )\Big>.
\end{split}}
\end{eqnarray*}
This together with $(\ref{DS})$ and  $\nabla g_t(x(t))=\nabla g(x(t))+\frac{a}{\delta(t)t^p}x(t)$ yields
\begin{eqnarray}\label{4.12}
\small {\begin{split}
&\frac{d}{dt}\left(\frac{1}{2}\left\lVert b( x( t ) -x_t ) +t^q ( \dot{x}( t ) +\beta  \nabla g( x( t ) )  ) \right\rVert ^2\right)\\
=&\left< b( x( t ) -x_t ) +t^q\dot{x}( t ) +\beta t^q\nabla g( x( t ) ) ,( b-\alpha+qt^{q-1} ) \dot{x}( t ) -b\frac{d}{dt}x_t\right.\\
&\ \ \ \ \ \ \left. -\delta ( t ) t^q\nabla g_t( x( t ) ) +\beta qt^{q-1}\nabla g( x( t ) ) \right>\\
=&b( b-\alpha +qt^{q-1} ) \left< x( t ) -x_t,\dot{x}( t ) \right> +\underset{A}{\underbrace{( b-\alpha +qt^{q-1} ) t^q\Big(\lVert \dot{x}( t ) \rVert ^2+ \left<\beta  \nabla g( x( t ) ) ,\dot{x}( t ) \right>\Big )}}\\
&\underset{B}{\underbrace{-b^2\left< x( t ) -x_t,\frac{d}{dt}x_t \right> }}\  \underset{C}{\underbrace{-bt^q\left< \dot{x}( t ) ,\frac{d}{dt}x_t \right> }}\ \underset{D}{\underbrace{-b\beta t^q\left<  \nabla g( x( t ) ),\frac{d}{dt}x_t \right> }}\\
&\underset{E}{\underbrace{-b\delta ( t ) t^q\left< x( t ) -x_t,\nabla g_t( x( t ) ) \right> }}-\delta ( t ) t^{2q}\left< \dot{x}( t ) ,\nabla g_t( x( t ) ) \right> \\
&\underset{F}{\underbrace{-\beta \delta ( t ) t^{2q}\left<  \nabla g( x( t ) ) ,\nabla g_t( x( t ) ) \right> }}+\underset{G}{\underbrace{b\beta qt^{q-1}\left< x( t ) -x_t, \nabla g( x( t ) ) \right> }}\\
&+\underset{H}{\underbrace{\beta qt^{2q-1}\left< \dot{x}( t ) , \nabla g( x( t ) ) \right> }}+\beta ^2qt^{2q-1}\lVert  \nabla g( x( t ) ) \rVert ^2.
\end{split}}
\end{eqnarray}
Note that
\begin{eqnarray*}
\left\{
\begin{split}
A&=\frac{1}{2}\left( b-\alpha +qt^{q-1} \right)t^q \Big( \lVert \dot{x}( t ) \rVert ^2+\lVert \dot{x}( t ) +\beta \nabla g( x( t ) ) \rVert ^2-\beta ^2\lVert \nabla g( x( t ) )\rVert ^2 \Big),\\
B & \le \frac{1}{2\lambda}b^2t^{q-p}\lVert x( t ) -x_t \rVert ^2+ \frac{1}{2}\lambda b^2t^{p-q} \left\lVert \frac{d}{dt}x_t \right\rVert ^2,\\
C & \le \frac{1}{2} \lambda bt^q \lVert \dot{x}( t ) \rVert ^2+\frac{1}{2\lambda}bt^q\left\lVert \frac{d}{dt}x_t \right\rVert ^2,
 D  \le\frac{1}{2}bt^q \lVert \nabla g( x( t ) ) \rVert ^2+\frac{1}{2}\beta ^2bt^q \left\lVert \frac{d}{dt}x_t \right\rVert ^2,\\
E&\le -b\delta ( t ) t^q( g_t( x( t ) ) -g_t( x_t ) ) -\frac{1}{2}abt^{q-p}\lVert x( t ) -x_t \rVert ^2,\\
F&=-\frac{1}{2}\beta \delta ( t ) t^{2q}\left( \lVert \nabla g_t( x( t ) ) \rVert ^2+\lVert \nabla g( x( t ) ) \rVert ^2-\frac{a^2}{\delta^2 ( t )t^{2p}}\lVert x( t ) \rVert ^2 \right)\\
&\leq-\frac{1}{2}\beta \delta ( t ) t^{2q}\left(\lVert \nabla g( x( t ) ) \rVert ^2-\frac{a^2}{\delta^2 ( t )t^{2p}}\lVert x( t ) \rVert ^2 \right),\\
G & \le  \frac{1}{2}bqt^{q-2} \lVert x( t ) -x_t \rVert ^2+ \frac{1}{2}\beta ^2bqt^{q} \lVert \nabla g( x( t ) ) \rVert ^2,\\
H & \le \frac{1}{2}qt^{2q-1} \lVert \dot{x}( t ) \rVert ^2+\frac{1}{2}\beta ^2qt^{2q-1} \lVert \nabla g( x( t ) ) \rVert ^2.\\
\end{split}
\right.
\end{eqnarray*}
Here  the inequality in $E$ holds due to $(\ref{L})$.
Combining $(\ref{4.8})$ with  $(\ref{4.11})$, $(\ref{4180})$ and $(\ref{4.12})$,   we obtain that for all  $t\geq t_1,$
\begin{eqnarray}\label{4.19}
\small {\begin{split}
&\dot{ {E}}( t )
\\
\le& \left( \dot{\delta}( t ) t^{2q}+2q\delta ( t )t^{2q-1}-b\delta ( t ) t^q \right) \left( g_t( x( t ) ) -g_t( x_t ) \right)\\
&+\frac{1}{2}b\left(  q( 1-q ) t^{q-2}+\frac{1}{\lambda}  ( \alpha -b-qt^{q-1} )t^{q-p}+\frac{1}{\lambda} b t^{q-p}-at^{q-p}+ qt^{q-2} \right) \lVert x( t ) -x_t \rVert ^2\\
&+\frac{1}{2}at^{2q-p} \left( \frac{\beta a }{\delta ( t )t^{p}}- \frac{p}{t}-\frac{\dot{\delta}(t)}{\delta (t) } \right) \lVert x( t ) \rVert ^2+\frac{1}{2} at^{2q-p} \left( \frac{p}{t }+\frac{\dot{\delta}( t )}{\delta( t )  } \right) \lVert x_t \rVert ^2\\
&+\frac{1}{2}\left( \lambda b( \alpha -b-qt^{q-1} ) t^{p-q}+ \lambda b^2t^{p-q} +\frac{1}{\lambda}bt^q+ \beta ^2bt^q\right) \left\lVert \frac{d}{dt}x_t \right\rVert ^2\\
&+\frac{1}{2}\Big(  ( b-\alpha +qt^{q-1} ) t^q + \lambda bt^q+ qt^{2q-1} \Big) \lVert \dot{x}( t ) \rVert ^2\\
&+ \frac{1}{2}( b-\alpha +qt^{q-1} ) t^q \lVert \dot{x}( t ) +\beta \nabla g( x( t ) ) \rVert ^2\\
&+\frac{1}{2}\Big( -\beta ^2( b-\alpha +qt^{q-1} ) t^q+bt^q-\beta \delta ( t ) t^{2q}+\beta ^2bqt^{q }+ 3\beta ^2qt^{2q-1} \Big) \lVert\nabla g( x( t ) ) \rVert ^2.
\end{split}}
\end{eqnarray}

On the other hand, by $(\ref{4.8})$ and
$
\frac{1}{2}\lVert b( x( t ) -x_t ) +t^q\left( \dot{x}( t ) +\beta \nabla g( x( t ) ) \right) \rVert ^2\le b^2\lVert x( t ) -x_t \rVert ^2+t^{2q}\lVert \dot{x}( t ) +\beta \nabla g( x( t ) ) \rVert ^2,
$
it follows that
\begin{eqnarray}\label{4.20}
\begin{split}
 {E}( t )& \le \delta ( t ) t^{2q}\left( g_t( x( t ) ) -g_t( x_t ) \right) +\left(\frac{1}{2}b( \alpha -b-qt^{q-1} )+b^2 \right ) \lVert x( t ) -x_t \rVert ^2\\
&~~~~+t^{2q}\lVert \dot{x}( t )+ \beta   \nabla g( x( t ) ) \rVert ^2.
\end{split}
\end{eqnarray}
Let $r:=\max \left\{ q,p-q \right\}$ and $K$ be an arbitrarily positive constant. By $(\ref{4.19})$  and $(\ref{4.20})$, we have for all $t\geq t_1,$
\begin{eqnarray}\label{4.22}
 {\begin{split}
&\dot{ {E}}( t )+\frac{K}{t^r} {E}( t )
\\
\le& \Big( \dot{\delta}( t ) t^{2q}+2q\delta ( t )t^{2q-1}-b\delta ( t ) t^q +K\delta ( t ) t^{2q-r}\Big) \left( g_t( x( t ) ) -g_t( x_t ) \right)\\
&+\frac{1}{2}b\Big(  q( 1-q ) t^{q-2}+\frac{1}{\lambda}  ( \alpha -b-qt^{q-1} )t^{q-p}+\frac{1}{\lambda} b t^{q-p}\\
&\quad\quad\quad -at^{q-p}+ qt^{q-2}+K ( \alpha -b-qt^{q-1} ) t^{-r}+2Kb t^{-r} \Big) \lVert x( t ) -x_t \rVert ^2\\
&+\frac{1}{2}at^{2q-p} \left( \frac{\beta a }{\delta ( t )t^{p}}- \frac{p}{t}-\frac{\dot{\delta}(t)}{\delta (t) } \right) \lVert x( t ) \rVert ^2+\frac{1}{2} at^{2q-p} \left( \frac{p}{t }+\frac{\dot{\delta}( t )}{\delta( t )  } \right) \lVert x_t \rVert ^2\\
&+\frac{1}{2}\left( \lambda b( \alpha -b-qt^{q-1} ) t^{p-q}+ \lambda b^2t^{p-q} +\frac{1}{\lambda}bt^q+ \beta ^2bt^q\right) \left\lVert \frac{d}{dt}x_t \right\rVert ^2\\
&+\frac{1}{2}\left(  ( b-\alpha +qt^{q-1} ) t^q + \lambda bt^q+ qt^{2q-1} \right) \lVert \dot{x}( t ) \rVert ^2\\
&+ \frac{1}{2}\Big(( b-\alpha +qt^{q-1} ) t^q +2Kt^{2q-r}\Big)\lVert \dot{x}( t ) +\beta \nabla g( x( t ) ) \rVert ^2\\
&+\frac{1}{2}\Big( -\beta ^2( b-\alpha +qt^{q-1} ) t^q+bt^q-\beta \delta ( t ) t^{2q}+\beta ^2bqt^{q }+ 3\beta ^2qt^{2q-1} \Big) \lVert\nabla g( x( t ) ) \rVert ^2.
\end{split}}
\end{eqnarray}
Further, together with $(\ref{4.3})$, $ t\dot{\delta}( t )<K_1\delta ( t )$ and $\lVert x_t \rVert \le \lVert \bar{x}^* \rVert$, we  deduce that there exists $t_2\ge t_1$ such that for all $t\geq t_2,$
\begin{eqnarray}\label{4.2}
 \frac{1}{2}at^{2q-p}  \left( \frac{p}{t }+\frac{\dot{\delta}( t )}{\delta ( t ) } \right) \lVert x_t \rVert ^2\le \frac{1}{2}at^{2q-p} \frac{p+K_1}{t }\lVert x_t \rVert ^2\leq  \frac{1}{2}a(p+K_1) t^{2q-p-1} \lVert \bar{x}^* \rVert ^2
\end{eqnarray}
and
\begin{eqnarray}\label{4.23}
\left\lVert \frac{d}{dt}x_t \right\rVert ^2\le \left( \frac{p}{t}+\frac{K_1}{t} \right) ^2\lVert x_t \rVert ^2\le \frac{K_2}{t^2}\lVert \bar{x}^* \rVert ^2,
\end{eqnarray}
where $K_2=(p+K_1)^2$.
Combining $(\ref{4.22})$, $(\ref{4.2})$ and $(\ref{4.23})$, we obtain that for all  $t\geq t_2,$
\begin{eqnarray}\label{4.24}
 {\begin{split}
&\dot{ {E}}( t )+\frac{K}{t^r} {E}( t )
\\
\le& \Big( \dot{\delta}( t ) t^{2q}+2q\delta ( t )t^{2q-1}-b\delta ( t ) t^q +K\delta ( t ) t^{2q-r}\Big) \left( g_t( x( t ) ) -g_t( x_t ) \right)\\
&+\frac{1}{2}b\Big(  q( 1-q ) t^{q-2}+\frac{1}{\lambda}  ( \alpha -b-qt^{q-1} )t^{q-p}+\frac{1}{\lambda} b t^{q-p}\\
&\quad\quad\quad -at^{q-p}+ qt^{q-2}+K ( \alpha -b-qt^{q-1} ) t^{-r}+2Kb t^{-r} \Big) \lVert x( t ) -x_t \rVert ^2\\
&+\frac{1}{2}at^{2q-p} \left( \frac{\beta a }{\delta ( t )t^{p}}- \frac{p}{t}-\frac{\dot{\delta}(t)}{\delta (t) } \right) \lVert x( t ) \rVert ^2\\
&+\frac{1}{2}\left(  ( b-\alpha +qt^{q-1} ) t^q + \lambda bt^q+ qt^{2q-1} \right) \lVert \dot{x}( t ) \rVert ^2\\
&+ \frac{1}{2}\Big(( b-\alpha +qt^{q-1} ) t^q +2Kt^{2q-r}\Big)\lVert \dot{x}( t ) +\beta \nabla g( x( t ) ) \rVert ^2\\
&+\frac{1}{2}\Big( -\beta ^2( b-\alpha +qt^{q-1} ) t^q+bt^q-\beta \delta ( t ) t^{2q}+\beta ^2bqt^{q }+ 3\beta ^2qt^{2q-1} \Big) \lVert\nabla g( x( t ) ) \rVert ^2\\
&+ \frac{1}{2}\Big( a(p+K_1)t^{2q-p-1} + K_2\lambda b( \alpha -b-qt^{q-1} ) t^{p-q-2 } \\
&\quad\quad\quad +K_2\lambda b^2t^{p-q-2} +\frac{1}{ \lambda}K_2bt^{q-2}+ K_2\beta ^2bt^{q-2 }\Big) \lVert \bar{x}^* \rVert ^2.
\end{split}}
\end{eqnarray}
Clearly, from $t^p\dot{\delta }( t )+pt^{p-1}\delta(t)-a\beta\geq 0$, we  deduce  that
\begin{eqnarray}\label{4.25101}
\frac{1}{2}at^{2q-p} \left( \frac{\beta a }{\delta ( t )t^{p}}- \frac{p}{t}-\frac{\dot{\delta}(t)}{\delta (t) } \right)\leq0.
\end{eqnarray}
Set $\frac{\alpha}{a}<\lambda <\frac{\alpha}{b}-1$ and
$
0<K<\min \left\{ b,\frac{\alpha -b}{2},\frac{a-\frac{1}{\lambda}\alpha}{b+\alpha} \right\}$.
Together with  $(\ref{4.25101})$,  we   deduce  from $(\ref{4.24})$ that   there exists $t_3\ge t_2$ such that for all $ t\ge t_3$,
\begin{eqnarray}\label{4.25}
\begin{split}
 \dot{E}( t ) +\frac{K}{t^r}E( t )
\le &  \frac{1}{2}\Big(  a(p+K_1)t^{2q-p-1} + K_2\lambda b( \alpha -b-qt^{q-1} ) t^{p-q-2 }  \\
&+K_2\lambda b^2t^{p-q-2} +\frac{1}{ \lambda}K_2bt^{q-2}+ K_2\beta ^2bt^{q-2 }\Big)   \lVert \bar{x}^* \rVert ^2.
\end{split}
\end{eqnarray}

Now,  we consider the following two cases:

$\mathbf{Case~ I}$:   $q<p<\frac{1}{2}(3q+1)$. In this case,    we have $ {2q-p-1}> {p-q-2}$ and $ 2q-p-1 >  q-2 $. Then, there exist $C'>0$ and $t_4\ge t_3$ such that for all $t\ge t_4$,
\begin{eqnarray*}
\begin{split}
 & \frac{1}{2}\Big(  a(p+K_1)t^{2q-p-1} + K_2\lambda b( \alpha -b-qt^{q-1} ) t^{p-q-2 }  \\
&\quad\quad\quad +K_2\lambda b^2t^{p-q-2} +\frac{1}{ \lambda}K_2bt^{q-2}+ K_2\beta ^2bt^{q-2 }\Big) \lVert \bar{x}^* \rVert ^2\le C't^{2q-p-1}.
\end{split}
\end{eqnarray*}
Hence, $(\ref{4.25})$ becomes
$
\dot{E}( t ) +\frac{K}{t^r}E(t)\le C't^{2q-p-1} $   for all $t\ge t_4.$
Using a similar argument as in \cite[Theorem 10]{qp2023},   there exists $t_5\ge t_4$ and $C_2>0$ such that
$
E( t ) \le C_2t^{2q-p-1+r}
$ for all $t \ge t_5$.
Consequently, as $ t\rightarrow +\infty , $ it holds that
$
\lVert \dot{x}( t) +\beta \nabla g( x( t  ) \rVert ^2=\mathcal{O}\left( \frac{1}{t^{p+1-r}} \right)$   and
 \begin{eqnarray}\label{4.7xxx}
 g_t( x( t ) ) -g_t( x_t ) =\mathcal{O}\left( \frac{1}{\delta ( t ) t^{p+1-r}} \right).
 \end{eqnarray}
From $g_t$ is $\frac{a}{\delta ( t ) t^p}$-strongly convex, we have
$
g_t(\bar{x}^*) -g_t( x_t ) \ge \left< \nabla g_t( x_t ) ,\bar{x}^*-x_t \right> +\frac{a}{2\delta ( t ) t^p}\lVert \bar{x}^*-x_t \rVert ^2 =\frac{a}{2\delta ( t ) t^p}\lVert \bar{x}^*-x_t \rVert ^2.
$
Thus,
\begin{eqnarray}\label{4.7}
\begin{split}
&g( x(t) ) -g(\bar{x}^* )\\
=&( g_t( x(t) ) -g_t( x_t ) ) +( g_t( x_t ) -g_t( \bar{x}^*) ) +\frac{a}{2\delta ( t ) t^p}\left( \lVert \bar{x}^* \rVert^2 -\lVert x(t) \rVert ^2\right)\\
\le& g_t( x(t) ) -g_t( x_t ) +\frac{a}{2\delta( t ) t^p}\lVert \bar{x}^* \rVert ^2.
\end{split}
\end{eqnarray}
Combining  $(\ref{4.7xxx})$,  $(\ref{4.7})$  and $r<1$, we get
$
g( x( t ) ) -g( \bar{x}^*)=\mathcal{O}\left( \frac{1}{\delta ( t ) t^p} \right), $  as $ t\rightarrow +\infty .
$
Moreover, it is easy to show that
$
\lVert x( t ) -x_t \rVert =\mathcal{O}\left( \frac{1}{t^{\frac{1-q}{2}}} \right), $ as  $t\rightarrow +\infty .
$
This  together with $\lim _{t\rightarrow +\infty}x_t=\bar{x}^*$  yields
$
\underset{t\rightarrow +\infty}{\lim}\lVert x(t)-\bar{x}^* \rVert =0.
$

$\mathbf{Case~ II}$: $ \frac{1}{2}(3q+1)\leq p< q+1$. In this case,   $p-q-2\geq 2q-p-1> q-2$. Then,   there exists $C>0$ and $t_4\ge t_3$ such that for all $t\geq t_4$,
\begin{eqnarray*}
\begin{split}
& \frac{1}{2}\Big(  a(p+K_1)t^{2q-p-1} + K_2\lambda b( \alpha -b-qt^{q-1} ) t^{p-q-2 }  \\
&\quad\quad\quad +K_2\lambda b^2t^{p-q-2} +\frac{1}{ \lambda}K_2bt^{q-2}+ K_2\beta ^2bt^{q-2 }\Big) \lVert \bar{x}^* \rVert ^2\le Ct^{p-q-2}.
\end{split}
\end{eqnarray*}
Thus, $(\ref{4.25})$ becomes
$
\dot{E}( t ) +\frac{K}{t^r}E( t ) \le Ct^{p-q-2} $  for all $ t\ge t_4.$
Moreover, since $0<q<1$ and $\frac{1}{2}(3q+1)\leq p<q+1$, we get $r=\max \left\{ q,p-q \right\}=p-q<1$.
Then, using a similar argument as in \cite[Theorem 10]{qp2023},  there exists $t_5\ge t_4$ and $C_1>0$ such that for all $t \ge t_5$,
\begin{eqnarray}\label{4.30}
E( t ) \le C_1t^{p-q-2+r} .
\end{eqnarray}
This  together with $r=p-q$  yields that
$
\lVert x( t ) -x_t \rVert =\mathcal{O}\left( \frac{1}{t^{q-p+1}} \right), $  as $t\rightarrow +\infty.
$
Combining this with $\lim _{t\rightarrow +\infty}x_t=\bar{x}^*$,  we obtain
$
\underset{t\rightarrow +\infty}{\lim}\lVert x(t)-\bar{x}^* \rVert =0.
$
Further, it follows from $(\ref{4.30})$ that
$
 \lVert \dot{x}( t ) +\beta \nabla g( x( t ) ) \rVert =\mathcal{O}\left( \frac{1}{t^{2q-p+1}} \right)
$
and
\begin{eqnarray}\label{o1}
g_t( x( t ) ) -g_t( x_t ) =\mathcal{O}\left( \frac{1}{\delta ( t ) t^{4q-2p+2}} \right), \ as\ \ t\rightarrow +\infty.
\end{eqnarray}

Now, according to $(\ref{4.7})$ and $(\ref{o1})$, if $4q-2p+2\geq p,$ that is $\frac{1}{2}(3q+1)\le p\le \frac{1}{3}(4q+2)$, we have
$
g( x( t ) ) -g( \bar{x}^*)=\mathcal{O}\left( \frac{1}{\delta ( t ) t^p} \right), $  as $t\rightarrow +\infty .
$
Conversely, if $4q-2p+2< p,$ that is $\frac{1}{3}(4q+2)<p<q+1$, we can deduce from   $(\ref{4.7})$ and $(\ref{o1})$ that
$
g( x( t ) ) -g( \bar{x}^*)=\mathcal{O}\left( \frac{1}{\delta ( t ) t^{4q-2p+2}} \right),$  as $t\rightarrow +\infty.
$
The proof is complete.\qed
\end{proof}

\begin{remark}
\begin{enumerate}
\item[{\rm (i)}]
Theorem \ref{the5.1} shows that we can simultaneously obtain a fast convergence rate for
$g(x(t))$ and the strong convergence of the trajectory $x(t)$ towards the minimum norm
solution of   problem (\ref{g}). However, in previous studies \cite{TRIGS1,zhong}, these results were achieved separately, relying on different dynamics associated with varying parameter settings.

\item[{\rm (ii)}] Clearly, when   $\beta\equiv0$ and $\delta(t)\equiv1$, that is when the time scaling and Hessian-driven damping terms are removed from   system $(\ref{DS})$, Theorem \ref{the5.1} reduces to  \cite[Theorem 10]{qp2023}.

\end{enumerate}
\end{remark}
\section{Convergence  of the associated  algorithm}
In this section, we propose an inertial proximal gradient algorithm obtained by temporal discretization of the system (\ref{DS})  that also exhibit corresponding convergence rates, which is compatible with the results in the continuous case.

  Take a fixed  time step size $h>0$, and set $t_k=kh$,  $x_{k+1}=x( kh ) $, $\delta _k=\delta ( kh )
 $.  Then, for all $k\ge 1$, the implicit time  discretization of the system (\ref{DS}) gives
\begin{eqnarray}\label{ipa110sun}
\begin{split}
&\frac{x_{k+1}-2x_k+x_{k-1}}{h^2}+\frac{\alpha}{( kh ) ^q}\frac{x_{k+1}-x_k}{h}+\frac{\beta ( \nabla g( x_{k+1} ) -\nabla g( x_k ) )}{h}\\
&+\delta _k\nabla g( x_{k+1} ) +\frac{a}{( kh ) ^p}x_{k+1}=0.
\end{split}
\end{eqnarray}
This follow that
\begin{eqnarray*}
\begin{split}
&\left( ( kh ) ^{q+p}+\alpha h( kh ) ^p+h^2a( kh ) ^q \right) ( x_{k+1}-x_k ) +( h\beta+h^2\delta _k ) ( kh ) ^{q+p}\nabla g( x_{k+1} )\\
=&( kh ) ^{q+p}\left( x_k-x_{k-1}+h\beta\nabla g( x_k ) \right) -h^2a( kh ) ^qx_k.
\end{split}
\end{eqnarray*}
Set $d _k:=\frac{( kh ) ^{q+p}}{( kh ) ^{q+p}+\alpha h( kh ) ^p+ah^2( kh ) ^q}$, $\rho _k:=d _k( h\beta+h^2\delta _k )$ and $a_k:=a( kh ) ^{-p}$. Then,
\begin{eqnarray}\label{ipa110}
( I+\rho _k\nabla g ) ( x_{k+1} ) =x_k+d _k\left( x_k-x_{k-1}+h\beta\nabla g( x_k ) \right) -h^2d _ka_kx_k.
\end{eqnarray}

Now, solving (\ref{ipa110}) with respective to $x_{k+1}$, we obtain the following inertial proximal gradient algorithm for solving the problem (\ref{g}):
\begin{algorithm}[H]
\caption{(IPGA): Inertial Proximal Gradient Algorithm.}
\textbf{Initialization:}  Let $x_0, x_1 \in \mathcal{X}$,  $\alpha ,\ a,\ \beta, h ,\ q,\ p>0$.

\begin{algorithmic}
\For{$k = 1, 2, \dots$}

    \State \textbf{Step 1:} Compute $d _k=\frac{( kh ) ^{q+p}}{( kh ) ^{q+p}+\alpha h( kh ) ^p+ah^2( kh ) ^q}$;

    \State \textbf{Step 2:} Choose $\delta _k>0$ and compute $y_k=x_k+d _k( x_k-x_{k-1}+\beta h\nabla g( x_k ) )$ and   $\rho _k=d _k( \beta h+h^2\delta _k )$;

    \State \textbf{Step 3:} Compute $a_k=a( kh ) ^{-p}$ and update the variable $x_{k+1}=\textup{prox}_{\rho _kg}( y_k-h^2d _ka_kx_k )$;

    \If{a stopping condition is satisfied}
        \State \textbf{Return} $(d _k, y_k, \rho _k, x_{k+1})$
    \EndIf
\EndFor
\end{algorithmic}
\end{algorithm}
\begin{remark}
\begin{enumerate}
\item[{\rm (i)}] In  $\mbox{(IPGA)}$, $d_k$ is the inertial parameter, $a_k$ is the Tikhonov regularization parameter and $\rho _k$ is the variable step size parameter. Further, note that   $\mbox{(IPGA)}$ involves both the gradient     and the proximal operator  relative to the function $f$. Consequently, we can extend  $\mbox{(IPGA)}$   to the case of  non-smooth convex functions which is beyond the scope  of this paper and deserves further research.
\item[{\rm (ii)}]If $p=q=0$,   $\mbox{(IPGA)}$ exhibits a structure similar to that of the inertial proximal gradient algorithm $\mbox{(IPATTH)}$ introduced in \cite{2024AC}. On the other hand, if $\beta\equiv0$, i.e., in the case without Hessian driven damping, $\mbox{(IPGA)}$ has a structure similar to that of $\mbox{(PIATR)}$ introduced in \cite{2025LC}. Further, if $a\equiv0$, i.e., in the case without Tikhonov regularization, $\mbox{(IPGA)}$ has a structure similar to that of the inertial proximal algorithms  introduced in \cite{amp222,aop23}.
\item[{\rm (iii)}]
Let $q_k:=( kh ) ^q$.  $\mbox{(IPGA)}$ can    equivalently be written as
\begin{eqnarray}\label{5.3}
\begin{split}
q_k( x_{k+1}-2x_k+x_{k-1} ) =&- \alpha h( x_{k+1}-x_k ) -h( \beta +h\delta _k )q_k\nabla g( x_{k+1} ) \\
&+h\beta q_k\nabla g( x_k ) -h^2q_ka_kx_{k+1},
\end{split}
\end{eqnarray}
 which directly follows from   (\ref{ipa110sun}).
\end{enumerate}
\end{remark}

Throughout this section, we assume that $\{\delta_k\}_{k\in\mathbb{ N}}$ is a non-decreasing sequence and $\lim _{k\rightarrow +\infty}\delta _k=+\infty$.

In the case that $q+1<p\le2$, we first give the fast convergence rates of the function values which is
compatible with the result  obtained in Theorem \ref{the4.3}.
\begin{theorem}\label{the6.1}
Let   $q+1<p\leq2$,   $a>q( 1-q ) $ for $p=2$, and $\alpha >h^{q-1}$,  Let  $\{x_k \} _{k\in \mathbb{N}}$ be  a sequence  generated by \textup{(IPGA)},
  Then, as $k\rightarrow +\infty$,    it holds that
$
g( x_k ) -g( x^* )  =\mathcal{O}\left( \frac{1}{k^{2q}\delta _k} \right)$ and $ \lVert x_k-x_{k-1}+h\beta\nabla g( x_k) \rVert =\mathcal{O}\left( \frac{1}{k^{q}} \right)$,  where  $x^*\in \textup{argmin}g$.
\end{theorem}

\begin{proof} Let $0<\lambda \leq\alpha h-h^q$. We consider the following energy sequence:
\begin{eqnarray*}
\begin{split}
\mathcal{E}_k:=&\mu_k( g( x_k ) -g( x^* ) ) +\frac{1}{2}\lVert \lambda ( x_k-x^* ) +q_k( x_k-x_{k-1}+h\beta\nabla g( x_k ) )\rVert ^2\\
&+\frac{1}{2}\lambda \ell _k\lVert x_k-x^* \rVert ^2+\frac{1}{2}h^2q_ka_k( q_{k+1}+\ell _k ) \lVert x_k \rVert ^2,
\end{split}
\end{eqnarray*}
where
$\mu_k:=h( q_{k+1}+\ell _k ) \mathcal{B}_k+h\beta \ell _kq_{k+1}  ,$
$\mathcal{B}_k:= \beta(q_k-  q_{k+1})+hq_k\delta _k   $ and
$\ell _k:=q_k -q_{k+1}+\alpha h-\lambda$.

Clearly, from $\{q_k-q_{k+1}\}_{k\in  \mathbb{N}}$ is a  non-decreasing sequence, we have $q_k-q_{k+1}\geq q_0-q_1=-h^q$, $\forall k\in  \mathbb{N} $. This together with $0<\lambda \leq\alpha h-h^q$ yields   $\ell _k\geq0$  for all $k\in  \mathbb{N} $.  Moreover, by  $\lim _{k\rightarrow +\infty}\delta _k=+\infty$,  there exists $k_0\in \mathbb{N}$ such that  $\mathcal{B}_k\geq0$ for all $k\ge k_0$. Consequently,  $\mu_k\geq0$  for all $k\ge k_0 $.  Then, $\mathcal{E}_k\ge0$  for all $k\ge k_0 $.

Now, we  prove that  $\mathcal{E}_k$ is a bounded sequence. Indeed, let
$\vartheta_k:=\lambda ( x_k-x^* ) +q_k( x_k-x_{k-1}+h\beta\nabla g( x_k ) ).$ Then,
\begin{eqnarray}\label{5.6}
\begin{split}
&\mathcal{E}_{k+1}-\mathcal{E}_k\\=&( \mu_{k+1}-\mu_k ) \left( g( x_{k+1} ) -g( x^* ) \right) +\mu_k( g( x_{k+1} ) -g( x_k ) )\\
&+\frac{1}{2}\left( \lVert \vartheta_{k+1} \rVert ^2-\lVert \vartheta_k \rVert ^2 \right) +\frac{1}{2}\lambda \left( \ell_{k+1}\lVert x_{k+1}-x^* \rVert ^2-\ell_k\lVert x_k-x^* \rVert ^2 \right)\\
&+\frac{1}{2}h^2\Big(q_{k+1}a_{k+1}\left( q_{k+2}+\ell_{k+1} \right)
 \lVert x_{k+1} \rVert ^2- q_ka_k( q_{k+1}+\ell _k )\lVert x_k \rVert ^2 \Big).\\
\end{split}
\end{eqnarray}
In the following, we   evaluate the third term in the right-hand side of (\ref{5.6}). Indeed,
\begin{eqnarray*}\label{5.7}
\begin{split}
&\vartheta_{k+1}-\vartheta_k\\=&\lambda ( x_{k+1}-x_k ) +q_{k+1}( x_{k+1}-x_k ) -q_k( x_k-x_{k-1} ) \\
&+h\beta q_{k+1}\nabla g( x_{k+1} ) -h\beta q_k\nabla g( x_k )\\
=&\lambda ( x_{k+1}-x_k ) +( q_{k+1}-q_k ) ( x_{k+1}-x_k ) +q_k( x_{k+1}-2x_k+x_{k-1} ) \\
&+h\beta q_{k+1}\nabla g( x_{k+1}) -h\beta q_k\nabla g( x_k )\\
=&( \lambda +q_{k+1}-q_k-\alpha h ) ( x_{k+1}-x_k ) +h(\beta q_{k+1}- q_k( \beta +h\delta _k ) ) \nabla g( x_{k+1} )
\\&-h^2q_ka_kx_{k+1}\\
=&-\ell _k( x_{k+1}-x_k ) -h\mathcal{B}_k\nabla g( x_{k+1} ) -h^2q_ka_kx_{k+1}.
\end{split}
\end{eqnarray*}
where the third equality holds due to (\ref{5.3}). Thus,
\begin{eqnarray*}\label{5.9}
{\begin{split}
&\frac{1}{2}\lVert \vartheta_{k+1} \rVert ^2-\frac{1}{2}\lVert \vartheta_k \rVert ^2\\
=&\left< \vartheta_{k+1}-\vartheta_k,\vartheta_{k+1} \right> -\frac{1}{2}\lVert \vartheta_{k+1}-\vartheta_k \rVert ^2\\
=&\left< -\ell_k( x_{k+1}-x_k ) -h\mathcal{B}_k\nabla g( x_{k+1} ) -h^2q_ka_kx_{k+1},\lambda ( x_{k+1}-x^* ) +q_{k+1}( x_{k+1}\right.\\
 &\left.-x_k+h\beta\nabla g( x_{k+1}) ) \right>-\frac{1}{2}\lVert  \ell_k( x_{k+1}-x_k ) +h\mathcal{B}_k\nabla g( x_{k+1} ) +h^2q_ka_kx_{k+1} \rVert ^2\\
=&\underset{\overline{A}}{\underbrace{-\lambda \ell _k\left< x_{k+1}-x_k,x_{k+1}-x^* \right>}} -h\lambda \mathcal{B}_k\left<\nabla g( x_{k+1},x_{k+1}-x^* ) \right>\\
&\underset{\overline{B}}{\underbrace{-\lambda h^2q_ka_k\left<x_{k+1}, x_{k+1}-x^*\right>}}-\ell_kq_{k+1}\lVert x_{k+1}-x_k \rVert ^2\\
&-hq_{k+1}\mathcal{{B}}_k\left<\nabla g( x_{k+1}) , x_{k+1}-x_k \right>\underset{\overline{C}}{\underbrace{-h^2q_ka_kq_{k+1}\left<x_{k+1}, x_{k+1}-x_k\right>}}\\
&-h\beta \ell_kq_{k+1}\left< x_{k+1}-x_k,\nabla g( x_{k+1} ) \right> -h^2\beta \mathcal{B}_kq_{k+1}\lVert \nabla g( x_{k+1} ) \rVert ^2\\
&\underset{\overline{D}}{\underbrace{-h^3\beta q_ka_kq_{k+1}\left< x_{k+1},\nabla g( x_{k+1} ) \right>}}-\frac{1}{2}\ell _k^2\lVert x_{k+1}-x_k \rVert ^2\\
&-h\ell _k\mathcal{B}_k\left< x_{k+1}-x_k,\nabla g( x_{k+1} )   \right>\underset{\overline{E}}{\underbrace{-h^2q_ka_k\ell _k\left< x_{k+1}-x_k,x_{k+1} \right>}}\\
&-\frac{1}{2}\lVert h\mathcal{B}_k\nabla g( x_{k+1} ) +h^2q_ka_kx_{k+1} \rVert ^2.
\end{split}}
\end{eqnarray*}
Note that
\begin{eqnarray*}\label{5.10}
~~~~\left\{
\begin{split}
\overline{A} =&\frac{1}{2}\lambda \ell_k\left(\lVert x_k-x^* \rVert ^2- \lVert x_{k+1}-x^* \rVert ^2-\lVert x_{k+1}-x_k \rVert ^2 \right),&\\
\overline{B}=&\frac{1}{2}\lambda h^2q_ka_k\left( \lVert x^* \rVert ^2-\lVert x_{k+1}\rVert ^2-\lVert x_{k+1}-x^* \rVert ^2 \right),\\
\overline{C}+\overline{E}=&\frac{1}{2}h^2q_ka_k( \ell_k+q_{k+1} ) \left( \lVert x_k \rVert ^2-\lVert x_{k+1} \rVert ^2 -\lVert x_{k+1}-x_k \rVert ^2\right),\\
\overline{D}\le& \frac{1}{2} h^4q^2_ka_k^2 \lVert x_{k+1} \rVert ^2+\frac{1}{2} {h}^2{\beta}^2{ q^2_{k+1}}\lVert \nabla g( x_{k+1} ) \rVert ^2.
\end{split}
\right.
\end{eqnarray*}
Then, together with $\mu_k=h( q_{k+1}+\ell _k ) \mathcal{B}_k+h\beta \ell_kq_{k+1} $, we obtain
\begin{eqnarray*}\label{5.11}
\begin{split}
&\frac{1}{2}\lVert \vartheta_{k+1} \rVert ^2-\frac{1}{2}\lVert \vartheta_k \rVert ^2\\
\le &-\frac{1}{2}\left( \lambda \ell_k+2\ell_kq_{k+1}+h^2q_ka_k( \ell_k+q_{k+1} )+\ell_k^2\right ) \lVert x_{k+1}-x_k \rVert ^2\\
&-h\lambda \mathcal{B}_k\left< \nabla g( x_{k+1} ) ,x_{k+1}-x^* \right> -\mu_k\left< x_{k+1}-x_k,\nabla g( x_{k+1} ) \right>\\
&-\frac{1}{2}h^2\beta q_{k+1}\left( 2\mathcal{B}_k- {\beta}  {q _{k+1}} \right) \lVert \nabla g( x_{k+1} ) \rVert ^2+\frac{1}{2}\lambda \ell_k  \lVert x_k-x^* \rVert ^2\\
&-\frac{1}{2}\lambda (h^2q_ka_k+\ell _k)\lVert x_{k+1}-x^* \rVert ^2+\frac{1}{2}h^2q_ka_k( \ell _k +q_{k+1}) \lVert x_k \rVert ^2\\
&-\frac{1}{2}h^2q_ka_k( \lambda +\ell_k+q_{k+1}-h^2q_ka_k ) \lVert x_{k+1} \rVert ^2+\frac{1}{2}\lambda h^2q_ka_k\lVert x^* \rVert ^2.
\end{split}
\end{eqnarray*}
This together with (\ref{5.6}) yields  for any $k\geq k_0,$
\begin{eqnarray}\label{5.12}
\small{\begin{split}
&\mathcal{E}_{k+1}-\mathcal{E}_k\\
\le &( \mu_{k+1}-\mu_k ) ( g( x_{k+1} ) -g( x^* ) ) +\mu_k( g( x_{k+1} ) -g( x_k ) )\\
&-\frac{1}{2}\left(\lambda \ell_k+2\ell_kq_{k+1}+\ell_k^2+h^2q_ka_k( \ell_k+q_{k+1} ) \right) \lVert x_{k+1}-x_k \rVert ^2\\
&-h\lambda \mathcal{B}_k\left< \nabla g( x_{k+1} ) ,x_{k+1}-x^* \right>-\mu_k\left< x_{k+1}-x_k,\nabla g( x_{k+1} ) \right>\\
&-\frac{1}{2}h^2\beta q_{k+1}\left( 2\mathcal{B}_k- {\beta} {q _{k+1}} \right) \lVert \nabla g( x_{k+1} ) \rVert ^2-\frac{1}{2}\lambda ( h^2q_ka_k+\ell_k-\ell _{k+1} ) \lVert x_{k+1}-x^* \rVert ^2\\
&-\frac{1}{2}h^2\Big(q_ka_k( \lambda +\ell_k+q_{k+1} - h^2q_ka_k) - q_{k+1}a_{k+1}( q_{k+2}+\ell_{k+1} ) \Big) \lVert x_{k+1} \rVert ^2\\
&+\frac{1}{2}\lambda h^2q_ka_k\lVert x^* \rVert ^2\\
\le&( \mu_{k+1}-\mu_k-h\lambda \mathcal{B}_k ) ( g( x_{k+1} ) -g( x^* ) ) \\
&-\underset{ m_k}{\underbrace{\frac{1}{2}h^2 \beta q_{k+1}\left(2 \mathcal{B}_k- {\beta}  {q _{k+1}} \right)}} \lVert \nabla g( x_{k+1} ) \rVert ^2-\underset{n_k}{\underbrace{\frac{1}{2}\lambda (h^2q_ka_k+\ell _k-\ell_{k+1} )}} \lVert x_{k+1}-x^* \rVert ^2\\
&-\underset{\zeta_k}{\underbrace{\frac{1}{2}h^2\Big(q_ka_k( \lambda +\ell_k+q_{k+1}-h^2q_ka_k ) - q_{k+1}a_{k+1}(  q_{k+2}+\ell_{k+1} ) \Big)}} \lVert x_{k+1} \rVert ^2\\
&+\frac{1}{2}\lambda h^2q_ka_k\lVert x^* \rVert ^2,
\end{split}}
\end{eqnarray}
where the second inequality holds due to $\mu_k\geq 0,$ $\mathcal{B}_k\geq0,$ $\left< \nabla g( x_{k+1} ) ,x_{k+1}-x_k \right> \ge g( x_{k+1} ) -g( x_k )$    and  $\left< \nabla g( x_{k+1} ) ,x_{k+1}-x^* \right> \ge g( x_{k+1} ) -g( x^* ) $.

 Now, according to Appendix A(i), there exists $k_1\ge k_0$ such that  $m_k\ge 0$, $n_k\ge0$ and $\zeta_k\ge0$ for all  $k\ge k_1$. Further, from $\mathcal{B}_k=\mathcal{O}\left( k^{q}\delta_k\right)$ and $\mu_k=\mathcal{O}\left( k^{2q}\delta_k\right)$  (see Appendix A(ii)),  we  deduce that  there exists $k_2\ge k_1$ such that  $ \mu_{k+1}-\mu_k-h\lambda \mathcal{B}_k\leq0$ for all $  k\ge k_2$.
  Then, it follows from $(\ref{5.12})$   that
$
\mathcal{E}_{k+1}-\mathcal{E}_k\le \frac{1}{2}\lambda h^2q_ka_k\lVert x^* \rVert ^2=\frac{1}{2}a\lambda h^{q-p+2}\lVert x^* \rVert ^2k^{q-p}$ for all $k\ge k_2.$
Summing  it from $k= k_2$ to $n-1\ge k_2$, we obtain
$
\mathcal{E}_{n}\le  E_{k_2}+\sum^{n-1}_{k= k_2}\frac{1}{2}a\lambda h^{q-p+2}\lVert x^* \rVert ^2{k^{q-p}} $   for all $k\ge k_2.$
From $q+1<p$, we have $\sum^{n-1}_{k= k_2}\frac{1}{2}a\lambda h^{q-p+2}\lVert x^* \rVert ^2{k^{q-p}}<+\infty.$ Then, there exists $C_0>0$ such that   $\mathcal{E}_{n}\le C_0 $  for all $k\ge k_2.$
This follows that
$
 \mu_n( g( x_n ) -g( x^* )) \le C_0,$
 $\frac{1}{2}\lVert \vartheta_n \rVert ^2\le C_0$ and $
 \frac{1}{2}\lambda\ell_n\lVert x_n-x^* \rVert \leq C_0.$
Moreover, according to Appendix A (ii), we have $\mu_n=\mathcal{O}\left( n^{2q}\delta_n\right)$ and $\ell_n=\mathcal{O} (n^0  )$, as $n\rightarrow+\infty$.
Then, we can deduce that
$
g( x_n ) -g( x^* )  =\mathcal{O}\left( \frac{1}{n^{2q}\delta _n} \right)$ and $ \lVert x_n-x_{n-1}+ h\beta\nabla g( x_n) \rVert =\mathcal{O}\left( \frac{1}{n^{q}} \right) $, as $ n\rightarrow +\infty.
$
The proof is complete.\qed
\end{proof}

In the sequel, we  simultaneously prove the strong convergence of the sequence generated by
 (IPGA) and  the convergence rates of the function values and discrete velocities to zero. To do this, let $\bar{x}_k$ be the unique solution of the  strongly convex optimization problem
$
\underset{x\in \mathcal{X}} \min ~g_k(x):=g( x ) +\frac{a_k}{2\delta_k  }\lVert x \rVert ^2.
$
Then,
 $\nabla g_k(\bar{x}_k)=\nabla g(\bar{x}_k)+\frac{a_k}{2\delta_k}\bar{x}_k=0$  and $ \lim _{k\rightarrow +\infty}\bar{x}_k=\bar{x}^*.$
Here, $\bar{x}^*$ is   the minimal norm element of the solution set of the problem (\ref{g}).  Obviously,
  $\lVert \bar{x}_k \rVert \le \lVert \bar{x}^* \rVert. $
Further, from \cite[Lemma A.1 (c)]{2025LC}, we can deduce that  for almost  every  $  k\ge 1$,
\begin{eqnarray}\label{axii}
\lVert \bar{x}_{k+1}-\bar{x}_k \rVert\le  \min\left\{\left(\frac{\delta _{k+1}a_k}{\delta _k a_{k+1}}-1 \right)
\lVert \bar{x}_k \rVert,\left(1-\frac{\delta _{k}a_{k+1}}{\delta _{k+1} a_{k}} \right)
\lVert \bar{x}_{k+1} \rVert\right\}.
\end{eqnarray}

In the following, to derive the strong convergence, we assume that  the following growth condition is satisfied:
\begin{eqnarray*}
& \mbox{ There exists } c>0    \mbox{ such that }  \delta _{k+1}\leq\frac{k^p}{(1+c)k^p-c(k+1)^p}\delta _k.  ~~~~~~~~   (\mathcal{G})
 \end{eqnarray*}
 \begin{remark}
Clearly,  the condition $(\mathcal{G}) $  can be rewritten as
$
\delta _{k+1} -\delta _k <\frac{c((k+1)^p-k^p)}{(1+c)k^p-c(k+1)^p}\delta_k.
$
By virtue of the gradient inequality, we can easily deduce that  the condition $(\mathcal{G}) $ is  a discretized form of $t\dot{\delta}_t\leq cp \delta(t)$, which is consistent with the condition $(\ref{b})$.
 \end{remark}

 Now,  we   give the following result  which is compatible with   Theorem $\ref{the5.1}$.
\begin{theorem}
Let $0<q<1$ and $1<p<q+1$.  Suppose that the condition $(\mathcal{G}) $ holds and there exists  $c_0>0$  such that $\delta^2_k\geq c_0 k^{p-q}$.
   Then,  for any sequence $( x_k ) _{k\in \mathbb{N}}$  generated by $\textup{(IPGA)}$, we have $\lim _{k\rightarrow +\infty}x_k=\bar{x}^*$ and the following conclusions hold:
   \begin{enumerate}
\item[{\rm (i)}] If $ 1<p<2q$,    as $ k\rightarrow +\infty$ it holds
\begin{eqnarray*}
\left\{\begin{split}
g( x_k ) -g( \bar{x}^* ) &=\mathcal{O}\left( \frac{1}{k^p\delta _k} \right),~~~~~~~\lVert x_k-\bar{x}_k \rVert =\mathcal{O}\left( \frac{1}{k^{\frac{1-q}{2}} } \right),\\
 \lVert\nabla g( x_k) \rVert &=\mathcal{O}\left( \frac{1}{k^{\frac{p-q+1}{2}}\delta _k} \right),\ \rVert x_k-x_{k-1}\rVert=\mathcal{O}\left( \frac{1}{k^{\frac{p-q+1}{2}}} \right).
\end{split}\right.
\end{eqnarray*}
\item[{\rm (ii)}]If $  2q\le p<\frac{1}{2}(3q+1)$, as $ k\rightarrow +\infty$ it holds
\begin{eqnarray*}
\left\{\begin{split}
g( x_k ) -g( \bar{x}^* ) &=\mathcal{O}\left( \frac{1}{k^p\delta _k} \right),~~~~\lVert x_k-\bar{x}_k \rVert =\mathcal{O}\left( \frac{1}{k^{\frac{1-q}{2}}} \right),\\
\lVert\nabla g( x_k) \rVert & =\mathcal{O}\left( \frac{1}{k^{\frac{q+1}{2}}\delta _k} \right),\ \rVert x_k-x_{k-1}\rVert=\mathcal{O}\left( \frac{1}{k^{\frac{q+1}{2}}} \right).
\end{split}\right.
\end{eqnarray*}
\item[{\rm (iii)}]If $  \frac{1}{2}(3q+1)\le p<q+1$, as $ k\rightarrow +\infty$ it holds
\begin{eqnarray*}
\left\{\begin{split}
\lVert x_k-\bar{x}_k \rVert& =\mathcal{O}\left( \frac{1}{k^{q-p+1}} \right),~~\lVert\nabla g( x_k) \rVert =\mathcal{O}\left( \frac{1}{k^{2q-p+1}\delta _k} \right),\\
\rVert x_k-x_{k-1}\rVert&=\mathcal{O}\left( \frac{1}{k^{2q-p+1 }} \right).
\end{split}\right.
\end{eqnarray*}
Further, if $\frac{1}{2}(3q+1)\leq p<\frac{1}{3}(4q+2)$, then $g( x_k ) -g( \bar{x}^* )= \mathcal{O}\left( \frac{1}{k^p\delta _k} \right) ,$ as  $ k\rightarrow +\infty,$ and
for $\frac{1}{3}(4q+2)\le p<q+1$, one has $g( x_k ) -g( \bar{x}^* )= \mathcal{O}\left( \frac{1}{k^{4q-2p+2}\delta _k} \right) ,$ as $k\rightarrow +\infty.$

 \end{enumerate}

\end{theorem}
\begin{proof}
Let   $0<\lambda<\frac{1}{2}\alpha h$ and $k\ge 2$. We define the energy sequence
\begin{eqnarray}\label{ener11}
\begin{split}
E_k=&b _k( g_k( x_k ) -g_k( \bar{x}_k ) ) +\frac{1}{2}\lVert \lambda ( x_k-\bar{x}_k ) +q_k( x_k-x_{k-1}+h\beta \nabla g( x_k ) ) \rVert ^2\\
&+\frac{1}{2}\lambda\gamma _k \lVert x_k-\bar{x}_k \rVert ^2+\frac{1}{2}\xi _{k-1}\lVert \nabla g( x_k ) \rVert ^2+\frac{1}{2}\eta _{k-1}\lVert x_k-x_{k-1} \rVert ^2+ \frac{1}{2}\sigma _k\lVert x_k \rVert ^2,
\end{split}
\end{eqnarray}
where
$b_k:= ( q_{k-1}+\alpha h ) B_{k-1}-h\beta q_k^{2},$ $B_k:=h( \beta +h\delta _k )q_k,$
$\gamma _k:=  q_{k-1}- q_k+\alpha h-\lambda + \frac{1}{2} h^2q_{k-1}a_{k-1} ,$
$\xi _{k-1}:=  \frac{1}{2}B_{k-1}^{2}-h^2\beta ^2   q_{k}^{2} - \frac{2\lambda\beta ^2   q_{k}^{2}}{q_{k-1}a_{k-1}} ,$
$\eta _{k-1}:=  ( q_{k-1}+\alpha h  ) ^2-q_{k}^{2}-\lambda q_{k}-3\lambda( q_{k-1}+\alpha h+ h^2q_{k-1}a_{k-1} ) +( q_{k-1}+\alpha h ) h^2q_{k-1}a_{k-1}  $ and
$\sigma _k:=(q_{k-1}+\alpha h-h^2q_{k-1}a_{k-1} )h^2q_{k-1}a_{k-1}-b_k \frac{a_k}{\delta _k}.$
Obviously, according to Appendix B(i), there exists $k_0\in \mathbb{N}$ such that $b_k\geq0$,  $\gamma _k\geq0$, $\xi _{k-1}\geq0$, $\eta _{k-1}\geq0$  and $\sigma _k\geq0$  for all $k\ge k_0$. Thus, $E_k\ge 0$   for all  $k\ge k_0$.

Let $v_k:=\lambda ( x_k-\bar{x}_k ) +q_k( x_k-x_{k-1}+h\beta \nabla g( x_k ) ).$
 Clearly,
\begin{eqnarray}\label{5.16}
\begin{split}
& E_{k+1}-E_k\\
=& b _{k+1} ( g_{k+1}( x_{k+1} ) -g_{k+1}( \bar{x}_{k+1} ) ) -b_k( g_k( x_k ) -g_k( \bar{x}_k ) ) \\
& +\frac{1}{2}\left( \lVert v_{k+1} \rVert ^2-\lVert v_k \rVert ^2 \right)+ \frac{1}{2}\lambda\gamma _{k+1}\lVert x_{k+1}-\bar{x}_{k+1} \rVert ^2-\frac{1}{2}\lambda\gamma _k \lVert x_k-\bar{x}_k \rVert ^2\\
 &+\frac{1}{2}\xi _{k}\lVert \nabla g( x_{k+1}) \rVert ^2-\frac{1}{2}\xi _{k-1}\lVert \nabla g( x_k ) \rVert ^2\\
&+\frac{1}{2}\eta _k\lVert x_{k+1}-x_k \rVert ^2-\frac{1}{2}\eta _{k-1}\lVert x_k-x_{k-1} \rVert ^2+\frac{1}{2}\sigma _{k+1}\lVert x_{k+1} \rVert ^2-\frac{1}{2}\sigma _k\lVert x_k \rVert ^2.
\end{split}
\end{eqnarray}
Now, we  evaluate the third term on the right-hand side of (\ref{5.16}). Firstly,
\begin{eqnarray*}
\begin{split}
&\frac{1}{2}\lVert v_k \rVert ^2\\
=&\frac{1}{2}\lVert \lambda ( x_k-\bar{x}_k ) +q_k\left( x_k-x_{k-1}+h\beta\nabla g( x_k ) \right) \rVert ^2\\
=&\frac{1}{2}\lVert \lambda ( x_k-\bar{x}_k ) +\left( q_k+\alpha h \right) ( x_{k+1}-x_k ) +B_k\nabla g\left( x_{k+1} \right) +h^2q_ka_kx_{k+1} \rVert ^2\\
=&\frac{1}{2}\lambda ^2\lVert x_k-\bar{x}_k \rVert ^2+\frac{1}{2}( q_k+\alpha h )^2 \lVert x_{k+1}-x_k \rVert ^2+\frac{1}{2} {B}_{k}^{2}\lVert \nabla g( x_{k+1} ) \rVert^2+\frac{1}{2}h^4q^2_ka^2_k\lVert x_{k+1} \rVert ^2\\
&+\underset{A}{\underbrace{\lambda ( q_k+\alpha h ) \left< x_k-\bar{x}_k,x_{k+1}-x_k \right> }}+\underset{B}{\underbrace{\lambda B_k\left< x_k-\bar{x}_k,\nabla g( x_{k+1} ) \right> }}\\
&+\underset{C}{\underbrace{\lambda h^2q_ka_k\left< x_k-\bar{x}_k,x_{k+1} \right> }}+( q_k+\alpha h ) B_k\left< x_{k+1}-x_k,\nabla g( x_{k+1} ) \right> \\ &+\underset{D}{\underbrace{( q_k+\alpha h ) h^2q_ka_k\left< x_{k+1}-x_k,x_{k+1} \right> }}
+h^2B_kq_ka_k\left< \nabla g( x_{k+1}) ,x_{k+1} \right>,
\end{split}
\end{eqnarray*}
where the second equality holds due to $B_k=h( \beta +h\delta _k )q_k$ and (\ref{5.3}).
Secondly,
\begin{eqnarray*}
\begin{split}
&\frac{1}{2}\lVert v_{k+1} \rVert ^2\\
=&\frac{1}{2}\lVert \lambda ( x_{k+1}-\bar{x}_{k+1} ) +q_{k+1}\left( x_{k+1}-x_k+h\beta\nabla g( x_{k+1} ) \right) \rVert ^2\\
=&\frac{1}{2}\lambda ^2\lVert x_{k+1}-\bar{x}_{k+1} \rVert ^2+\frac{1}{2}q_{k+1}^{2}\lVert x_{k+1}-x_k \rVert ^2+\frac{1}{2}h^2\beta ^2q_{k+1}^{2}\lVert \nabla g( x_{k+1} ) \rVert ^2\\
&+\underset{E}{\underbrace{\lambda q_{k+1}\left< x_{k+1}-\bar{x}_{k+1},x_{k+1}-x_k \right> }}+\underset{F}{\underbrace{\lambda h\beta q_{k+1}\left< x_{k+1}-\bar{x}_{k+1},\nabla g( x_{k+1} ) \right> }}\\
&+ h\beta q_{k+1}^{2}\left< x_{k+1}-x_k,\nabla g( x_{k+1} ) \right> .\\
\end{split}
\end{eqnarray*}

Finally, note that
\begin{eqnarray*}
\left\{\begin{split}               	A=&\frac{1}{2}\lambda\left( q_k+\alpha h \right) \left( \lVert x_{k+1}-\bar{x}_k \rVert ^2-\lVert x_k-\bar{x}_k \rVert ^2-\lVert x_{k+1}-x_k \rVert ^2 \right),\\
	B=&\lambda B_k \Big(\langle x_{k+1}-\bar{x}_k,\nabla g( x_{k+1} ) \rangle - \langle x_{k+1}-x_k,\nabla g ( x_{k+1}  )  \rangle\Big),\\
C=&\lambda h^2q_ka_k\Big(\left< x_k-\bar{x}_k,x_k \right> +\left< x_k-\bar{x}_k,x_{k+1}-x_k \right>\Big)\\
	=&\frac{1}{2}\lambda h^2q_ka_k \Big( \lVert x_k-\bar{x}_k \rVert ^2+\lVert x_k \rVert ^2-\lVert \bar{x}_k \rVert ^2\\
	&~~~~~~~~~~~~~+ \lVert x_{k+1}-\bar{x}_k \rVert ^2-\lVert x_k-\bar{x}_k \rVert ^2-\lVert x_{k+1}-x_k \rVert ^2 \Big)\\
	=&\frac{1}{2}\lambda h^2q_ka_k\Big(\lVert x_k \rVert ^2 -\lVert \bar{x}_k \rVert ^2+ \lVert x_{k+1}-\bar{x}_k \rVert ^2- \lVert x_{k+1}-x_k \rVert ^2\Big),\\
	D=&\frac{1}{2}\left( q_k+\alpha h \right) h^2q_ka_k\left( \lVert x_{k+1}-x_k \rVert ^2+\lVert x_{k+1} \rVert ^2-\lVert x_k \rVert ^2 \right),\\
	E=&\frac{1}{2}\lambda q_{k+1}\left( -\lVert x_k-\bar{x}_{k+1} \rVert ^2+\lVert x_{k+1}-\bar{x}_{k+1} \rVert ^2+\lVert x_{k+1}-x_k \rVert ^2 \right),\\
    F\le& \frac{1}{4}\lambda h^2q_ka_k \lVert x_{k+1}-\bar{x}_{k+1} \rVert ^2+\frac{\beta ^2\lambda q_{k+1}^{2}}{q_ka_k} \lVert \nabla g( x_{k+1} ) \rVert ^2.
\end{split}\right.
\end{eqnarray*}
Thus,
\begin{eqnarray*}
\small{\begin{split}
&\frac{1}{2}\lVert v_{k+1} \rVert ^2-\frac{1}{2}\lVert v_k \rVert ^2\\
=&\left( \lambda B_k+h\beta q_{k+1}^{2}-( q_k+\alpha h ) B_k\right) \left<\nabla g ( x_{k+1}  ), x_{k+1}-x_k  \right> -\lambda B_k\left<\nabla g ( x_{k+1}  ), x_{k+1}-\bar{x}_k  \right>\\
&+\frac{1}{2}\left( h^2\beta ^2q_{k+1}^{2}- {B}_{k}^{2}+  \frac{2\beta ^2\lambda q_{k+1}^{2}}{q_ka_k}  \right) \lVert \nabla g\left( x_{k+1} \right) \rVert ^2\\
&+\frac{1}{2}\lambda\left( \lambda  + q_{k+1}+  \frac{1}{2}h^2q_ka_k \right) \lVert x_{k+1}-\bar{x}_{k+1} \rVert ^2+\frac{1}{2}\lambda( q_k+\alpha h -\lambda  ) \lVert x_k-\bar{x}_k \rVert ^2\\
&+\frac{1}{2}\Big( q_{k+1}^{2}+\lambda q_{k+1}-( q_k+\alpha h ) ^2+\lambda( q_k+\alpha h + h^2q_ka_k)-( q_k+\alpha h ) h^2q_ka_k \Big) \lVert x_{k+1}-x_k \rVert ^2\\
&-\frac{1}{2}  \lambda( q_k+\alpha h +  h^2q_ka_k ) \lVert x_{k+1}-\bar{x}_k \rVert ^2-\frac{1}{2}\lambda q_{k+1}\lVert x_k-\bar{x}_{k+1} \rVert ^2\\
&-\frac{1}{2}h^2q_ka_k\left(   q_k+\alpha h  +h^2q_ka_k \right) \lVert x_{k+1} \rVert ^2+\frac{1}{2}h^2q_ka_k\left(   q_k+\alpha h  -\lambda \right) \lVert x_k \rVert ^2\\
&+\frac{1}{2}\lambda h^2q_ka_k\lVert \bar{x}_k \rVert ^2-h^2B_kq_ka_k\left< \nabla g( x_{k+1} ) ,x_{k+1} \right>.
\end{split}}
\end{eqnarray*}
Note that there   exists $k_1\geq k_0$ such that $\lambda B_k+\beta hq_{k+1}^{2}-( q_k+\alpha h ) B_k\leq0$  for all $k\geq k_1$.
Together with $b _{k+1} = ( q_{k}+\alpha h ) B_{k}-h\beta q_{k+1}^{2} $, $\left< \nabla g( x_{k+1} ) ,x_{k+1}-x_k \right> \ge g( x_{k+1} ) -g( x_k )$ and $\left< \nabla g( x_{k+1} ) ,x_{k+1}-\bar{x}_k \right> \ge g( x_{k+1} ) -g( \bar{x}_k )$,   we obtain that for any $k\geq k_1,$
\begin{eqnarray}\label{1053x}
\small{\begin{split}
&\frac{1}{2}\lVert v_{k+1} \rVert ^2-\frac{1}{2}\lVert v_k \rVert ^2\\
\le&( b _{k+1} -\lambda B_k ) ( g( x_k ) -g( x_{k+1} ) ) +\lambda B_k( g( \bar{x}_k ) -g( x_{k+1} ) )\\
&+\frac{1}{2}\left(h^2 \beta ^2q_{k+1}^{2}-B_{k}^{2}+  \frac{2\beta ^2\lambda q_{k+1}^{2}}{q_ka_k}  \right) \lVert \nabla g( x_{k+1} ) \rVert ^2\\
&+\frac{1}{2}\lambda\left( \lambda  + q_{k+1}+  \frac{1}{2}h^2q_ka_k \right) \lVert x_{k+1}-\bar{x}_{k+1} \rVert ^2+\frac{1}{2}\lambda( q_k+\alpha h -\lambda  ) \lVert x_k-\bar{x}_k \rVert ^2\\
&+\frac{1}{2}\Big( q_{k+1}^{2}+\lambda q_{k+1}-( q_k+\alpha h ) ^2+\lambda( q_k+\alpha h+ h^2q_ka_k)\\
&~~~~~~~~~~~~~~~~~~~~~~~~~~-( q_k+\alpha h ) h^2q_ka_k \Big) \lVert x_{k+1}-x_k \rVert ^2\\
&-\frac{1}{2}  \lambda( q_k+\alpha h  +  h^2q_ka_k ) \lVert x_{k+1}-\bar{x}_k \rVert ^2-\frac{1}{2}\lambda q_{k+1}\lVert x_k-\bar{x}_{k+1} \rVert ^2\\
&-\frac{1}{2} h^2q_ka_k\left(   q_k+\alpha h  +h^2q _ka _k \right) \lVert x_{k+1} \rVert ^2+\frac{1}{2}h^2q_ka_k\left(  q_k+\alpha h  -\lambda \right) \lVert x_k \rVert ^2\\
&+\frac{1}{2}\lambda h^2q_ka_k\lVert \bar{x}_k \rVert ^2-h^2B_kq_ka_k\left< \nabla g( x_{k+1} ) ,x_{k+1} \right>.
\end{split}}
\end{eqnarray}
Further, note that
\begin{eqnarray*}
\begin{split}
&( b _{k+1}-\lambda B_k ) ( g( x_k ) -g( x_{k+1} ) )+\lambda B_k ( g( \bar{x}_k ) -g( {x}_{k+1} ) )\\
= &-b _{k+1}( g_{k+1}( x_{k+1} ) -g_{k+1}( \bar{x}_{k+1} ) )+(b _{k+1}-\lambda B_k ) ( g_k( x_k ) -g_k( \bar{x}_k ))\\
&-b _{k+1}( g_{k+1}( \bar{x}_{k+1} ) -g_k( \bar{x}_k ) )+b _{k+1}\frac{a_{k+1}}{2\delta _{k+1}} \lVert x_{k+1} \rVert ^2  \\
&-(b _{k+1}-\lambda B_k ) \frac{a_k}{2\delta _k} \lVert x_k \rVert ^2 -\lambda B_k  \frac{a_k}{2\delta _k}\lVert \bar{x}_k \rVert ^2
\end{split}
\end{eqnarray*}
and
\begin{eqnarray*}
\begin{split}
&-b _{k+1} ( g_{k+1}( \bar{x}_{k+1} ) -g_k( \bar{x}_k ) )\\
&=b _{k+1}\left( g_k( \bar{x}_k ) -g_k( \bar{x}_{k+1} ) +\left( \frac{a_k}{2\delta _k}-\frac{a_{k+1}}{2\delta _{k+1}} \right) \lVert \bar{x}_{k+1} \rVert ^2 \right)\\
&\le b _{k+1}\left( -\frac{a_k}{2\delta _k}\lVert \bar{x}_{k+1}-\bar{x}_k \rVert ^2+\left( \frac{a_k}{2\delta _k}-\frac{a_{k+1}}{2\delta _{k+1}} \right) \lVert \bar{x}_{k+1} \rVert ^2 \right),\\
\end{split}
\end{eqnarray*}
 where the inequality holds due to the $\frac{a_k}{\delta _k }$-strongly convexity of $g_k$. Thus, it follow from  (\ref{1053x}) that
\begin{eqnarray}\label{5.241}
\small{\begin{split}
&\frac{1}{2}\lVert v_{k+1} \rVert ^2-\frac{1}{2}\lVert v_k \rVert ^2\\
\le&-b _{k+1}( g_{k+1}( x_{k+1} ) -g_{k+1}( \bar{x}_{k+1} ) ) +( b _{k+1}-\lambda B_k ) ( g_k( x_k ) -g_k( \bar{x}_k ))\\
&+\frac{1}{2}\left(h^2\beta ^2q_{k+1}^{2}-B_{k}^{2}+ \frac{2\lambda \beta ^2q_{k+1}^{2}}{q_ka_k}  \right) \lVert \nabla g( x_{k+1} ) \rVert ^2\\
&+\frac{1}{2}\lambda\left( \lambda  + q_{k+1}+  \frac{1}{2}h^2q_ka_k \right) \lVert x_{k+1}-\bar{x}_{k+1} \rVert ^2+\frac{1}{2} \lambda\left( q_k+\alpha h  -\lambda  \right) \lVert x_k-\bar{x}_k \rVert ^2\\
&+\frac{1}{2}\Big( q_{k+1}^{2}+\lambda q_{k+1}-( q_k+\alpha h ) ^2+\lambda( q_k+\alpha h+ h^2q_ka_k) \\
&~~~~~~~~~~~~~~~~~~~~~~~~~~-( q_k+\alpha h ) h^2q_ka_k \Big) \lVert x_{k+1}-x_k \rVert ^2\\
&-\frac{1}{2}  \lambda( q_k+\alpha h  + h^2q_ka_k ) \lVert x_{k+1}-\bar{x}_k \rVert ^2-\frac{1}{2}\lambda q_{k+1}\lVert x_k-\bar{x}_{k+1} \rVert ^2\\
&+\frac{1}{2}\left(b _{k+1} \frac{a_{k+1}}{\delta _{k+1}} -\left( q_k+\alpha h+h^2q_ka_k \right) h^2q_ka_k \right) \lVert x_{k+1} \rVert ^2\\
&+\frac{1}{2}\left( ( q_k+\alpha h-\lambda ) h^2q_ka_k -( b _{k+1}-\lambda B_k ) \frac{a_k}{\delta _k}\right) \lVert x_k \rVert ^2\\
&+\frac{1}{2}b _{k+1}\left( \frac{a_k}{ \delta _k}-\frac{a_{k+1}}{ \delta _{k+1}} \right)  \lVert \bar{x}_{k+1} \rVert ^2+\frac{1}{2}\lambda\left( h^2q_ka_k- B_k\frac{a_k}{\delta _k} \right)  \lVert \bar{x}_k \rVert ^2\\
&-\frac{1}{2}b _{k+1}\frac{a_k}{ \delta _k}\lVert \bar{x}_{k+1}-\bar{x}_k \rVert ^2-h^2B_kq_ka_k\left< \nabla g( x_{k+1} ) ,x_{k+1} \right>.\\
\end{split}}
\end{eqnarray}

Now, we estimate  the  terms $-\frac{1}{2}  \lambda(q_k+\alpha h +  h^2q_ka_k ) \lVert x_{k+1}-\bar{x}_k \rVert ^2$, $-\frac{1}{2}\lambda q_{k+1}\lVert \bar{x}_{k+1}-x_k \rVert ^2$ and $-h^2B_kq_ka_k\left< \nabla g( x_{k+1} ),x_{k+1} \right>$ in (\ref{5.241}).
Clearly,
\begin{eqnarray*}
\small{\begin{split}
&-\frac{1}{2}  \lVert x_{k+1}-\bar{x}_k \rVert ^2\\
=&-\frac{1}{2}  \lVert x_{k+1}-\bar{x}_{k+1} \rVert ^2-\frac{1}{2}  \lVert \bar{x}_{k+1}-\bar{x}_k \rVert ^2-  \left< x_{k+1}-x_k,\bar{x}_{k+1}-\bar{x}_k \right>-  \left<x_k- \bar{x}_{k+1},\bar{x}_{k+1}-\bar{x}_k \right>\\
=&-\frac{1}{2}  \lVert x_{k+1}-\bar{x}_{k+1} \rVert ^2-\frac{1}{2}  \lVert \bar{x}_{k+1}-\bar{x}_k \rVert ^2-  \left< x_{k+1}-x_k,\bar{x}_{k+1}-\bar{x}_k \right>\\
&-  \left< x_k-\bar{x}_k,\bar{x}_{k+1}-\bar{x}_k \right>+  \lVert \bar{x}_{k+1}-\bar{x}_k \rVert ^2\\
\le& -\frac{1}{2}  \lVert x_{k+1}-\bar{x}_{k+1} \rVert ^2-\frac{1}{2}  \lVert \bar{x}_{k+1}-\bar{x}_k \rVert ^2+ \lVert x_{k+1}-x_k \rVert ^2+\frac{1}{4}\lVert \bar{x}_{k+1}-\bar{x}_k \rVert ^2\\
&+  \left< x_k-\bar{x}_k,\bar{x}_k-\bar{x}_{k+1} \right> +  \lVert \bar{x}_{k+1}-\bar{x}_k \rVert ^2\\
=&-\frac{1}{2}  \lVert x_{k+1}-\bar{x}_{k+1} \rVert ^2+\frac{3}{4}  \lVert \bar{x}_{k+1}-\bar{x}_k \rVert ^2+  \lVert x_{k+1}-x_k \rVert ^2+\left< x_k-\bar{x}_k,\bar{x}_k-\bar{x}_{k+1} \right>,
\end{split}}
\end{eqnarray*}
$
-\frac{1}{2}\lambda q_{k+1}\lVert \bar{x}_{k+1}-x_k \rVert ^2=-\frac{1}{2}\lambda q_{k+1}\lVert x_k-\bar{x}_k \rVert ^2-\frac{1}{2}\lambda q_{k+1}\lVert \bar{x}_k-\bar{x}_{k+1} \rVert ^2
-\lambda q_{k+1}\langle x_k-\bar{x}_k,\bar{x}_k-\bar{x}_{k+1} \rangle
$
and $
-h^2B_kq_ka_k\langle \nabla g( x_{k+1} ),x_{k+1} \rangle\le \frac{1}{4}B_{k}^{2}\lVert \nabla g( x_{k+1} ) \rVert ^2+h^4q_{k}^2a_{k}^2\lVert x_{k+1} \rVert ^2.
$
Note that there exists $k_2\geq k_1$ such that   $q_k+\alpha h +  h^2q_ka_k  -  q_{k+1}\geq0$ for all $ t\geq t_2$. Then, for all $t\geq t_2,$
$
 \lambda( q_k+\alpha h +  h^2q_ka_k  -  q_{k+1})\langle x_k-\bar{x}_k,\bar{x}_k-\bar{x}_{k+1} \rangle
\le \lambda( q_k+\alpha h  + h^2q_ka_k - q_{k+1}) (s_k\lVert x_k-\bar{x}_k \rVert ^2+ \frac{1}{s_k}\lVert\bar{x}_{k+1}-\bar{x}_k \rVert ^2 ),
$
where $s_k:= s( kh ) ^{q-p}$ with $s<\frac{ha}{4\alpha}$. Therefore, for all $t\geq t_2,$
\begin{eqnarray*}
\begin{split}
&-\frac{1}{2} \lambda( q_k+\alpha h + h^2q_ka_k ) \lVert x_{k+1}-\bar{x}_k \rVert ^2-\frac{1}{2}\lambda q_{k+1}\lVert \bar{x}_{k+1}-x_k \rVert ^2\\
&-h^2B_kq_ka_k\left< \nabla g( x_{k+1} ),x_{k+1} \right>\\
\le&-\frac{1}{2} \lambda( q_k+\alpha h  +  h^2q_ka_k )\lVert x_{k+1}-\bar{x}_{k+1} \rVert ^2+ \lambda( q_k+\alpha h + h^2q_ka_k ) \lVert x_{k+1}-x_k \rVert ^2\\
&+\frac{1}{2}\lambda\Big(2( q_k+\alpha h + h^2q_ka_k -q_{k+1}) s_k-  q_{k+1} \Big) \lVert x_k-\bar{x}_k \rVert ^2\\
&+\left( \frac{3}{4} \lambda( q_k+\alpha h + h^2q_ka_k )+  \lambda( q_k+\alpha h + h^2q_ka_k -  q_{k+1})\frac{1}{s_k} -\frac{1}{2}\lambda q_{k+1} \right) \lVert \bar{x}_{k+1}-\bar{x}_k \rVert ^2\\
&+\frac{1}{4}B_{k}^{2}\lVert \nabla g( x_{k+1} ) \rVert ^2+h^4q_{k}^2a_{k}^2\lVert x_{k+1} \rVert ^2.
\end{split}
\end{eqnarray*}
 This together with  (\ref{5.241}) yields for all $t\geq t_2,$
\begin{eqnarray*}
{\begin{split}
& \frac{1}{2}\lVert v_{k+1} \rVert ^2-\frac{1}{2}\lVert v_k \rVert ^2\\
\le&-b _{k+1} ( g_{k+1}( x_{k+1} ) -g_{k+1}( \bar{x}_{k+1}))+(b _{k+1}-\lambda B_k ) ( g_k( x_k ) -g_k( \bar{x}_k ) ) \\
&-\frac{1}{2}\xi_k \lVert \nabla g( x_{k+1} ) \rVert ^2-\frac{1}{2}\lambda\gamma_{k+1}\lVert x_{k+1}-\bar{x}_{k+1} \rVert ^2\\
&+\frac{1}{2} \lambda\left(  q_k+\alpha h  -\lambda   +2 ( q_k+\alpha h+ h^2q_ka_k - q_{k+1}) s_k-  q_{k+1}\right) \lVert x_k-\bar{x}_k \rVert ^2\\
&-\frac{1}{2}\eta_k \lVert x_{k+1}-x_k \rVert ^2-\frac{1}{2}\sigma_{k+1} \lVert x_{k+1} \rVert ^2\\
&+\frac{1}{2}\left( ( q_k+\alpha h-\lambda ) h^2q_ka_k -(b _{k+1}-\lambda B_k ) \frac{a_k}{\delta _k}\right) \lVert x_k \rVert ^2\\
&+\frac{1}{2}b _{k+1}\left( \frac{a_k}{ \delta _k}-\frac{a_{k+1}}{ \delta _{k+1}} \right)  \lVert \bar{x}_{k+1} \rVert ^2+\frac{1}{2}\lambda \left( h^2q_ka_k- B_k\frac{a_k}{\delta _k} \right)  \lVert \bar{x}_k \rVert ^2\\
&+\left( \frac{3}{4}\lambda( q_k+\alpha h+ h^2q_ka_k )+ \lambda( q_k+\alpha h+ h^2q_ka_k - q_{k+1})\frac{1}{s_k}\right.\\
&~~~~~~~~~~~~~-\frac{1}{2}\lambda q_{k+1}-\left.\frac{1}{2}b _{k+1}\frac{a_k}{ \delta _k} \right) \lVert \bar{x}_{k+1}-\bar{x}_k \rVert ^2.
\end{split}}
\end{eqnarray*}
Combining this   with (\ref{5.16}), we obtain that for all $t\geq t_2,$
\begin{eqnarray}\label{5.30}
{\begin{split}
&E_{k+1}-E_k\\
\leq&- \left(b_k+ \lambda B_k-b _{k+1} \right) ( g_k( x_k ) -g_k( \bar{x}_k) )\\
&- \frac{1}{2}\lambda\Big( -2 q_k -2 ( q_k+\alpha h+ h^2q_ka_k - q_{k+1}) s_k \\
&\left.~~~~~~~~~~+ q_{k+1}+ q_{k-1}+ \frac{1}{2}h^2q_{k-1}a_{k-1}
\right)  \lVert x_k-\bar{x}_k \rVert ^2-\frac{1}{2}\xi _{k-1}\lVert \nabla g( x_k ) \rVert ^2\\
&- \frac{1}{2}\lambda\left(B_k\frac{a_k}{\delta _k}-h^2q_ka_k \right)   \lVert \bar{x}_k \rVert ^2-\frac{1}{2}\eta _{k-1}\lVert x_k-x_{k-1} \rVert ^2\\
&-  \frac{1}{2}\left( ( b _{k+1}-\lambda B_k) \frac{a_k}{\delta _k}-( q_k+\alpha h-\lambda ) h^2q_ka_k  + \sigma _k\right) \lVert x_k \rVert ^2\\
& +\frac{1}{2} b _{k+1}\left( \frac{a_k}{ \delta _k}-\frac{a_{k+1}}{ \delta _{k+1}} \right) \lVert \bar{x}_{k+1} \rVert ^2+\left( \frac{3}{4}\lambda( q_k+\alpha h+ h^2q_ka_k )\right.\\
&+\lambda( q_k+\alpha h+ h^2q_ka_k - q_{k+1})\frac{1}{s_k} -\frac{1}{2}\lambda q_{k+1}-\left.b _{k+1}\frac{a_k}{2\delta _k} \right) \lVert \bar{x}_{k+1}-\bar{x}_k \rVert ^2.
\end{split}}
\end{eqnarray}

Now, we analyze the coefficients of the fourth, seventh and eighth terms   on the right-hand side of (\ref{5.30}).
Firstly,  from  the condition $(\mathcal{G}) $ and $a_k=a(kh)^{-p}$, we have $ \frac{a_{k+1}}{\delta _k}-\frac{a_{k+1}}{\delta _{k+1}}\leq c \left(\frac{a_k}{\delta _k}-\frac{a_{k+1}}{\delta _k}\right)  $. Then,  there exists $C_0>0$ and $k_3\geq k_2$ such that for all $k\geq k_3$,
\begin{eqnarray}\label{5.32}
\begin{split}
&\frac{1}{2}b _{k+1}\left( \frac{a_k}{\delta _k}-\frac{a_{k+1}}{\delta _{k+1}} \right)  \lVert \bar{x}_{k+1} \rVert ^2\\
\le&\frac{1}{2} b _{k+1}\left( \left( \frac{a_k}{\delta _k}-\frac{a_{k+1}}{\delta _k} \right) +\left( \frac{a_{k+1}}{\delta _k}-\frac{a_{k+1}}{\delta _{k+1}} \right) \right) \lVert \bar{x}_{k+1} \rVert ^2\\
\le&\frac{1}{2}(1+c) b _{k+1}\left( \frac{a_k}{\delta _k}-\frac{a_{k+1}}{\delta _k} \right) \lVert \bar{x}_{k+1} \rVert ^2\\
\le &C_0( k+1 ) ^{2q-p-1}.
\end{split}
\end{eqnarray}
where the last inequality holds due to   $\lVert \bar{x}_{k+1} \rVert \le \lVert \bar{x}^* \rVert$ and $b _{k+1}\Big( \frac{a_k}{\delta _k}-\frac{a_{k+1}}{\delta _k} \Big)=\mathcal{O}\Big( (k+1)^{2q-p-1} \Big)$, as $k\rightarrow+\infty$.

Secondly, according to Appendix B(iii), we have
$
\lVert \bar{x}_{k+1}-\bar{x}_k \rVert \le\frac{(1+c)p }{ k-cp  }\lVert \bar{x}^* \rVert,$  for  $ k$   big   enough.
Then,  there exist $C_1>0$ and $k_4\geq k_3$ such that for all $k\geq k_4,$
\begin{eqnarray}\label{5.33}
\begin{split}
&\left( \frac{3}{4}\lambda( q_k+\alpha h+ h^2q_ka_k )+ \lambda( q_k+\alpha h+ h^2q_ka_k - q_{k+1})\frac{1}{s_k}  \right) \lVert \bar{x}_{k+1}-\bar{x}_k \rVert ^2\\
\le &C_1( k+1 ) ^{\max \{q-2, p-q-2\}}.
\end{split}
\end{eqnarray}
Here  the inequality holds due to $\frac{3}{4}\lambda( q_k+\alpha h+ h^2q_ka_k )+ \lambda( q_k+\alpha h+ h^2q_ka_k - q_{k+1})\frac{1}{s_k}=\mathcal{O}\Big( k^{\max \{q, p-q \}} \Big)$, as $k\rightarrow+\infty$.

Finally,   combining (\ref{5.32})  and (\ref{5.33}) and noting that $ \frac{1}{2}\lambda\Big(B_k\frac{a_k}{\delta _k}-h^2q_ka_k \Big)=\frac{1}{2}\lambda h\beta a_kq_k\frac{1}{\delta_k}\ge 0,$ we can derive from (\ref{5.30}) that for all $ k\ge k_4,$
\begin{eqnarray}\label{5.34}
\small{\begin{split}
&E_{k+1}-E_k\\
\leq&- \underset{\nu_k}{ \underbrace{\left(b_k+ \lambda B_k-b _{k+1} \right)}}( g_k( x_k ) -g_k( \bar{x}_k) )\\
&-\underset{\omega _k}{\underbrace{\frac{1}{2}\lambda\left( -2 q_k -2 ( q_k+\alpha h+ h^2q_ka_k - q_{k+1}) s_k+  q_{k+1}+ q_{k-1}+ \frac{1}{2}h^2q_{k-1}a_{k-1}
\right)}} \lVert x_k-\bar{x}_k \rVert ^2\\
&-\frac{1}{2}\xi _{k-1}\lVert \nabla g( x_k ) \rVert ^2 -\frac{1}{2}\eta _{k-1}\lVert x_k-x_{k-1} \rVert ^2\\
&- \underset{\tau _k}{\underbrace{\frac{1}{2}\left( ( b _{k+1}-\lambda B_k) \frac{a_k}{\delta _k}-( q_k+\alpha h-\lambda ) h^2q_ka_k  + \sigma _k\right)}}\lVert x_k \rVert ^2\\
 &+C_0( k+1 ) ^{2q-p-1}+ C_1( k+1 ) ^{\max \{q-2, p-q-2 \}}.
\end{split}}
\end{eqnarray}

On the other hand, from (\ref{ener11}) and noting that
$
\frac{1}{2}\lVert v_k \rVert ^2\le \lambda ^2\lVert x_k-\bar{x}_k \rVert ^2+2q_k^2\lVert x_k-x_{k-1} \rVert ^2+2\beta ^2h^2q_{k}^{2}\lVert \nabla g( x_k ) \rVert ^2,
$
 we can derive that
\begin{eqnarray}\label{5.3444}
\small{\begin{split}
E_k\le& b_k( g_k( x_k ) -g_k( \bar{x}_k ) )+\left(\lambda ^2+\frac{1}{2}\lambda\gamma _k\right)\lVert x_k-\bar{x}_k \rVert ^2+\left(2\beta ^2h^2q_{k}^{2}+\frac{1}{2}\xi _{k-1}\right)\lVert \nabla g( x_k ) \rVert ^2\\
&+\left(2q_k^2+\frac{1}{2}\eta _{k-1}\right)\lVert x_k-x_{k-1} \rVert ^2+\frac{1}{2}\sigma _k\lVert x_k \rVert ^2.
\end{split}}
\end{eqnarray}

Let  $r=\max\{q, p-q\}$ and $H$ be arbitrarily positive constant. Then, it follows from $(\ref{5.34})$  and $(\ref{5.3444})$ that for all $ k\ge k_4$,
\begin{eqnarray}\label{5.35xa}
\small{\begin{split}
&E_{k+1}-E_k+\frac{H}{k^{r}}E_k\\
\le&-( \nu_k-Hb_kk^{-r})( g_k( x_k )-g_k( \bar{x}_k) )-\left(\omega _k-H\left(\lambda ^2+\frac{1}{2}\lambda\gamma _k\right)k^{-r}\right)\lVert x_{k}-\bar{x}_{k} \rVert ^2\\
&-\frac{1}{2}\left(\xi _{k-1}-2H\left(2\beta ^2h^2q_{k}^{2}+\frac{1}{2}\xi _{k-1}\right)k^{-r}\right)\lVert \nabla g( x_k ) \rVert ^2\\
&-\frac{1}{2}\left(\eta _{k-1}-2H\left(2q_k^2+\frac{1}{2}\eta _{k-1}\right)k^{-r}\right)\lVert x_k-x_{k-1} \rVert ^2-\left(\tau _k-\frac{1}{2}H\sigma _kk^{-r}\right)\lVert x_k \rVert ^2\\
&+C_0( k+1 ) ^{2q-p-1}+ C_1( k+1 ) ^{r-2 }.
\end{split}}
\end{eqnarray}
 According to Appendix B (i) and (ii), for $k $ big enough, we have
 \begin{eqnarray}\label{5.3xs5}
\left\{\begin{split}
&\nu_k=\mathcal{O}\left( k^q\delta _k \right),\ \ Hb_kk^{-r}=\mathcal{O}\left( k^{2q-r}\delta _k \right),\\
&\omega _k=\mathcal{O}\left( k^{q-p} \right),\ \ H\left(\lambda ^2+\frac{1}{2}\lambda\gamma _k\right)k^{-r}=\mathcal{O}\left( k^{-r} \right),\\
&\xi _{k-1}- 2H\left(2\beta ^2h^2q_{k}^{2}+\frac{1}{2}\xi _{k-1}\right)k^{-r}=\mathcal{O}\left( k^{2q}\delta _{k}^{2} \right)\geq 0,\\
&\eta _{k-1}=\mathcal{O}\left( k^q \right),\ \ 2H\left(2q_k^2+\frac{1}{2}\eta _{k-1}\right)k^{-r}=\mathcal{O}\left( k^{2q-r} \right),\\
&\tau _k-\frac{1}{2}H\sigma _kk^{-r}  \geq 0.
\end{split}\right.
\end{eqnarray}
Now, take $0 < H < \min\left\{\lambda h^{-q},\frac{1}{2}(\alpha h-2\lambda)h^{-q},\frac{\left( \frac{1}{2}ah-2\alpha s \right) h^{q-p+1}}{\alpha h+\lambda}\right\}$.  Thus,  we   deduce from (\ref{5.35xa}) and (\ref{5.3xs5}) that  there exists  $k_5\geq k_4$   such that for all $t\geq t_5$,
\begin{eqnarray}\label{5.35}
\begin{split}
E_{k+1}-E_k+\frac{H}{k^{r}}E_k\le& C_0( k+1 ) ^{2q-p-1}+ C_1( k+1 ) ^{r-2}.
\end{split}
\end{eqnarray}
By multiplying (\ref{5.35}) with  $\pi _{k+1}=\frac{1}{\prod_{i=k_5}^{k}{\left( 1-\frac{H}{i^{r}} \right)}}$, we have
\begin{eqnarray}\label{5.36}
\begin{split}
&\pi _{k+1}E_{k+1}-\pi _kE_k\le \pi _{k+1} C_0( k+1 ) ^{2q-p-1}+ \pi_{k+1}C_1( k+1 ) ^{r-2}.
\end{split}
\end{eqnarray}
Summing  (\ref{5.36}) from $k=k_5$  to $n>k_5$ and using  \cite[Lemma A.2 (b)]{2025LC},   there exist $C_2>0$ and  $C_3>0$
such that
$
\pi _{n+1}E_{n+1}\le \pi _{n+1}C_0C_2( n+1 ) ^{ 2q-p-1+r }+\pi _{n+1}C_1C_3( n+1 ) ^{2r-2 }+\pi _{k_5}E_{k_5}.
$
This follows that
\begin{eqnarray}\label{123tian}
E_{n+1}\le C_0C_2( n+1 ) ^{2q-p-1+r}+C_1C_3( n+1 ) ^{2r-2}+\frac{\pi _{k_5}E_{k_5}}{\pi _{n+1}}.
\end{eqnarray}
Now,  we consider the following two cases:

$\mathbf{Case~ I}$: $ 1<p<2q$. In this case,   $r=\max \left\{ p-q,q \right\} =q$ and $2r-2< 2q-p-1+r=3q-p-1$. Then, we can deduced from $(\ref{123tian})$ that there exists $C_4>0$ such that $E_{n+1}\le C_4( n+1 )^{3q-p-1}$, which follows that
\begin{eqnarray*}
\left\{\begin{split}
b _n( g_n( x_n ) -g_n( \bar{x}_n ) ) &\le C_4( n+1 )^{3q-p-1},~~\frac{1}{2}  \lVert v_n\rVert ^2\le C_4( n+1 )^{3q-p-1},\\
\frac{1}{2} \xi _{n-1}\lVert \nabla g( x_n ) \rVert ^2&\le C_4( n+1 )^{3q-p-1},~~\frac{1}{2} \eta_{n-1}\lVert x_n-x_{n-1} \rVert ^2\le C_4( n+1 )^{3q-p-1}.
\end{split} \right.
\end{eqnarray*}
 According to Appendix B (i), we have $b_n=\mathcal{O}\left( n^{2q}\delta _n \right)$ and $\xi _n=\mathcal{O}\left( n^{2q}\delta _{n}^{2} \right)$, as $n\rightarrow+\infty$. Then,
$
g_n( x_n ) -g_n( \bar{x}_n )  =\mathcal{O}\left( \frac{1}{n^{p-q+1}\delta _n} \right) $ and $\lVert\nabla g( x_n) \rVert =\mathcal{O}\left( \frac{1}{n^{\frac{p-q+1}{2}}\delta _n} \right),
$ as  $n\rightarrow +\infty.$
Note that
\begin{eqnarray}\label{5.40}
\begin{split}
g( x_n ) -g( \bar{x}^* ) &=( g_n( x_n) -g_n( \bar{x}_n ) ) +( g_n( \bar{x}_n ) -g_n( \bar{x}^* ) ) -\frac{a_n}{2 \delta _n}\left( \lVert x_n  \rVert ^2-\lVert \bar{x}^* \rVert ^2 \right)\\
&\le g_n(  x_n ) -g_n( \bar{x}_n ) +\frac{a_n}{2\delta _n}\lVert \bar{x}^* \rVert ^2.
\end{split}
\end{eqnarray}
This together with $q-p-1<-p$ yields
$
g( x_n ) -g( \bar{x}^* ) =\mathcal{O}\left( \frac{1}{n^p\delta _n} \right),$   as $ n\rightarrow +\infty.
$
Further, by $g_n$ is $\frac{a_n}{2\delta_n}$-strongly convex, we   deduce from  $(\ref{L})$ that
$
g_n(x_n) - g_n(\bar{x}_n) \geq \langle \nabla g_n(\bar{x}_n), y - \bar{x}_n \rangle + \frac{a_n}{2\delta_n} \| {x}_n - \bar{x}_n\|^2
= \frac{a_n}{2\delta_n} \| {x}_n - \bar{x}_n\|^2.
$
Hence, $\lVert x_n-\bar{x}_n \rVert =\mathcal{O}\left( \frac{1}{n^{\frac{1-q}{2}} } \right)$, as $ n\rightarrow +\infty.
$  This together with $\lim _{n\rightarrow +\infty}\bar{x}_n=\bar{x}^*$ yields $\lim _{n\rightarrow +\infty}x_n=\bar{x}^*$.

 On the other hand, from  $\frac{1}{2} \eta_{n-1}\lVert x_n-x_{n-1} \rVert ^2\le C_4( n+1 )^{3q-p-1} $,
$
q_n^2\|x_n-x_{n-1}\|^2\leq\|v_n\|^2+\lambda^2\lVert x_n-\bar{x}_n \rVert^2+h^2\beta^2 q_n^2 \lVert\nabla g( x_n) \rVert^2
$
and
$\eta _n=\mathcal{O}\left( n^{q}\right)$ (see Appendix B (i)), we  get
$\rVert x_n-x_{n-1}\rVert=\mathcal{O}\left( \frac{1}{n^{\frac{p-q+1}{2}}}\right),$   as  $ n\rightarrow+\infty. $

$\mathbf{Case~ II}$: $  2q\le p<\frac{1}{2}(3q+1)$. In this case,   $r=\max \left\{ p-q,q \right\} =p-q$ and $2r-2 <2q-p-1+r=q-1$.   Then,  we   deduced from $(\ref{123tian})$ that there exists $C_4>0$ such that $E_{n+1}\le C_4( n+1 )^{q-1}$, which follows that
\begin{eqnarray*}
\left\{\begin{split}
&b _n( g_n( x_n ) -g_n( \bar{x}_n ) ) \le C_4( n+1 )^{q-1},\ \frac{1}{2}\gamma_n\lVert x_n-\bar{x}_n \rVert ^2\le C_4( n+1 )^{q-1},\\
& \frac{1}{2}\xi _{n-1}\lVert \nabla g( x_n ) \rVert ^2\le C_4( n+1 )^{q-1},\ \eta _{n-1}\lVert x_n-x_{n-1} \rVert ^2\le C_4( n+1 )^{q-1}.
\end{split}\right.
\end{eqnarray*}
Then,  as $n\rightarrow +\infty $,  it holds that
$
g_n( x_n ) -g_n( \bar{x}_n )   =\mathcal{O}\left( \frac{1}{n^{q+1}\delta _n} \right),$ $\lVert x_n-\bar{x}_n \rVert =\mathcal{O}\left( \frac{1}{n^{\frac{1-q}{2}}} \right)$, and $
\lVert\nabla g( x_n) \rVert   =\mathcal{O}\left( \frac{1}{n^{\frac{q+1}{2}}\delta _n} \right).$
Now,  from (\ref{5.40}) and $p<q+1$, we have
$
g( x_n ) -g( \bar{x}^* ) =\mathcal{O}\left( \frac{1}{n^{p}\delta _n} \right),$  as $ n\rightarrow +\infty.
$
Further, from $\lVert x_n-\bar{x}_n \rVert =\mathcal{O}\left( \frac{1}{n^{\frac{1-q}{2}}} \right)$ and $\lim _{n\rightarrow +\infty}\bar{x}_n=\bar{x}^*$, we have $\lim _{n\rightarrow +\infty}x_n=\bar{x}^*$.
On the other hand, in the same manner as Case I, we   obtain that
$  \rVert x_n-x_{n-1}\rVert=\mathcal{O}\left( \frac{1}{n^{\frac{q+1}{2}}} \right), $   as $ n\rightarrow+\infty. $

$\mathbf{Case~ III}$: $  \frac{1}{2}(3q+1)\le p<q+1$. In this case,   $r=\max \left\{ p-q,q \right\} =p-q$ and $2q-p-1+r<2r-2=2q-2p-2$. From $(\ref{123tian})$, there exists $C_4>0$ such that  $E_{n+1}\le C_4( n+1 )^{2p-2q-2}$. Then, in the same manner as Case II, as $n\rightarrow +\infty$ it holds
\begin{eqnarray*}
\left\{\begin{split}
g_n( x_n ) -g_n( \bar{x}_n ) & =\mathcal{O}\left( \frac{1}{n^{4q-2p+2}\delta _n} \right),\lVert x_n-\bar{x}_n \rVert =\mathcal{O}\left( \frac{1}{n^{q-p+1}} \right),\\
\lVert\nabla g( x_n) \rVert &=\mathcal{O}\left( \frac{1}{n^{2q-p+1}\delta _n} \right),\ \rVert x_n-x_{n-1}\rVert=\mathcal{O}\left( \frac{1}{n^{2q-p+1}} \right).
\end{split}\right.
\end{eqnarray*}
From $\lVert x_n-\bar{x}_n \rVert =\mathcal{O}\left( \frac{1}{n^{q-p+1}} \right)$ and $\lim _{n\rightarrow +\infty}\bar{x}_n=\bar{x}^*$, we have $\lim _{n\rightarrow +\infty}x_n=\bar{x}^*$.
Further, if $4q-2p+2>p$, that is $\frac{1}{2}(3q+1)\leq p<\frac{1}{3}(4q+2)
$,
 we  deduce from (\ref{5.40}) that
$
g( x_n ) -g( \bar{x}^* ) =\mathcal{O}\left( \frac{1}{n^{p}\delta _n} \right), $ as $ n\rightarrow +\infty.
$
Conversely, if $4q-2p+2\leq p$, that is $\frac{1}{3}(4q+2)\le p<q+1
$,
 we   deduce from (\ref{5.40}) that
$
g( x_n ) -g( \bar{x}^* ) =\mathcal{O}\left( \frac{1}{n^{4q-2p+2}\delta _n} \right),$ as $ n\rightarrow +\infty.
$
The proof is complete.\qed
\end{proof}
\section{Numerical experiments}
In this secction,  we give some numerical experiments to validate the obtained convergence results. All codes are run on a PC (with 1.600GHz Dual-Core Intel Core i5 and 8GB memory) using MATLAB R2020b.
\subsection{Continue case}
 In the following numerical experiments, we will compare    systems (\ref{DS}) with  (\ref{DS1}) and (\ref{DS2}).  The   systems (\ref{DS1}), (\ref{DS2}) and (\ref{DS}) are solved by ode45 adaptive method in MATLAB R2020b.

 Motivated by   \cite[Example 5.2]{A.C. Bagy} and \cite[Example 2]{amo23att},  we first give the following example  to illustrate the  convergence results.

\begin{example} Let $x:=(x_1,x_2)\in \mathbb{R}^2$.  Consider  the convex optimization problem
\begin{eqnarray*}
\min_{x\in \mathbb{R}^2}{g(x)}:=5( x_1+x_2-1 ) ^2.
\end{eqnarray*}
Clearly,  $\textup{argmin }g=\left\{ \left( x_1,1-x_1 \right) :x_1\in \mathbb{R} \right\}$ and the minimal norm solution  is $( \bar{x}_1^*,\bar{x}_2^* ) =(\frac{1}{2},\frac{1}{2} )$.

In the following  experiment, we take  the initial conditions $x (1 ) =( 1,1 )$ and $\dot{x}( 1 ) =( -1,-1 )$. The   systems (\ref{DS1}), (\ref{DS2}) and (\ref{DS}) are  solved on the time interval $[1,100]$.
We compare  system (\ref{DS})  with
  (\ref{DS1}) and (\ref{DS2})   under the following parameters setting:

$ \blacktriangleright$ Dynamical system (\ref{DS1}): ~~ $\alpha =3.5$,  $a=1$, $p=1.2$ and $q=0.9$;

$ \blacktriangleright$ Dynamical system (\ref{DS2}): ~~ $\alpha =3.5$, $\beta =4$, $a=1$  and $\delta ( t ) =t$;

$ \blacktriangleright$ Dynamical system (\ref{DS}): ~~ $\alpha =3.5$, $\beta =4$, $a=1$, $p=1.2$, $q=0.9$ and $\delta ( t ) =t$.

For any $ ( {x}_1^*,{x}_2^* ) \in \textup{argmin}g
$ and $( \bar{x}_1^*,\bar{x}_2^* ) =(\frac{1}{2},\frac{1}{2} )$, we investigate the evolution of the iteration error $\lVert ( x_1( t ) ,x_2( t ) ) -(   \bar{x}_1^*,  \bar{x}_2^* ) \rVert$ and the energy error $ g( x_1(t) ,x_2( t )) -g(  {x}_1^*, {x}_2^* ) $  of the systems (\ref{DS1}), (\ref{DS2}) and (\ref{DS}).
The results are depicted in Figure \ref{fig1}.
\begin{figure}[h]
\centering
\begin{subfigure}[t]{0.45\textwidth}
\rotatebox{90}{\scriptsize{~~~~~~~~\small{$\lVert ( x_1( t ) ,x_2( t ) ) -(  \bar{x}_1^*, \bar{x}_2^* ) \rVert$}}}
\hspace{-1mm}
\includegraphics[width=1\linewidth]{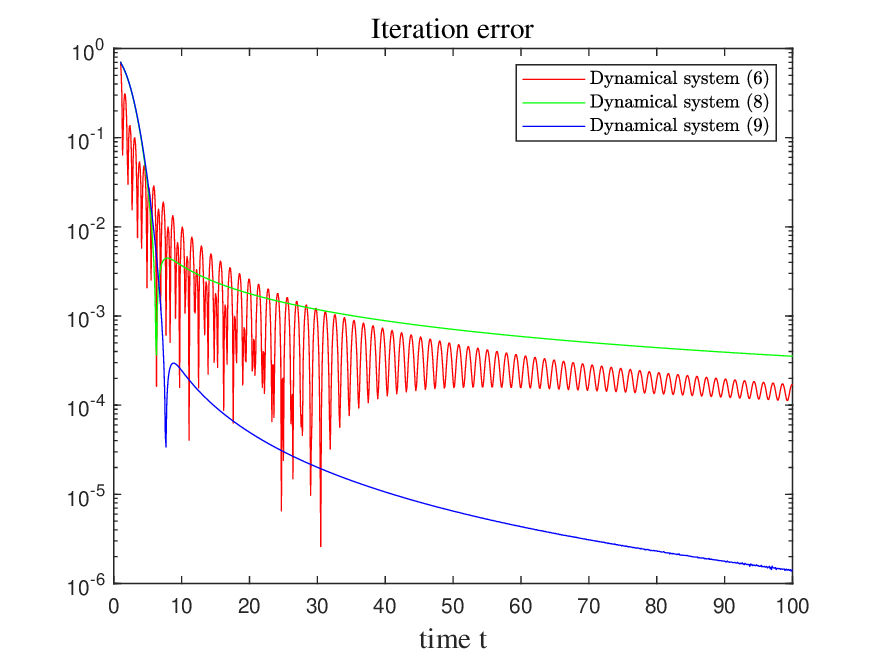}
\end{subfigure}
\hfill
\begin{subfigure}[t]{0.45\textwidth}
\rotatebox{90}{\scriptsize{~~~~~~~~~~\small{$g( x_1(t) ,x_2( t )) -g(  {x}_1^*, {x}_2^* ) $}}}
\hspace{-1mm}
\includegraphics[width=1\linewidth]{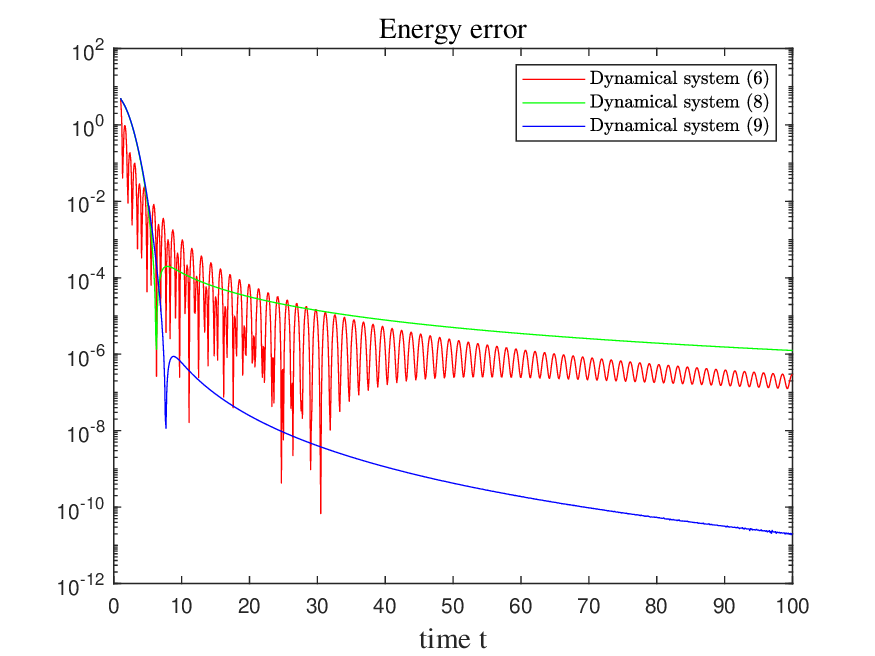}
\end{subfigure}
\centering
\caption{The behaviors of the trajectory  generated by the   systems (\ref{DS1}),  (\ref{DS2}) and (\ref{DS}).}
\label{fig1}
\end{figure}

As shown in Figure \ref{fig1}, when compared with the   systems (\ref{DS2}) and (\ref{DS}), the trajectory generated by the system (\ref{DS1}) exhibits significant oscillations. Furthermore, these curves indicate that selecting a   time  scaling parameter $\delta(t)$  can lead to improved convergence rates.  To sum up, based on these curves, the system (\ref{DS}) demonstrates better performance in terms of energy error and iteration error compared to the systems  (\ref{DS1}) and (\ref{DS2}).

 To evaluate the practical efficiency of  systems (\ref{DS1}), (\ref{DS2}) and (\ref{DS}), we record  the wall clock time (CPU time), average step size, and number of time points selected by the ode45 solver. The results are summarized in Table \ref{ta1}.
 
\begin{table}[htbp]
\centering
\caption{Stiffness-related information for Example 6.1}
\label{tab:stiffness}
\begin{tabular}{lccc}
\hline
\textbf{Dynamical System} & \textbf{Wall Clock Time (s)} & \textbf{Avg Step Size} & \textbf{Number of Time Points} \\
\hline
System (6) & 0.0083 & 5.7292e-02 & 1729 \\
System (8) & 0.0255 & 1.2544e-02 & 7893 \\
System (9) & 0.0278 & 1.3371e-02 & 7405 \\
\hline
\end{tabular}
\label{ta1}
\end{table}

 As shown in Table \ref{ta1}, although our system (\ref{DS}) achieves the best performance in energy error and iteration error (see Figure \ref{fig1}), it incurs the highest computational cost.  This can be attributed to the stiffness penalty: the introduction of Hessian-driven damping and time scaling accelerates the theoretical convergence but makes the system more stiff, forcing the adaptive solver to take smaller steps to maintain stability.

\end{example}

\begin{example}\cite[Example 1]{amo23att}  Let $x:=(x_1,x_2)\in \mathbb{R}^2$. Consider the   strictly convex optimization problem
\begin{eqnarray}\label{exam2ex}
\min_{x\in \mathbb{R}^2}{g(x):=( x_1+x_{2}^{2} ) -2\ln ( x_1+1 ) ( x_2+1 )}.
\end{eqnarray}
Obviously, $( \bar{x}_1^*,\bar{x}_2^* ) =(1,( \sqrt{5}-1 ) /2
 )$ is the    minimal norm solution of   problem (\ref{exam2ex}).

In our experiment,  we take the initial conditions $x (1 ) =( 1,1 )$ and $\dot{x}( 1 ) =( 1,1 )$. The dynamical systems (\ref{DS1}), (\ref{DS2}) and (\ref{DS}) are  solved on the time interval $[1,250]$. We compare   systems (\ref{DS})  with
 (\ref{DS1}) and (\ref{DS2})   under the following parameters setting:

$ \blacktriangleright$ Dynamical system (\ref{DS1}): ~~ $\alpha =2$,  $a=1$, $p=1.2$ and $q=0.9$;

$ \blacktriangleright$ Dynamical system (\ref{DS2}): ~~ $\alpha =2$, $\beta =4$, $a=1$  and $\delta ( t ) =t$;

$ \blacktriangleright$ Dynamical system (\ref{DS}): ~~ $\alpha =2$, $\beta =4$, $a=1$, $p=1.2$, $q=0.9$ and $\delta ( t ) =t$.

We also investigate the evolution of the iteration error $\lVert ( x_1( t ) ,x_2( t ) ) -( \bar{x}_1^*,\bar{x}_2^* ) \rVert$ and the energy error $g( x_1(t) ,x_2( t )) -g( \bar{x}_1^*,\bar{x}_2^* ) $ of the systems (\ref{DS1}), (\ref{DS2}) and (\ref{DS}).
The results  are depicted in Figure \ref{fig12}.

\begin{figure}[h]
\centering
\begin{subfigure}[t]{0.45\textwidth}
\rotatebox{90}{\scriptsize{~~~~~~~~\small{$\lVert ( x_1( t ) ,x_2( t ) ) -( \bar{x}_1^*,\bar{x}_2^* ) \rVert$}}}
\hspace{-1mm}
\includegraphics[width=1\linewidth]{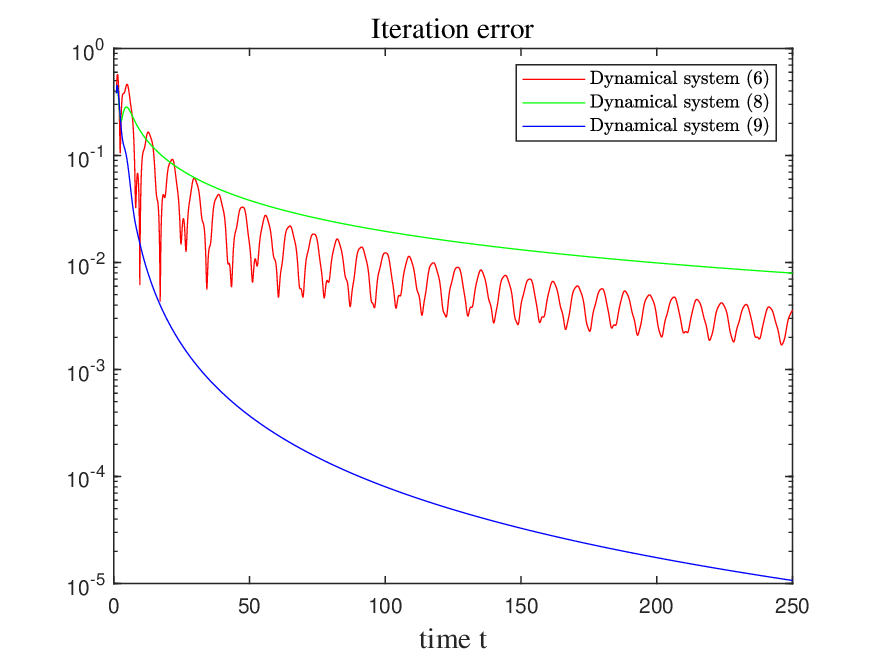}
\end{subfigure}
\hfill
\begin{subfigure}[t]{0.45\textwidth}
\rotatebox{90}{\scriptsize{~~~~~~~~~~\small{$ g( x_1(t) ,x_2( t )) -g( \bar{x}_1^*,\bar{x}_2^* ) $}}}
\hspace{-1mm}
\includegraphics[width=1\linewidth]{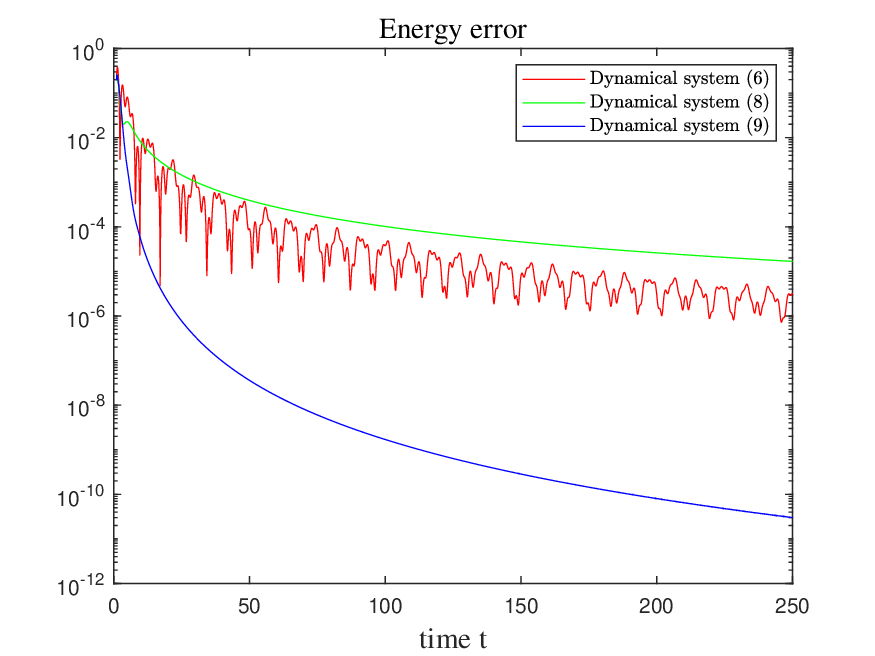}
\end{subfigure}

\caption{The behaviors of the trajectory  generated by the  systems (\ref{DS1}), (\ref{DS2}) and (\ref{DS}).}
\label{fig12}
\end{figure}
As illustrated in Figure \ref{fig12},    the   system (\ref{DS}) also performs better in energy error and iterate error than    (\ref{DS1}) and (\ref{DS2}).

 Table \ref{t2} reports stiffness-related information for the three systems in  problem (\ref{exam2ex}). Consistent with the findings in Table \ref{ta1}, although our system (\ref{DS}) demonstrates superior convergence performance in Figure \ref{fig12}, it also incurs the highest computational cost.
\begin{table}[htbp]
\centering
\caption{Stiffness-related information for Example 6.2}
\label{tab:stiffness}
\begin{tabular}{lccc}
\hline
\textbf{Dynamical System} & \textbf{Wall Clock Time (s)} & \textbf{Avg Step Size} & \textbf{Number of Time Points} \\
\hline
System (6) & 0.0128 & 1.3301e-01 & 1873 \\
System (8) & 0.0239 & 4.2318e-02 & 5885 \\
System (9) & 0.0295 & 4.2175e-02 & 5905 \\
\hline
\end{tabular}
\label{t2}
\end{table}

 In summary,  Figures \ref{fig1} and  \ref{fig12}  illustrate that the
time scaling parameter $\delta ( t )$  ensures the  trajectory generated by our   system (\ref{DS})  exhibits superior convergence rates. Additionally, Hessian damping term can induces significant attenuation of the oscillations. Meanwhile, as demonstrated in
Tables \ref{ta1} and   \ref{t2},  the theoretical acceleration attained by our system (\ref{DS}) is accompanied by increased stiffness and higher computational overhead.  This trade-off must be taken into account when assessing the practical applicability of enhanced dynamical systems.

\end{example}
\subsection{Algorithm case}
 In the following example, we will compare the algorithm  (IPGA)  with    (PIATR) proposed in \cite{2025LC}
 and  (IPATTH) proposed in \cite{2024AC}.

\begin{example}  Let $x\in \mathbb{R}^n$, $K\in \mathbb{R}^{m\times n}$ and $b\in \mathbb{R}^m$.  Consider the   $\ell _2$-regularized problem
\begin{eqnarray*}
\min_{x\in \mathbb{R}^n}{g(x)}:=\frac{1}{2}\lVert Kx-b \rVert ^2+  \lVert x \rVert ^2.
\end{eqnarray*}
In our  experiments, $K$ and $b$ are generated from the standard Gaussian distribution. We take stepsize $h=1$ and the initial condition $x_0=x_1=\mathbf{1}_n$. All algorithms are  solved on the iterations $[1,100]$.

We compare the performance of  (IPGA) with  (PIATR) proposed in \cite{2025LC} and (IPATTH) proposed in \cite{2024AC} under the following parameters setting:

$ \blacktriangleright$ (IPGA): ~~ $\alpha =15$, $\beta =4$, $a=1$, $p=1.9$, $q=0.95$ and $\delta_k =(kh)^5$;

$ \blacktriangleright$  (PIATR): ~~ $\alpha =15$,  $c=1$, $p=1.9$, $q=0.95$, $\lambda=1$ and $\delta=0$;

$ \blacktriangleright$  (IPATTH): ~~ $\alpha =15$, $\delta =4$, $c=1$  and $\beta_k =(kh)^5$.

Figures \ref{fig6} and \ref{fig7} depicts the behaviors of    $g( x_k) -g( \bar{x}^* ) $,   $ \rVert x_k -\bar x^*\rVert $, $ \rVert x_k -x_{k-1} \rVert $ and $ \rVert \nabla g(x_k) \rVert $  under different choices of the dimensionality $m$ and $n$.

As shown in Figures \ref{fig6} and \ref{fig7},  under different dimensional settings, our algorithm (IPGA) has a better convergence behaviour  than (PIATR) and (IPATTH).

\begin{figure}[h]
\centering
\begin{subfigure}[t]{0.45\textwidth}
\rotatebox{90}{\scriptsize{~~~~~~~~~~~~~~~~~~\small{$ g(x_k) -g( x^* )$}}}
\hspace{-1mm}
\includegraphics[width=1\linewidth]{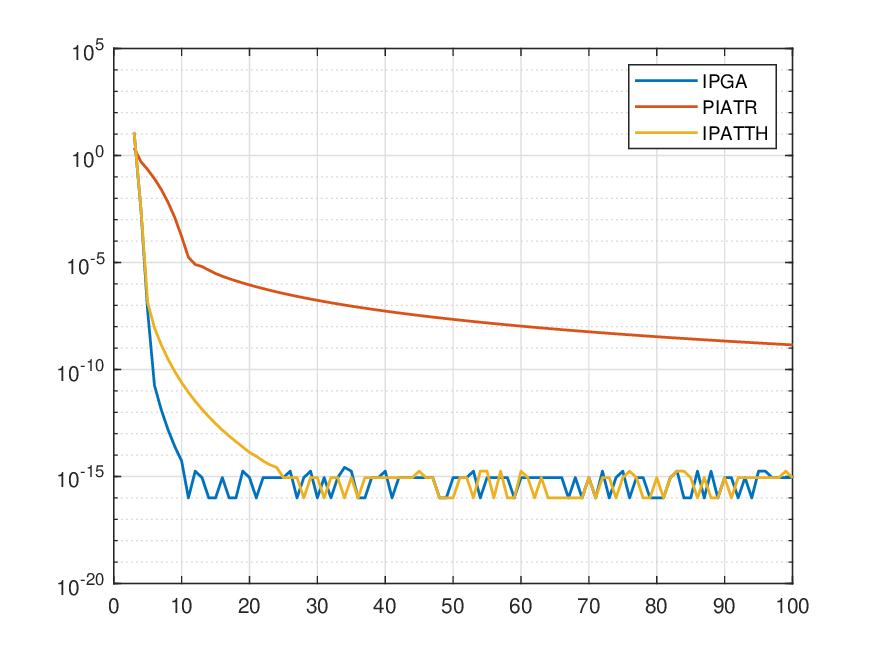}
\end{subfigure}
\hfill
\begin{subfigure}[t]{0.45\textwidth}
\rotatebox{90}{\scriptsize{~~~~~~~~~~~~~~~~~~\small{$ \rVert x_k -\bar x^*\rVert  $}}}
\hspace{-1mm}
\includegraphics[width=1\linewidth]{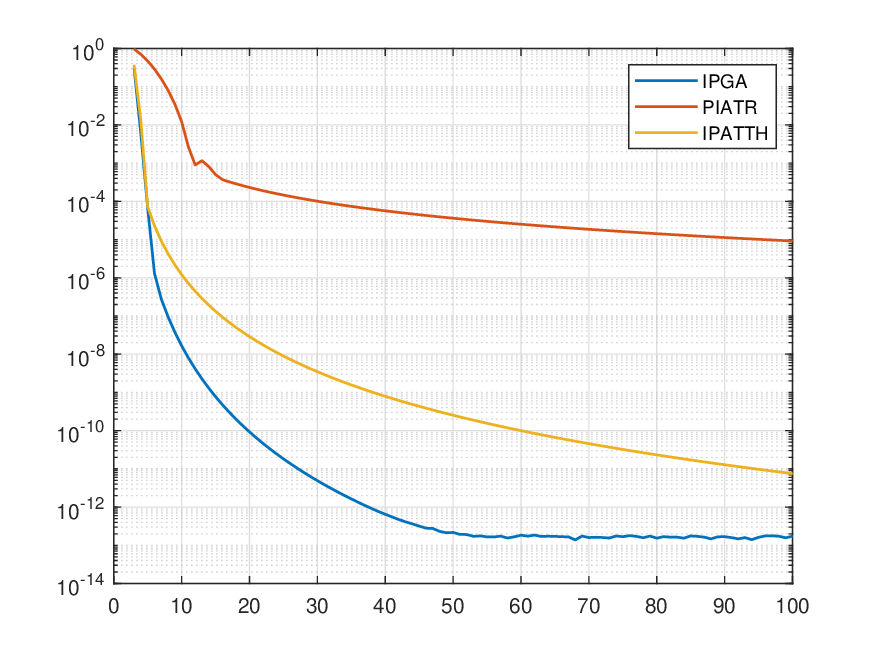}
\end{subfigure}
\centering
\begin{subfigure}[t]{0.45\textwidth}
\rotatebox{90}{\scriptsize{~~~~~~~~~~~~~~~~~~\small{$ \rVert x_k -x_{k-1} \rVert $}}}
\hspace{-1mm}
\includegraphics[width=1\linewidth]{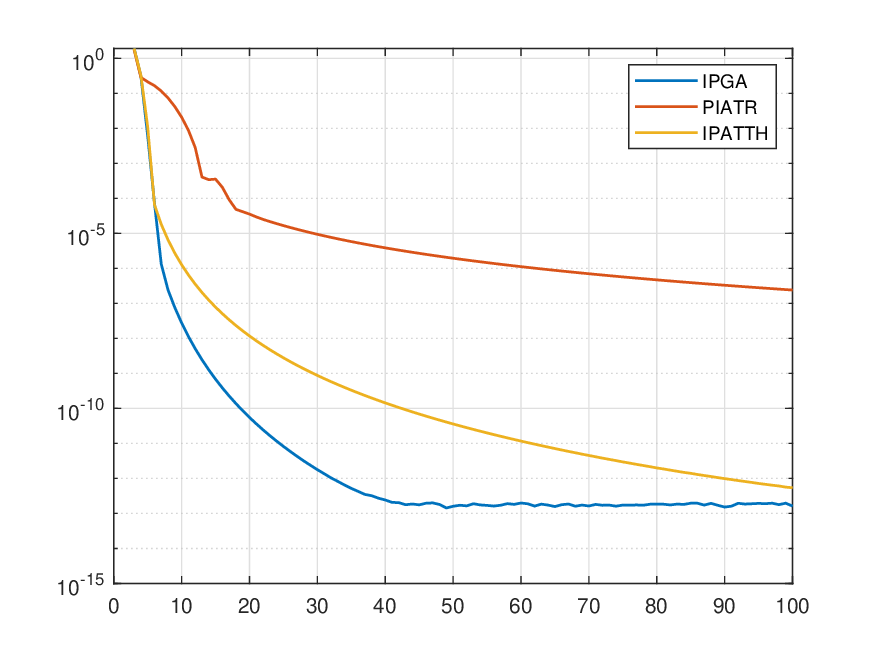}
\end{subfigure}
\hfill
\begin{subfigure}[t]{0.45\textwidth}
\rotatebox{90}{\scriptsize{~~~~~~~~~~~~~~~~~~\small{$ \rVert \nabla g(x_k) \rVert  $}}}
\hspace{-1mm}
\includegraphics[width=1\linewidth]{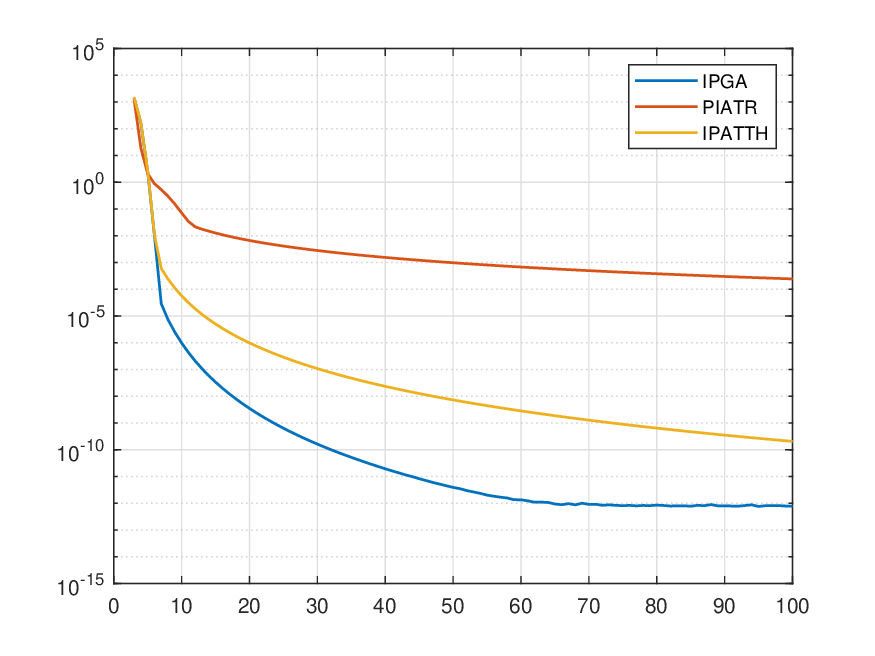}
\end{subfigure}
\caption{  $m=400$ and $n=500$.}
\label{fig6}
\end{figure}
\begin{figure}[h]
\centering
\begin{subfigure}[t]{0.45\textwidth}
\rotatebox{90}{\scriptsize{~~~~~~~~~~~~~~~~~~\small{$ g(x_k) -g( \bar{x}^* )$}}}
\hspace{-1mm}
\includegraphics[width=1\linewidth]{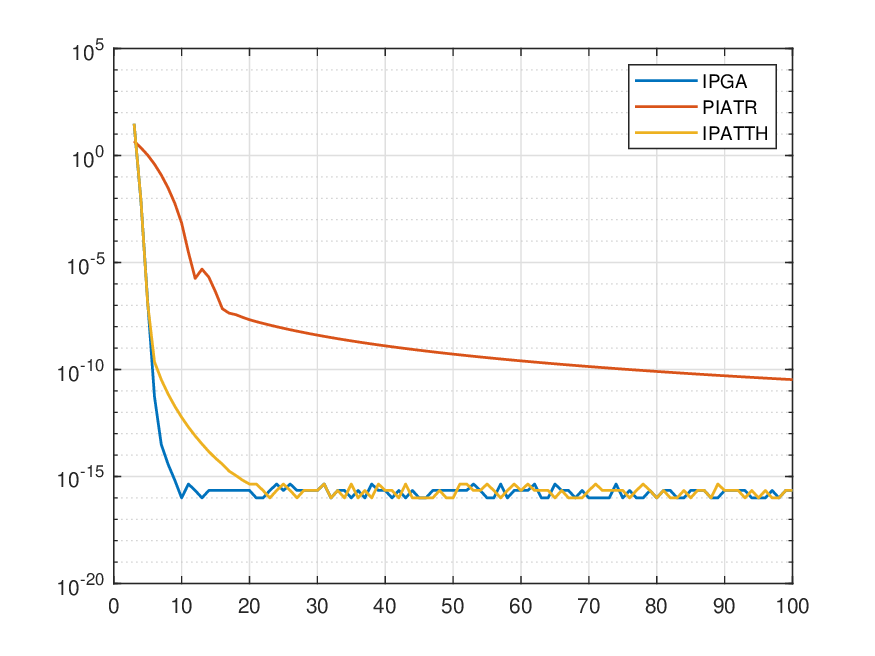}
\end{subfigure}
\hfill
\begin{subfigure}[t]{0.45\textwidth}
\rotatebox{90}{\scriptsize{~~~~~~~~~~~~~~~~~~\small{$ \rVert x_k -\bar x^*\rVert  $}}}
\hspace{-1mm}
\includegraphics[width=1\linewidth]{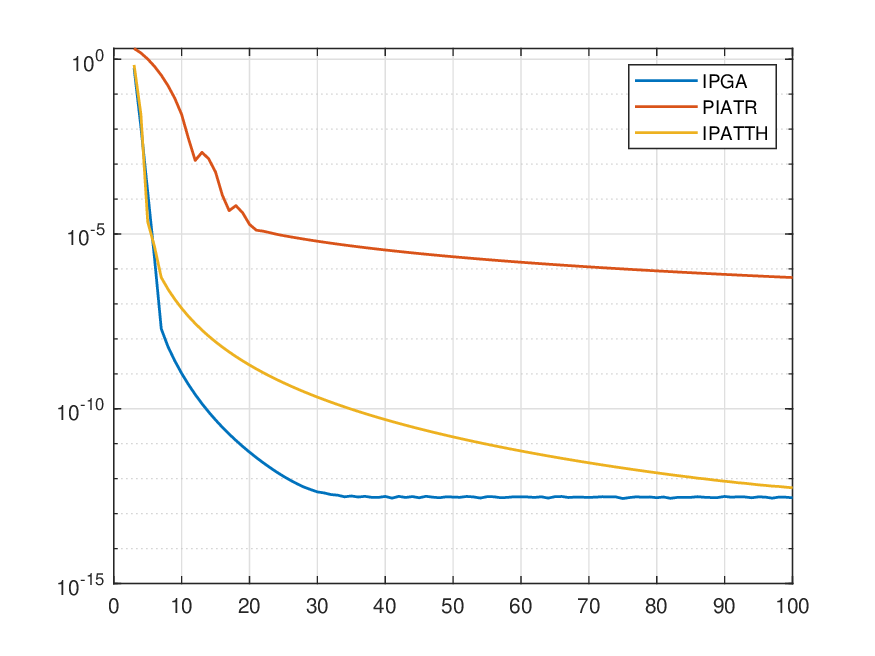}
\end{subfigure}
\centering
\begin{subfigure}[t]{0.45\textwidth}
\rotatebox{90}{\scriptsize{~~~~~~~~~~~~~~~~~~\small{$ \rVert x_k -x_{k-1} \rVert $}}}
\hspace{-1mm}
\includegraphics[width=1\linewidth]{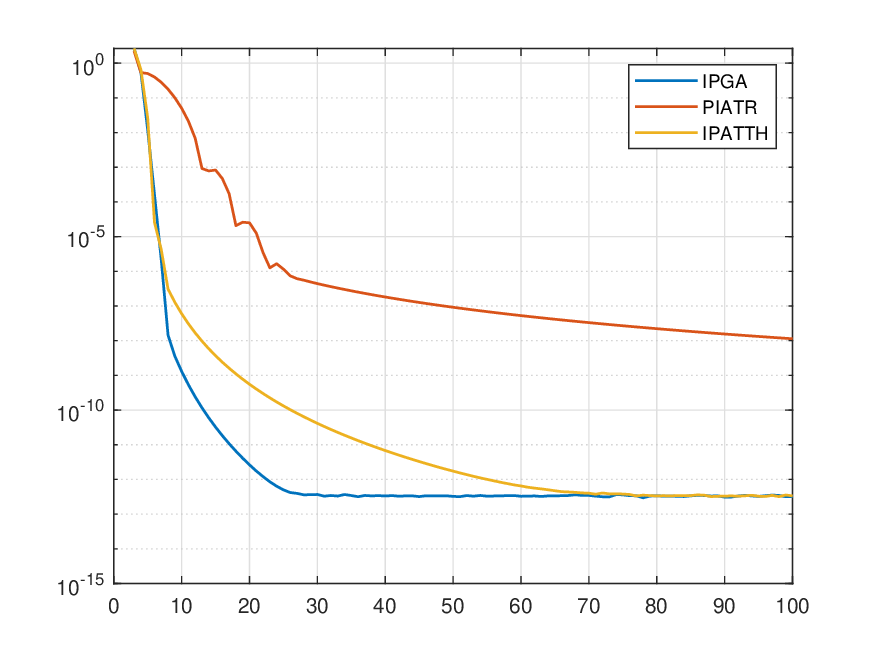}
\end{subfigure}
\hfill
\begin{subfigure}[t]{0.45\textwidth}
\rotatebox{90}{\scriptsize{~~~~~~~~~~~~~~~~~~\small{$ \rVert \nabla g(x_k) \rVert  $}}}
\hspace{-1mm}
\includegraphics[width=1\linewidth]{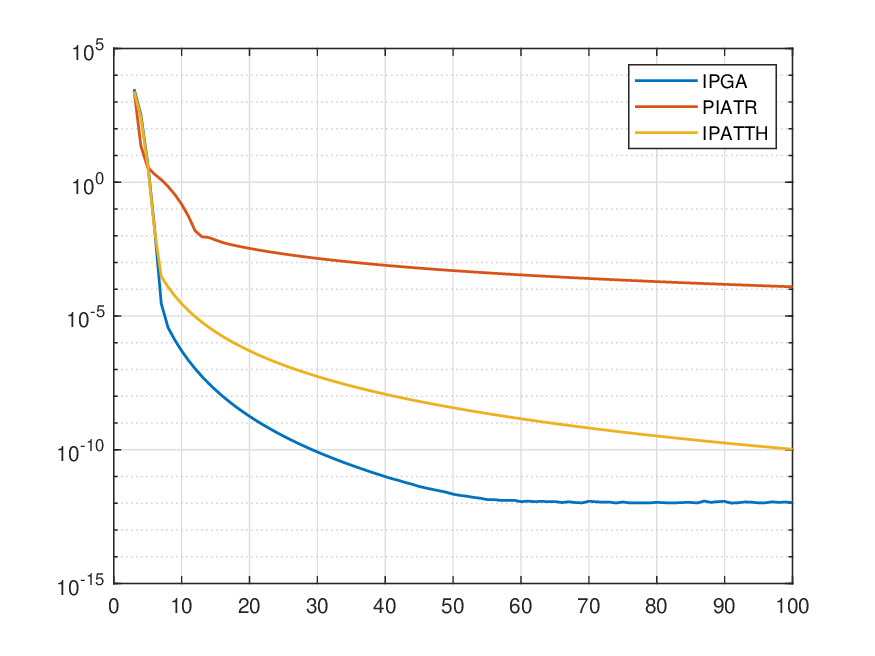}
\end{subfigure}
\caption{ $m=500$ and $ n=1000$.}
\label{fig7}
\end{figure}
\end{example}

\section{Conclusions}
In this paper, we investigate the asymptotic properties of the trajectories generated by  the second-order  dynamical system (\ref{DS}), which incorporates time scaling,  Tikhonov regularization and Hessian-driven damping, for solving the convex optimization problem (\ref{g}).
We first derive fast convergence rates for the function values and demonstrate that any trajectory generated by the system (\ref{DS}) converges weakly to a minimizer of the problem (\ref{g}).
Then, by using appropriate setting of the parameters, we achieve a fast convergence rate of the function value and strong convergence of the trajectory towards the minimum norm solution  of the   problem (\ref{g}).
We also study  the  convergence properties of the  algorithm (IPGA) obtained by the temporal  discretization of the continuous dynamic system (\ref{DS}).
In addition, our numerical experiments reveal that the algorithm (IPGA)  exhibits superior performance  as compared to the  algorithms (PIATR) and (IPATTH) introduced in \cite{2025LC}  and \cite{2024AC}, respectively.

In the future, a natural direction for our research will be to investigate the  explicit discretization  of the  system (\ref{DS}), which leads to
 a new class of fast algorithms for solving the convex optimization problem (\ref{g}). On the other hand, another challenge related to our research is whether we can extend our approach to the investigation of nonsmooth convex optimization problems.

\section*{Appendix A. Completion of the proof of Theorem 4.1}
\subsection*{\textbf{(i): We show that there exists $k_1\geq k_0 $ such that   $m_k\geq 0$, $n _k\geq0$ and $\zeta _k\geq0$  for all $k\ge k_1$.}}
\ \ \ \ \textbf{a.} Clearly,
$
m_k=\frac{1}{2}h^{2+2q}\beta (k+1)^q( 2\beta k^{q}-3\beta(k+1)^q+2hk^q \delta _k).
$
From $\lim _{k\rightarrow +\infty}\delta _k=+\infty$, there exists $k'_1\geq k_0$ such that $m_k\ge 0  $   for all $  k\ge k'_1.$

\textbf{b.}   Obviously, $n_k= \frac{1}{2}\lambda h^q( k^q-2( k+1 ) ^q+( k+2) ^q+ah^{2-p}k^{q-p})$.  We show that $n_k\ge 0$ and    $n_k=\mathcal{O}\left( k^{q-p} \right)$, as $k\rightarrow +\infty$.
Indeed, let $\phi ( x ) = \frac{1}{2}\lambda h^q\left( x^q-2( x+1 ) ^q+( x+2 ) ^q+ah^{2-p}x^{q-p} \right)$.  Then,
\begin{eqnarray*}
\begin{split}
\lim _{x\rightarrow +\infty}\frac{\phi ( x )}{x^{q-p}}
=&\lim _{x\rightarrow +\infty}\frac{ \frac{1}{2}\lambda h^q\left( x^q-2( x+1 ) ^q+( x+2 ) ^q+ah^{2-p}x^{q-p} \right)}{x^{q-p}}\\
=&\lim_{x\rightarrow +\infty}\frac{1}{2}\lambda h^q\left( \frac{1-2\left( 1+\frac{1}{x} \right) ^q+\left( 1+\frac{2}{x} \right) ^q}{x^{-p}}+a h^{2-p} \right)\\
=&\lim_{x\rightarrow +\infty}\frac{1}{2}\lambda h^q\left( \frac{ 2 q     \left( 1+\frac{2}{x} \right) ^{q-1}-2q    \left( 1+\frac{1}{x} \right) ^{q-1}}{ px^{1-p}}+ah^{2-p} \right)\\
=&\lim_{x\rightarrow +\infty}\frac{1}{2}\lambda h^q\left( \frac{ 2 q( q-1 )     \left( 1+\frac{1}{x} \right) ^{q-2}-4 q( q-1 )     \left( 1+\frac{2}{x} \right) ^{q-2}}{p(1-p)x^{2-p}}+ah^{2-p} \right).
\end{split}
\end{eqnarray*}
Obviously, $\lim _{x\rightarrow +\infty}\frac{\phi ( x )}{x^{q-p}}=\frac{1}{2}a\lambda h^{q-p+2}>0$ if $p<2$ and we  deduce  from $a>q(1-q)$ that $\lim _{x\rightarrow +\infty}\frac{\phi ( x )}{x^{q-p}}=\frac{1}{2}\lambda h^q (   q( q-1) +a  )>0$ if $p=2$.
Hence, $n_k=\mathcal{O}\left( k^{q-p} \right)$,  as $k\rightarrow +\infty$. Consequently, there exists $k''_1\geq k_0$ such that $n_k\ge0$  for all $k\ge k''_1$.

\textbf{c.} Clearly,
$
\zeta_k= \frac{1}{2}a h^{q-p+2}\Big(h^qk ^{2q-p} - ah^{q-p+2}k^{2q-2p}+\alpha h k^{ q- p} - h^q(k+1) ^{2q-p} -(\alpha h-\lambda ) (k+1) ^{q-p}\Big).
$
We show that $\zeta_k\ge 0$ and    $\zeta_k=\mathcal{O}\left( k^{q-p} \right)$, as $k\rightarrow +\infty$.
Indeed, let
$
\psi(x)= \frac{1}{2}a h^{q-p+2}\Big(h^qx ^{2q-p} - ah^{q-p+2}x^{2q-2p}+\alpha h x^{ q- p} - h^q(x+1) ^{2q-p} -(\alpha h-\lambda ) (x+1) ^{q-p}\Big).
$
Then,
\begin{eqnarray*}
\small{\begin{split}
&\lim_{x\rightarrow +\infty}\frac{\phi ( x )}{x^{q-p}}\\
=&\lim_{x\rightarrow +\infty}\frac{1}{2}a h^{q-p+2}\frac{h^qx ^{2q-p} - ah^{q-p+2}x^{2q-2p}+\alpha h x^{ q- p} - h^q(x+1) ^{2q-p} -(\alpha h-\lambda ) (x+1) ^{q-p}}{x^{q-p}}\\
=&\lim_{x\rightarrow +\infty}\frac{1}{2}a h^{q-p+2}\left( \frac{h^q}{x^{-q}} -\frac{ah^{q-p+2}}{x^{p-q}}+\alpha h-\frac{h^q \left( 1+\frac{1}{x} \right) ^{2q-p}}{x^{-q}}-(\alpha h-\lambda) \left( 1+\frac{1}{x} \right) ^{q-p}   \right)\\
=&\lim_{x\rightarrow +\infty}\frac{1}{2}a h^{q-p+2}\left( \frac{ h^q-h^q \left( 1+\frac{1}{x} \right) ^{2q-p}}{x^{-q}}+\lambda\right)\\
=&\lim_{x\rightarrow +\infty}\frac{1}{2}a h^{q-p+2}\left( \frac{( p-2q)h^q\left( 1+\frac{1}{x} \right) ^{2q-p-1}}{q x^{1-q}}+\lambda\right)
=\frac{1}{2}a\lambda h^{q-p+2}>0.
\end{split}}
\end{eqnarray*}
Hence, $\zeta_k=\mathcal{O}\left( k^{q-p} \right)$,  as $k\rightarrow +\infty$. Consequently, there exists $k'''_1\geq k_0$ such that  $\zeta_k\ge0$    for all $k\ge k'''_1$.

Finally, take $k_1=\max \{ k'_1,k''_1,k'''_1 \}$. Then,   $m_k\geq 0$, $n _k\geq0$ and $\zeta _k\geq0$  for all $k\ge k_1$.
\subsection*{\textbf{(ii): We show that $\ell_k=\mathcal{O} (k^0  )$ and $\mu_k=\mathcal{O}\left( k^{2q}\delta_k\right)$, as $k\rightarrow+\infty$.}}
\ \ \ \ \textbf{a.} Clearly, $\ell_k=h^q(k^q-(k+1)^q)+\alpha h-\lambda.$ We show that $\ell_k=\mathcal{O}\left( k^{0} \right)$, as $k\rightarrow +\infty$. Indeed, let $\phi ( x ) = h^q(x^q-(x+1)^q)+\alpha h-\lambda$.  Then,
\begin{eqnarray*}
\begin{split}
\lim _{x\rightarrow +\infty} \phi ( x )
=\lim _{x\rightarrow +\infty}\left(h^q\frac{    1-( 1+\frac{1}{x} ) ^q   }{x^{-q}}+\alpha h-\lambda\right)
=&\lim _{x\rightarrow +\infty}\left(-h^q\frac{( 1+\frac{1}{x} ) ^{q-1}   }{x^{1-q}}+\alpha h-\lambda\right)\\
=&\alpha h-\lambda>0.
\end{split}
\end{eqnarray*}
This means that $\ell_k=\mathcal{O}\left( k^{0} \right)$, as $k\rightarrow +\infty$.

\textbf{b.} It is easy to show that
$\mathcal{B}_k=\mathcal{O}\left( k^{q}\delta_k \right)$, as $k\rightarrow +\infty$.   Then, together with $q_{k+1}=h^q(k+1) ^q$ and $\ell_k=\mathcal{O}\left( k^{0} \right)$, we have $\mu_k =\mathcal{O}\left( k^{2q}\delta_k \right)$, as $k\rightarrow +\infty$.

\section*{Appendix B. Completion of the proof of Theorem 4.2}
\subsection*{\textbf{(i)
We show that there exists $k_0\in \mathbb{N}$ such that   $b_k\geq 0$, $\gamma _k\geq0$, $\xi _k\geq0$, $\eta _k\geq0$   and $\sigma _k\geq0$  for all $k\ge k_0$.}}
\ \ \ \ \textbf{a.} Clearly, $B_k =\mathcal{O}\left( k^{q}\delta _k \right)$,  as $k\rightarrow +\infty.$ Thus, we can deduce that
$b_k =\mathcal{O}\left( k^{2q}\delta _k \right)$,  as $k\rightarrow +\infty.$
Consequently, there exists $k'_0\in \mathbb{N} $ such that $b_k\ge 0  $  for all  $k\ge k'_0.$

\textbf{b.} Obviously, $\gamma _k=  h^q( k-1 ) ^q -h^qk^q+\alpha h-\lambda+\frac{1}{2}a h^{q-p+2}( k-1 )^{q-p}
$. We show that $\gamma _k\ge0$  and  $ \gamma _k =\mathcal{O}\left( k^{0} \right)$, as $k\rightarrow +\infty$. Indeed, let $\phi ( x ) = h^q( x-1 ) ^q -h^qx^q+\alpha h-\lambda+\frac{1}{2}a h^{q-p+2}( x-1 )^{q-p}   $. Then,
\begin{eqnarray*}
\begin{split}
\lim_{x\rightarrow +\infty} \phi ( x )
=&\lim_{x\rightarrow +\infty}  \left( \frac{h^q \left( 1-\frac{1}{x} \right) ^q-h^q}{x^{-q}}+\alpha h-\lambda+\frac{h^{q-p+2}}{2( x-1 ) ^{p-q}}\right)\\
=&\lim_{x\rightarrow +\infty}  \frac{-h^q\left( 1-\frac{1}{x} \right) ^{q-1}}{x^{1-q}}+\alpha h-\lambda
=   \alpha h-\lambda  .
\end{split}
\end{eqnarray*}
Since $\alpha h>2\lambda$, there exists $k''_0\in \mathbb{N}$ such that $\gamma _k\ge0$  for all $k\ge k''_0$, and  $ \gamma _k =\mathcal{O}\left( k^{0} \right)$, as $k\rightarrow +\infty$.

\textbf{c.} Since $ {B}^2_k =\mathcal{O}\left( k^{2q}\delta^2 _k \right)$, we have
 $\frac{1}{2}B_{k}^{2}-h^2\beta ^2   q_{k+1}^{2} =\mathcal{O}\left( k^{2q}\delta^2 _k \right), $   as $k\rightarrow +\infty$.
 Thus, from $\delta^2_k\geq c_0 k^{p-q}$ and  $ \frac{2\lambda\beta ^2   q_{k+1}^{2}}{q_{k}a_{k}} =\mathcal{O}\left( k^{q+p}\right)$, we can easily get $\xi _k =\mathcal{O}(k^{2q}\delta _{k}^{2} )$, as $k\rightarrow +\infty$. Consequently, there exists $k'''_0\in \mathbb{N}$ such that
$
\xi _k\ge 0
$
 for all $k\ge k'''_0$.

\textbf{d.} Obviously,
$
\eta _k= h^{2q}k^{2q}+2\alpha h^{q+1}k^q+\alpha ^2h^2-h^{2q}\left( k+1 \right) ^{2q}-\lambda h^q\left( k+1 \right) ^q-3\lambda h^qk^q-3\lambda \alpha h -3\lambda ah^{q-p+2}k^{q-p}+ah^{2q-p+2}k^{2q-p}+a\alpha h^{q-p+3}k^{q-p} .
$
We show that  $\eta _k\ge0$  and  $ \eta _k =\mathcal{O}\left( k^{q} \right)$, as $k\rightarrow +\infty$. Indeed, let
$
\phi(x)= h^{2q}x^{2q}+2\alpha h^{q+1}x^q+\alpha ^2h^2-h^{2q}( x+1 ) ^{2q}-\lambda h^q( x+1 ) ^q-3\lambda h^qx^q-3\lambda \alpha h
-3\lambda ah^{q-p+2}x^{q-p}+ah^{2q-p+2}x^{2q-p}+a\alpha h^{q-p+3}x^{q-p} .
$
Then,
\begin{eqnarray*}
{\begin{split}
\lim_{x\rightarrow +\infty}\frac{\phi \left( x \right)}{x^q}
=&\lim_{x\rightarrow +\infty} \left(h^{2q}x^q+2\alpha h^{q+1}-3\lambda h^q+\frac{\alpha ^2h^{2 }-3\lambda \alpha h}{x^q}-\frac{h^{2q}(x+1)^{2q}}{x^q}\right.\\
&~~~~~~~~~~~~\left.-\frac{\lambda h^q(x+1)^q}{x^q}-\frac{3\lambda ah^{q-p+2}}{x^{p}}+\frac{ah^{2q-p+2}}{x^{p-q}}+\frac{a\alpha h^{q-p+3}}{x^{p}}\right) \\
=&\lim_{x\rightarrow +\infty}   \frac{h^{2q}-h^{2q}\left( 1+\frac{1}{x} \right) ^{2q}}{x^{-q}}+2\alpha h^{q+1}-4\lambda h^q  \\
=&\underset{x\rightarrow +\infty}{\lim}  \frac{-2h^{2q}\left( 1+\frac{1}{x} \right) ^{2q-1}}{x^{1-q}}+2\alpha h^{q+1}-4\lambda h^q \\
=&2h^q\left( \alpha h-2\lambda \right).
\end{split}}
\end{eqnarray*}
Since $\alpha h>2\lambda$, we get $ \eta _k =\mathcal{O}\left( k^{q} \right)$, as $k\rightarrow +\infty$.
Consequently, there exists $k''''_0\in\mathbb{ N}$ such that $\eta _k\ge0$ for all  $k\ge k''''_0$.

\textbf{e.}
Obviously, together with $\delta_{k-1}\leq\delta_{k }$, it is easy to show that
\begin{eqnarray*}
\begin{split}
\sigma _k=& a h^{q-p+2}\Big( h^q( k-1 ) ^{2q-p}+\alpha h( k-1 ) ^{q-p} -ah^{q-p+2}( k-1 ) ^{2q-2p}\Big)\\
&- a\Big( \beta h^{2q-p+1} ( k-1 ) ^{2q}k^{-p}+h^{2q-p+2}( k-1 ) ^{2q}k^{-p}\delta _{k-1}+\alpha\beta h^{q-p+2} ( k-1 ) ^{q}k^{-p}\\
&~~~~~+\alpha h^{q-p+3} ( k-1 ) ^{q}k^{-p}\delta _{k-1}-\beta h^{2q-p+1}k^{2q-p} \Big) \frac{ 1}{ \delta _k}\\
=&  a h^{2q-p+2}  ( k-1 ) ^{2q}\left((k-1)^{-p}-k^{-p} \frac{ \delta_{k-1}}{ \delta _k}\right)-a^2h^{2q-2p+4}( k-1 ) ^{2q-2p}\\
&+a\alpha h^{q-p+3} ( k-1 ) ^{q}\left((k-1)^{-p}-k^{-p}\frac{\delta _{k-1}}{\delta_k}\right)\\
& - a\Big( \beta h^{2q-p+1} ( k-1 ) ^{2q}k^{-p}+\alpha\beta h^{q-p+2} ( k-1 ) ^{q}k^{-p}-\beta h^{2q-p+1}k^{2q-p} \Big) \frac{ 1}{ \delta _k}\\
\geq&  a h^{2q-p+2}  ( k-1 ) ^{2q}\left((k-1)^{-p}-k^{-p}  \right) -a^2h^{2q-2p+4}( k-1 ) ^{2q-2p}\\
&+a\alpha h^{q-p+3} ( k-1 ) ^{q}\left((k-1)^{-p}-k^{-p} \right)\\
& - a\Big( \beta h^{2q-p+1} ( k-1 ) ^{2q}k^{-p}+\alpha\beta h^{q-p+2} ( k-1 ) ^{q}k^{-p}-\beta h^{2q-p+1}k^{2q-p} \Big) \frac{ 1}{ \delta _k}.
\end{split}
\end{eqnarray*}
Since $p>1$, we can deduce that there exists $k'''''_0\in \mathbb{N}$ such that $\sigma _k\ge0$  for all $k\ge k'''''_0$.

Now, take $k_0=\max \{ k'_0,k''_0,k'''_0, k''''_0,k'''''_0 \}$. Then,  $b_k\geq 0$, $\gamma _k\geq0$, $\xi _k\geq0$, $\eta _k\geq0$   and $\sigma _k\geq0$  for all $k\ge k_0$.

\subsection*{\textbf{(ii) We show that   $\nu_k=\mathcal{O}\left( k^{q}\delta _k \right)$, $\omega _k=\mathcal{O}\left( k^{q-p} \right)$  and $   \tau _k\geq\frac{1}{2} \sigma _k $, as $k\rightarrow +\infty$.}}

\ \ \ \ \textbf{a.}  From $ \delta_{k }<\delta_{k+1 }$ and $b_k =\mathcal{O}\left( k^{2q}\delta _k \right)$, we can deduce that there exists $c'>0$ such that   ${b}_{k }-b_{k+1} \leq -c'  k^{2q-1}\delta _k $ for $k $ big enough. This together with $ B_{k} =\mathcal{O}\left( k^{q}\delta _k \right)$ yields
$\nu_k =b_k-b _{k+1}+ \lambda B_k=\mathcal{O}\left( k^{q}\delta _k \right)$,  as $k\rightarrow +\infty.$

\textbf{b.}   Obviously,
$
\omega _k=-\lambda h^qk^q- \lambda s h^{q-p}\left( h^qk^q+\alpha h+\alpha h^{q-p+2}k^{q-p}-h^q( k+1 ) ^q \right)k^{q-p}+\frac{1}{2}\lambda h^q( k+1 ) ^q+ \frac{1}{2}\lambda h^q( k-1 ) ^q+\frac{1}{4}\lambda a  h^{q-p+2}(k-1)^{q-p}.
$
We show that  $ \omega _k =\mathcal{O}\left( k^{q-p} \right)$, as $k\rightarrow +\infty$.
Indeed, let
$
\phi( x ) =-\lambda h^qx^q- \lambda s h^{q-p}( h^qx^q+\alpha h+ \alpha h^{q-p+2}x^{q-p}-h^q( x+1 ) ^q )x^{q-p}+\frac{1}{2}\lambda h^q( x+1 ) ^q+ \frac{1}{2}\lambda h^q( x-1 ) ^q+\frac{1}{4}\lambda a  h^{q-p+2}(x-1)^{q-p}.
$
Then,
\begin{eqnarray*}
\small{\begin{split}
&\underset{x\rightarrow +\infty}{\lim}\frac{\phi \left( x \right)}{x^{q-p}}\\
=&\underset{x\rightarrow +\infty}{\lim}\left( \frac{- \lambda h^q+\frac{1}{2}\lambda h^q\left( 1+\frac{1}{x} \right) ^q+\frac{1}{2}\lambda h^q\left( 1-\frac{1}{x} \right) ^q}{x^{-p}}+\frac{1}{4}\lambda a h^{q-p+2}\left(1-\frac{1}{x}\right)^{q-p} \right.\\
& \left.- \lambda s h^{q-p}\left( \frac{h^q\left( 1-\left( 1+\frac{1}{x} \right) ^q \right)}{x^{-q}}+\alpha h+\frac{  \alpha h^{q-p+2}}{x^{p-q}}
 \right)\right)\\
=&\underset{x\rightarrow +\infty}{\lim}  \left(\frac{\lambda q h^q \left( 1+\frac{1}{x} \right) ^{q-1}-\lambda q h^q \left( 1-\frac{1}{x} \right) ^{q-1}}{2p x^{1-p}}+\frac{1}{4}\lambda a h^{ q-p+2}+ \lambda s h^{q-p}\left(\frac{h^q\left( 1+\frac{1}{x} \right) ^{q-1}}{x^{1-q}}-\alpha h\right)
 \right)\\
=&\underset{x\rightarrow +\infty}{\lim}  \frac{\lambda q(q-1) h^q \left( 1+\frac{1}{x} \right) ^{q-2}+\lambda q(q-1) h^q \left( 1-\frac{1}{x} \right) ^{q-2}}{2p(p-1) x^{2-p}}+\frac{1}{4}\lambda ah^{q-p+2}- \lambda s\alpha h^{q-p+1} \\
=&\frac{1}{4}\lambda ah^{q-p+2}- \lambda s\alpha h^{q-p+1}.
\end{split}}
\end{eqnarray*}
Since $s<\frac{ha}{4\alpha}$, we get $\frac{1}{4}\lambda ah^{q-p+2}- \lambda h^{q-p}s\alpha h>0$. Hence,    $ \omega _k =\mathcal{O}\left( k^{q-p} \right)$, as $k\rightarrow +\infty$.

\textbf{c.} It is easy to show that
\begin{eqnarray*}
\begin{split}
\tau _k
=\frac{1}{2}ah^{q-p+1}\Big( \beta h^{q} k^{2q-p}+\alpha \beta h k^{q-p}-\lambda \beta k^{q-p}-\beta h^{q}( k+1 ) ^{2q-p} \Big) \frac{1}{\delta _k}+\frac{1}{2} \sigma _k.
\end{split}
\end{eqnarray*}
Since $\alpha h> \lambda$, we have $\alpha \beta h k^{q-p}-\lambda \beta k^{q-p}\geq0.$ This means that
$\beta h^{q} k^{2q-p}+\alpha \beta h k^{q-p}-\lambda \beta k^{q-p}-\beta h^{q}( k+1 ) ^{2q-p} \geq0$  for $k$ big enough.
  Then, $  \tau _k \geq \frac{1}{2}\sigma _k$  for $k$ big enough.

\subsection*{\textbf{(ii) We show that $\lVert \bar{x}_{k+1}-\bar{x}_k \rVert \le\frac{(1+c)p }{ k-cp }\lVert \bar{x}^* \rVert$   for $k$   big   enough.}}

\ \ \ \  In this proof, we consider two cases:

$\mathbf{Case \ I}$: $0<p<1$. In this case, by the convexity of the function $x\mapsto -x^p$, we can deduce from
the gradient differential inequality that
$pk^{p-1}\geq -k^p+(k+1)^p $. Thus, it follow from the condition $(\mathcal{G})$ that $\frac{\delta _{k+1}}{\delta _k}\leq\frac{k^p}{k^p-cpk^{p-1}}=\frac{k}{k-cp}$.
This  together with $(\ref{axii})$ and $\|\bar{x}_k\|\leq \lVert \bar{x}^* \rVert $ yields
\begin{eqnarray*}
{\begin{split}
 \|\bar{x}_{k+1}-\bar{x}_k\|
\leq&\left( \frac{\delta_{k+1}}{\delta _k}\frac{a_k}{a_{k+1}}-1 \right)
\lVert \bar{x}_k \rVert
=\left( \frac{\delta_{k+1}}{\delta _k}\frac{(k+1)^p}{k^p}-1 \right)
\lVert \bar{x}_k \rVert\\
\leq&\left( \frac{k}{k-cp}\frac{k^p+pk^{p-1}}{k^p}-1 \right)
\lVert \bar{x}^* \rVert
= \frac{(1+c)p}{k-cp}
\lVert \bar{x}^* \rVert.
\end{split}}
\end{eqnarray*}

$\mathbf{Case \ II}$: $p\geq1$. In this case, by the convexity of the function $x\mapsto x^p$, we can deduce from
the gradient differential inequality that
$p(k+1)^{p-1}\geq (k+1)^p-k^p.$
  Then, together with $(\ref{axii})$, $\frac{\delta_k}{\delta_{k+1}}\geq\frac{(1+c)k^p-c(k+1)^p}{k^p}$ and $\|\bar{x}_{k+1}\|\leq \lVert \bar{x}^* \rVert $,  we obtain
\begin{eqnarray*}
{\begin{split}
 \|\bar{x}_{k+1}-\bar{x}_k\|
\leq&\left(1-\frac{\delta_{k}}{\delta _{k+1}}\frac{a_{k+1}}{a_k} \right)
\lVert \bar{x}_{k+1} \rVert
=  \frac{(1+c) ((k+1)^p- k^{p })}{(k+1)^p}
\lVert \bar{x}_{k+1} \rVert\\
\le&\frac{(1+c)p  (k+1)^{p-1} }{(k+1)^p}
\lVert \bar{x}_{k+1} \rVert
\le \frac{(1+c)p}{k-cp}
\lVert \bar{x}^* \rVert.
\end{split}}
\end{eqnarray*}
Now, we conclude that $\lVert \bar{x}_{k+1}-\bar{x}_k \rVert \le\frac{(1+c)p }{ k-cp }\lVert \bar{x}^* \rVert$  for $k$   big   enough.

 \section*{Funding}
{\small This research is supported by the Natural Science Foundation of Chongqing (CSTB2024NSCQ-MSX0651) and the Team Building Project for Graduate Tutors in Chongqing (yds223010).}

\section*{Declaration}
{\small $\mathbf{Conflict ~of~ interest}$ No potential conflict of interest was reported by the authors.}

\end{document}